\documentclass[a4paper, reqno, english]{smfart}

\usepackage{amsfonts, amsthm, amsmath, amscd, amssymb}

\RequirePackage{mathrsfs} \let\mathcal\mathscr

\renewcommand\AA{\mathbb{A}}
\newcommand\QQ{\mathbb{Q}}

\newcommand\NN{\mathbb{N}}
\newcommand\FF{\mathbb{F}}
\newcommand\RR{\mathbb{R}}

\newcommand\R{\mathcal{R}}
\newcommand\B{\mathcal{B}}

\newcommand\EE{\mathcal{E}}

\newcommand\ZZ{\mathbb{Z}}

\newcommand\ZZnz{\ZZ_{\ne 0}}
\newcommand\PP{\mathbb{P}}
\newcommand\x{\mathbf{x}}
\newcommand\y{\mathbf{y}}
\renewcommand\k{\mathbf{k}}
\renewcommand\d{\,\mathrm{d}}
\newcommand\Pone{{\PP_{\QQ}^1}}
\newcommand\Ptwo{{\PP_{\QQ}^2}}

\newcommand\Pfour{{\PP_{\QQ}^4}}

\DeclareMathOperator{\hcf}{gcd}
\DeclareMathOperator{\Pic}{Pic}
\DeclareMathOperator{\rank}{rank}

\DeclareMathOperator{\sign}{sgn}
\DeclareMathOperator{\vol}{vol}
\DeclareMathOperator{\Spec}{Spec}
\DeclareMathOperator{\disc}{Disc}
\newcommand{\cp}[2]{{\hcf(#1,#2)=1}}

\newcommand{\Aone}{{\mathbf A}_1}
\newcommand{\Atwo}{{\mathbf A}_2}

\newcommand{\Afour}{{\mathbf A}_4}

\newcommand{\Dfive}{{\mathbf D}_5}
\newcommand{\Esix}{{\mathbf E}_6}

\newcommand{\tM}{{\widetilde M}}

\renewcommand{\leq}{\leqslant}
\renewcommand{\geq}{\geqslant}

\newcommand{\bnu}{\boldsymbol{\nu}}
\newcommand{\bla}{\boldsymbol{\lambda}}

\newcommand\ve{\varepsilon}
\newcommand\D{\Delta}

\renewcommand\phi{\varphi}
\newcommand\phis{\phi^*}
\newcommand\phid{\phi^\dagger}
\newcommand\phidd{\phi^\ddag}
\newcommand\la{\lambda}
\newcommand\al{\alpha}
\newcommand\be{\beta}

\newcommand{\mcal}{\mathcal}

\newcommand\sfl{\mathsf{\Lambda}}

\newtheorem{lemma}{Lemma}

\newtheorem*{thm*}{Theorem}

\theoremstyle{definition}
\newtheorem*{ack}{Acknowledgements}
\newtheorem*{notat}{Notation}

\newcommand{\odd}[2]{{(#1,#2)_\flat}}

\newcommand{\A}{\mathcal{A}}

\newcommand{\Sab}{\mathcal{S}_{b/a}}

\renewcommand{\rho}{\varrho}
\newcommand\cQ{{\overline{\mathbb{Q}}}}
\DeclareMathOperator{\Gal}{Gal}

\DeclareMathOperator*{\Osum}{\sum{}^\flat}
\DeclareMathOperator{\osum}{\sum{}^\flat}

\DeclareMathOperator\Nef{Nef}
\DeclareMathOperator\rk{rk}
\newcommand\kbar{{\overline k}}

\numberwithin{equation}{section}

\begin{document}
\title[Quartic del Pezzo surfaces with a conic fibration]{Manin's
  conjecture for quartic del Pezzo surfaces with a conic fibration}

\author{R. de la Bret\`eche}
\address{
Institut de Math\'ematiques de Jussieu\\
Universit\'e Paris 7 Denis Diderot\\
Case Postale 7012\\
2, Place Jussieu\\
F-75251 Paris cedex 05\\ France}
\email{breteche@math.jussieu.fr}

\author{T.D. Browning}
\address{School of Mathematics\\
University of Bristol\\ Bristol\\ BS8 1TW\\ United Kingdom}
\email{t.d.browning@bristol.ac.uk}

\date{\today}

\begin{abstract}
An asymptotic formula is established for the number of $\QQ$-rational
points of bounded height on a non-singular quartic del Pezzo surface
with a conic bundle structure.
\end{abstract}


\maketitle
\tableofcontents

\section{Introduction}\label{s:intro}

Let $k$ be a number field.
This investigation centres upon the distribution of $k$-rational
points on conic bundle
surfaces $X/\PP_k^1$. These are defined to be projective non-singular surfaces $X$
defined over $k$, which are equipped with a dominant $k$-morphism
$\pi:X\rightarrow \PP_k^1$, all of whose fibres are conics.
A summary of our knowledge concerning the arithmetic of conic
bundle surfaces can be found in the  article by Colliot-Th\'el\`ene \cite{ct-4}.
If $-K_X$ denotes the anticanonical divisor of $X$, then the degree of
$X$ is defined to be the self-intersection number $(-K_X,-K_X)=8-r$,
where $r$ is the number of geometric fibres above $\pi$ that are degenerate.
When $0\leq r\leq 3$ it is known that the Hasse 
principle holds for these surfaces, and 
furthermore, that
such~$X$ are $k$-rational as soon as they possess $k$-rational points.
When $r=4$, so that $X$ has degree $4$, it has been shown by
Iskovskikh \cite[Proposition~1]{isk}
that two basic cases arise. Either $-K_X$ is not 
ample, in which case $X$ is $k$-birational to
a generalised Ch\^atelet
surface, or else $-K_X$ is ample, in which case $X$ is a non-singular
quartic del Pezzo
surface.

Our interest lies with the quantitative arithmetic of
degree $4$ conic bundle surfaces. When $X(k)\neq \emptyset$ and 
$H: X(k)\rightarrow \RR_{\geq 0}$ is a height %T%
associated to $-K_X$, we seek to
determine the  asymptotic behaviour of the counting function
$$
N_{U,H}(B):=\#\{x \in U(k): H(x) \leq B\},
$$
as $B\rightarrow \infty$, for a suitable Zariski open subset $U\subseteq X$.
The conjecture that drives our work is due to
Manin \cite{f-m-t}. Let $\Pic X$ be the Picard 
group of  $X$. Still under the assumption that
$X(k)$ is non-empty, this conjecture
predicts the existence of a positive constant $c_{X,H}$ such that
\begin{equation}\label{manin}
N_{U,H}(B) = c_{X,H} B (\log B)^{\rank (\Pic X) -1}\big(1+o(1)\big),
\end{equation}
as $B \rightarrow \infty$. Peyre  \cite{p} has given a conjectural
interpretation of $c_{X,H}$ in terms of the geometry of
$X$.  Until very recently we were not in possession of a
single conic bundle surface of degree $4$ for 
which this refined conjecture could be
established.

Henceforth we will be interested in the case $k=\QQ$.
In joint work with Peyre~\cite{bbp}, the authors have made a
study of generalised Ch\^atelet surfaces, ultimately
establishing \eqref{manin} for a family of such surfaces that fail
to satisfy weak approximation.
The aim of the present investigation is
to produce a satisfactory treatment of a non-singular del
Pezzo surface of degree $4$ with a conic bundle structure which admits
a section over $\QQ$. Such surfaces are defined as the intersection
of two quadrics in $\Pfour$. When $X(\QQ)\neq \emptyset$, we may
assume that  $X \subset \Pfour$ is cut out by the system
\begin{equation}
   \label{eq:bundle}
\begin{cases}
   \Phi_1(x_0,\ldots,x_4):=x_0x_1-x_2x_3=0,\\
\Phi_2(x_0,\ldots,x_4)=0,
\end{cases}
\end{equation}
for quadratic forms 
$\Phi_1,\Phi_2\in\ZZ[x_0,\ldots,x_4]$ such that 
the Jacobian
matrix $(\nabla \Phi_1, \nabla \Phi_2)$ has full rank throughout
$X$. A proof of this familiar fact can be found in 
\cite[Lemma~2.2]{ferran}, for example.

Let $\|\cdot\|:\RR^5\rightarrow \RR_{\geq 0}$ be a norm.
Given a point $x=[\x] \in \Pfour(\QQ),$ with
$\x=(x_0,\ldots,x_4) \in \ZZ^5$ such that $\hcf(x_0,\ldots,x_4)=1$, we let
$H(x):=\|\x\|$. Then $H$ is
the anticanonical height metrized by the choice of norm.
Any line contained in $X$
that is defined over $\QQ$ will automatically contribute
$cB^2+O(B\log B)$ to $N_{U,H}(B)$, for an appropriate constant
$c>0$. Thus one takes $U\subset X$ to be the open
subset formed by deleting the $16$ lines from $X$.

The best evidence that we have for \eqref{manin} in the setting of
non-singular  surfaces of the shape \eqref{eq:bundle}
is due to Salberger. In
work communicated  at the conference ``Higher 
dimensional varieties and rational
   points'' at Budapest in 2001, he establishes the upper bound
\begin{equation}
   \label{eq:salb}
N_{U,H}(B)=O_{X}(B^{1+\ve}),
\end{equation}
for any $\ve>0$.  Here, as throughout our work, we allow the implied
constant to depend on the choice of $\ve$. 
The Manin conjecture has received a great deal of
attention in the context of singular del Pezzo surfaces of degree $3$
and $4$. An account of recent progress can be found in the
second author's work \cite{ferran}.
There is general agreement among researchers that the
level of difficulty in establishing the expected asymptotic formula
for del Pezzo surfaces increases as the degree 
decreases or as the singularities
become milder. Among the non-singular del Pezzo surfaces, those of
degree at least $6$ are all toric and so are handled by the work of
Batyrev and Tschinkel \cite{b-t}. In \cite{dp5} the first author gave
the earliest satisfactory treatment of a non-singular del Pezzo surface
of degree $5$. As highlighted by Swinnerton-Dyer \cite[Question
15]{swd-prog}, it has become something of milestone to establish the
Manin conjecture for a single non-singular del Pezzo surface of
degree $3$ or $4$.

In this paper we will be concerned with a quartic del Pezzo surface
$X\subset \Pfour$ of the shape \eqref{eq:bundle}, with
\begin{equation}\label{eq:phi2}
\Phi_2(\x):=x_0^2+x_1^2+x_2^2-x_3^2-2x_4^2.
\end{equation}
In particular $X$ contains obvious lines defined
over $\QQ$, as we shall see in \S \ref{s:local},
from which it follows that $X$ is $\QQ$-rational. The latter fact
is recorded by Colliot-Th\'el\`ene,
Sansuc and Swinnerton-Dyer \cite[Proposition~2]{ct1}, for example. If
$L$ is a line in $X$ defined over $\QQ$ then it
is a simple consequence of the fact that any plane through $L$ must cut out a
pair of lines $L,L_i$ on each quadric $\Phi_i=0$ defining $X$, and the
intersection $L_1\cap L_2$ is a $\QQ$-point contained in $X$.

Let $\|\cdot\|$ be the norm on $\RR^5$
given by
\begin{equation}
   \label{eq:norm}
\|\x\|:=\max\Big\{|x_0|,|x_1|,
|x_2|,|x_3|,\sqrt{\frac{2}{3}}|x_4|\Big\}.
\end{equation}
All that is required of this norm is that
$\|\x\|=\max\{|x_0|,|x_1|,|x_2|,|x_3|\}$ for every
$\x=(x_0,\ldots,x_4)\in\RR^5$
such that $[\x] \in X$, and furthermore
$\|\x^\sigma\|=\|\x\|$, where
$\x^\sigma$ is the vector obtained by permuting the variables
$x_0,\ldots,x_3$ and leaving $x_4$ fixed.
It would  be possible to work instead with the norm
$|\x|:=\max\{|x_0|,\ldots, |x_4|\}$, but not without
introducing extra technical difficulties that we 
wish to suppress in the present investigation.
We are now ready to reveal our main result.

\begin{thm*}
We have
$$
N_{U,H}(B)=c_{X,H} B(\log B)^4+O\Big( \frac{B(\log B)^4}{\log\log B} \Big),
$$
where $c_{X,H}>0$ is the constant predicted by Peyre.
\end{thm*}

During the final preparation of this paper, the authors have learnt of
independent work by Fok-Shuen Leung \cite{fok} on the conic bundle
surface given by \eqref{eq:bundle} and \eqref{eq:phi2}. This sharpens
Salberger's estimate in \eqref{eq:salb}, 
ultimately providing upper and lower bounds
for $N_{U,H}(B)$ that are of the expected order of magnitude. Our result
supersedes this, and
confirms the estimate predicted
by Manin and Peyre in \eqref{manin} for the non-singular quartic del Pezzo
surface under consideration. The fact that the
Picard group has rank $5$, as needed to verify the power of $\log B$,
will be established in due course.

The proof of our theorem is long and complicated. We therefore dedicate the
remainder of this introduction to surveying some of its key
ingredients and indicating some obvious lines for further enquiry.
Given any non-singular surface defined by the system
\eqref{eq:bundle},  it is possible to define a pair of conic bundle
morphisms $f_i:X\rightarrow \Pone$, for $i=1,2$.
Specifically, for any $x\in X$,  one takes
$$
f_1(x)=
\begin{cases}
  [x_0,x_2] , & \mbox{if $(x_0,x_2)\neq(0,0)$,} \cr
[x_3,x_1] , &  \mbox{if $(x_1,x_3)\neq(0,0)$,} 
\end{cases}
$$
and
$$
f_2(x)=
  \begin{cases}
  [x_0,x_3] , & \mbox{if $(x_0,x_3)\neq(0,0)$,} \cr
   [x_2,x_1] , & \mbox{if $(x_1,x_2)\neq(0,0)$.}
\end{cases}
$$
%%NOTE: morphisms need to be defined on open subsets which match up on
%%the intersection of them!
For a given point $x\in X(\QQ)$ of height $H(x)\leq B$,
it follows from the general theory of height functions that
there exists
an index $i$ such that
$x\in f_i^{-1}(t)$ for some $t\in \Pone(\QQ)$ of height
$O(B^{1/2})$.
The idea is now to count rational points of bounded height on the
fibres $f_1^{-1}(t)$ and $f_2^{-1}(t)$, uniformly for
points $t\in \Pone(\QQ)$ of height
$O(B^{1/2})$.  This is the strategy adopted by Salberger in his proof
of \eqref{eq:salb}.

In the present situation, with the quadratic form \eqref{eq:phi2},
the fibres that we need to examine have the shape
\begin{equation}
   \label{eq:conic}
C_{a,b}: \quad (a^2-b^2)x^2+(a^2+b^2)y^2=2z^2,
\end{equation}
for coprime $a,b\in\ZZ$.
It is clear that $C_{a,b}\subset\Ptwo$ is a non-singular plane
conic when the discriminant $\D(a,b)=-2(a^4-b^4)$ is non-zero.
The reduction of the counting problem to one involving the family of
conics \eqref{eq:conic} is carried out in \S \ref{s:prelim}, where we
have avoided using the height machinery by doing things in a
completely explicit manner.

As is well-known there is a group homomorphism $\Pic X\rightarrow \ZZ$,
which to a divisor class $D \in \Pic X$ associates the intersection
number of $D$ with a fibre. The kernel of this map is generated by the
``vertical'' divisors, which up to linear equivalence are the
irreducible components of the fibres.  Since the
non-singular fibres are all linearly equivalent, it follows that $\Pic
X$ has rank $2+n$, where $n$ is the number of split singular fibres
above closed points of $\Pone$. In our case there are three closed
points, corresponding to the irreducible factors $a-b, a+b$ and
$a^2+b^2$ of $\D(a,b)$. Since each singular fibre is split it
follows that $\Pic X\cong \ZZ^5$, as previously claimed.

For fixed 
$[a,b]\in \Pone(\QQ)$ for which the conic 
$C_{a,b}$ is non-singular and has a $\QQ$-point, the
number of $\QQ$-points of height $B$ is asymptotically $c_{a,b}B$,
with a constant $c_{a,b}$ depending on $a,b$. 
As the size of $a,b$ increases with $B$, however,
one finds that $c_{a,b}$ decreases in magnitude so that there are
fewer $\QQ$-points of height $B$ on the conic overall. 
The preliminary reduction to
$[a,b]\in~\Pone(\QQ)$ of height $O(B^{1/2})$ is absolutely pivotal
here, since it is only through this device that we can cover $U(\QQ)$
with a satisfactory number of divisors.
Were we charged instead with establishing an upper bound
like~\eqref{eq:salb}, our analysis would now be 
relatively straightforward, thanks
to the control over the growth rate of rational
points on conics afforded by the second
author's joint work with Heath-Brown \cite[Theorem~6]{bhb}. A key
aspect of this estimate is that it is
uniform in the height of the conic, becoming sharper as the
discriminant grows larger.
In \S \ref{s:ham} we will take advantage of these arguments to
eliminate certain awkward ranges for $a,b$ in \eqref{eq:conic}.

Obtaining an asymptotic formula is a far more exacting task.
Using the large sieve inequality Serre \cite{serre} has shown
that most plane conics defined over $\QQ$ don't contain rational
points.  This phenomenon  might pose problems for us, given that we want
a uniform asymptotic formula for a Zariski dense set of rational
points on the fibres. Our choice of surface has been tailored to
guarantee that this doesn't happen, since the corresponding fibres
\eqref{eq:conic} always contain the
rational point
\begin{equation}
   \label{eq:xi}
\xi=[1,-1,a].
\end{equation}
In the classical manner we can use this point to parametrise
all of the rational points on the conic, which ultimately leads us to
evaluate asymptotically the number of points belonging to a
$2$-dimensional sublattice $\sfl\subset \ZZ^2$ which are constrained to
lie in an appropriate region $\R\subset \RR^2$.
Both $\sfl$ and $\R$ depend on the parameters $a$ and $b$,  so this
estimate needs to be achieved with a sufficient degree of uniformity.
Assuming that $\R$ has piecewise continuous boundary, we would ideally
like to apply the familiar estimate
\begin{equation}
   \label{eq:classic}
\#(\sfl\cap \R)=\frac{\vol(\R)}{\det \sfl} + O(\partial \R+1),
\end{equation}
where $\partial \R$ denotes the perimeter of $\R$ and the implied
constant is absolute. 
Unhappily this estimate is too crude for our purposes. We will use
Poisson summation to make the error term explicit in
\eqref{eq:classic}. 
The reduction of the problem to a lattice point 
counting problem
is carried out in \S \ref{sec:par} and its execution is the subject of 
\S\S \ref{sec:1}--\ref{s:2}.

The one outstanding task, which is the focus of
\S \ref{s:3}, is to evaluate asymptotically the main term
arising from the lattice point counting problem. 
It turns out that this involves a sum of the shape
$$
\sum_{\substack{(a,b) \in \ZZ^2\cap \R\\ \gcd(a,b)=1}} \frac{g(|a^4-b^4|)}{\max\{a,b\}^2},
$$
for a certain multiplicative arithmetic function $g$ that is very similar
to the ordinary divisor function $\tau(n):=\sum_{d\mid n}1$.
The problem of determining the average order of $\tau$
as it ranges over the values of polynomials has enjoyed considerable
attention in the literature. In the setting of binary forms current
technology has limited us to handling forms of degree at most $4$.
When $g=\tau$ Daniel \cite{daniel} has
dealt with the case of irreducible binary quartic forms.
We need to extend this
argument to deal with
a more general class of arithmetic functions and to binary quartic forms
that are no longer irreducible.  We have found it convenient to
corral the necessary estimates into a separate investigation
\cite{L1L2Q}, which is of a more technical nature.
The facts that we will need are recalled in \S \ref{s:tech}.

Our main goal in this paper is to outline  a general strategy for
proving the Manin conjecture for non-singular del Pezzo surfaces of
degree $4$ equipped with a conic bundle structure admitting a section
over $\QQ$. In order to minimise the
length and technical difficulty of the work, which is already
considerable, we have chosen to illustrate the approach by selecting
a concrete surface to work with. It is likely that the
argument we present can be adapted to handle 
other conic bundle surfaces. For example, 
consider the family
of surfaces \eqref{eq:bundle}, with
$$
\Phi_2(\x):=c_0x_0^2+c_1x_1^2+c_2x_2^2+c_3x_3^2+c_4x_4^2.
$$
It is easy to check that $X$ will be non-singular if and only if
$c_0\cdots c_4 \neq 0$ and $c_0c_1\neq c_2c_3$. 
As we've already mentioned, our success with \eqref{eq:phi2} is closely
linked to the existence of an obvious rational point
\eqref{eq:xi} on all of the fibres \eqref{eq:conic}. This is 
equivalent to the map $X(\QQ)\rightarrow 
\Pone(\QQ)$ being surjective. A necessary and sufficient 
condition for ensuring this is that the morphism $X\rightarrow \Pone$
should admit a section over $\QQ$. This fact has a long history and
can be traced back to work of Lewis and Schinzel \cite{LS}. 
Even when the conic bundle surface 
admits a section, however, there remains
considerable work to be done in handling a general diagonal form
$\Phi_2$ as above. The corresponding fibres will take the shape
$$
f_1(a,b)x^2+f_2(a,b)y^2+c_4z^2=0,
$$
for binary forms $f_1(a,b)=c_0a^2+c_3b^2$ and
$f_2(a,b)=c_2a^2+c_1b^2$.
Thus, at the very least, one needs analogues of
our investigation \cite{L1L2Q} for the case in which $f_1f_2$
factors over $\QQ$ as the product
of two quadratic forms or splits completely.

\begin{notat}
Throughout our work $\NN$ will denote the set of positive
integers.  If $a,b\in \NN$ then we write 
$\gcd(a,b)$ for the greatest 
common divisor of $a,b$ and $[a,b]=ab/\gcd(a,b)$ for the
least common multiple. We set 
\begin{equation}
   \label{eq:odd-notat}
\odd{a}{b}:= 2^{-\nu_2(\hcf(a,b))}\hcf(a,b)
\end{equation}
for the odd part of the greatest common divisor, and use
the symbol $\osum$ to indicate a
summation in which all the variables of summation are
restricted to odd integers. We will let $\lfloor\alpha \rfloor$ denote
the integer part of any number $\alpha\in \RR$.
Furthermore,  our work will involve the arithmetic functions
\begin{equation}
   \label{eq:funcp}
\phi(n):=n\prod_{p\mid n}\Big(1-\frac{1}{p}\Big),\quad
\phis(n):=\frac{\phi(n)}{n},\quad
\phid(n):=\prod_{p\mid n}\Big(1+\frac{1}{p}\Big),
\end{equation}
and 
\begin{equation}
   \label{eq:funcpp}
\phidd(n):=\prod_{p^\sigma\| n}\Big(1+\frac{\sigma^5}{p}\Big).
\end{equation}
A simple convolution argument  shows that 
\begin{equation}
  \label{eq:1}
\sum_{n\leq x}\phis(n)^A   \leq \sum_{n\leq x}\phid(n)^A
 \leq \sum_{n\leq x}\phidd(n)^A
 \ll_A x,  
\end{equation}
for any $A>0$. 
In terms of the parameter $B$, we set
\begin{equation}
   \label{eq:zet}
   Z_1:=B^{1/\log\log B}, \quad Z_2:=\log\log B.
\end{equation}
We will reserve $c>0$ for a generic absolute positive
constant, whose value is always effectively computable, and $\ve$ will
denote a small positive parameter whose value may vary, so that
$x^\ve \log x=O(x^{\ve})$, for example. 
Finally we will follow the convention that all
implied constants in this paper 
are allowed to depend on the parameters $c$ and
$\ve$, with any other dependence explicitly indicated with an
appropriate subscript.
\end{notat}

\begin{ack}
This investigation was undertaken while the second author was visiting
the first author at the {\it Universit\'e Paris 6 Pierre et Marie Curie} and the
{\it Universit\'e Paris 7 Denis Diderot}. The hospitality and financial support
of these institutions is gratefully acknowledged. 
It is a pleasure to thank Olivier
Wittenberg for useful conversations relating to the geometry
of conic bundle surfaces and the anonymous referees for numerous
pertinent remarks that have greatly improved the exposition of the
paper. 
While working on this paper the second author was supported
by EPSRC grant number \texttt{EP/E053262/1} and the work has also 
received the financial support of the ANR project
{\em Points entiers points rationnels}.
\end{ack}

\section{Technical results}\label{s:tech}

Our work requires a number of auxiliary results, ranging from basic
estimates using the geometry of numbers to more sophisticated results
concerning the divisor problem for binary quartic forms.

\subsection{Counting rational points on curves}

As made clear in the introduction to this paper, our
proof of the theorem uses the conic bundle structure
in order to focus the effort on a family of 
curves of low degree and rather low height.
The following  result is due to Heath-Brown \cite[Lemma
2]{square}, %T%
and deals with the situation for lines in
$\Ptwo$, the rational points on which basically correspond to
integer lattices of rank $2$.

\begin{lemma}\label{lem:lattice} %T%
Let $\sfl \subseteq \ZZ^{2}$ be a lattice of rank $2$ and determinant
  $\det\sfl$, and let $E\subset \RR^2$ be an ellipse, centered at the
  origin, together with its interior. Then we have
\[
\#\{\x\in \sfl\cap E: \cp{x_1}{x_2}\}
\ll 1+ \frac{\vol(E)}{\det\sfl}.
\]
\end{lemma}

Our next uniform upper bound
is extracted from joint work of the second author with Heath-Brown
\cite[Corollary 2]{bhb}, and handles the case of non-singular plane conics.

\begin{lemma}\label{lem:bhb}
Let $C\subset \Ptwo$ be a non-singular conic. Assume that the underlying
quadratic form has  matrix of determinant $\Delta$, and
that the $2\times 2$ minors have
greatest common divisor $\Delta_0$. Then we have
$$
\#\left\{\x\in C\cap \ZZ^3:
\begin{array}{l}
\hcf(x_1,x_2,x_3)=1,\\
|x_i|\leq B_i, (1\leq i\leq 3)
\end{array}
\right\}
\ll \tau(|\Delta|)\Big(1+ \frac{B_1B_2B_3\Delta_0^{3/2}}{|\Delta|}\Big)^{1/3}.
$$
\end{lemma}

Taking $B_1=B_2=B_3=B$ in Lemma \ref{lem:bhb}, we retrieve the
well-known fact that a non-singular 
plane conic $C\subset \Ptwo$ contains $O_C(B)$ 
rational points of height $B$.

\subsection{Generalisation of Nair's lemma}

In the setting of polynomials in only one variable there is a
well-known result due to Nair \cite{mnair} which provides upper bounds
for the average order of suitable non-negative arithmetic functions as they ranges over the values
of the polynomial. Taking advantage of the authors' refinement \cite{nair} of this work
we have the following result.

\begin{lemma}\label{lem:sky}
Let $\ve>0$, let $a\in\NN$ and let $\kappa\in \{0,1\}$.  Let
$\delta\in [0,1)$ and define $\tau'$ multiplicatively via
$$
\tau'(p^\nu) :=
\begin{cases}
1+\nu, & \mbox{if $p>2$,}\\
(1+\nu)^2, & \mbox{if $p=2$.}
\end{cases}
$$
Then for $x\gg a^\ve$ there exists an absolute constant $c>0$ such
that 
$$
\sum_{n \leq x}  \frac{\tau'(n)^{\kappa}\tau'(n^2+a)}{n^\delta} \ll 
\phidd(a)^c x^{1-\delta} (\log x)^{1+\kappa},
$$
where $\phidd$ is given by \eqref{eq:funcpp}.
\end{lemma}

\begin{proof}
Let us denote by 
$S_{\delta,\kappa}(x)$ the sum that is to be estimated. We begin by
dealing with the case $\delta=0$. 
Define the multiplicative arithmetic function
$$
\tau''(p^\nu) :=
\begin{cases}
2, & \mbox{if $\nu=1$ and $p>2$,}\\
4, & \mbox{if $\nu=1$ and $p=2$,}\\
(1+\nu)^4, & \mbox{if $\nu\geq 2$,}
\end{cases}
$$
for any prime power $p^\nu$. We have
$\tau'(n_1)\tau'(n_2)\leq \tau''(n_1n_2)$ for any $n_1,n_2\in \NN$, whence
$$
S_{0,\kappa}(x)\leq \sum_{n \leq x}  \tau''(n^{\kappa}(n^2+a)).
$$
Now for any $a \in \NN$ it is clear that the polynomial $f(t)=t^\kappa
(t^2+a)$ has degree $2+\kappa$, that it has discriminant
$\D_f=-4a^{1+2 \kappa}$ and that it has no fixed prime divisor if
$\kappa=0$ or $a$ is even. If $\kappa=1$ and $a$ is odd then $2$ is a
fixed prime divisor of $f$. But then we may break the sum into two
sums according to whether $n$ is even or odd and make a corresponding
change of variables, absorbing the additional factor $\tau''(2)=4$ into
an implied constant. 
An application of \cite[Theorem 2]{nair} with $\delta=\ve$ 
now reveals that
$$
S_{0,\kappa}(x)\ll x \prod_{p\leq x}\Big(1-\frac{\rho_f(p)}{p}\Big)
  \sum_{m \leq x}  \frac{\tau''(m)\rho_f(m)}{m},
$$
for $x \gg a^\ve$, where $\rho_f(m)$ denotes the number of roots
modulo  $m$ of the congruence $f(n)\equiv 0 \bmod{m}$.  It follows
from work of 
Nagell \cite[Paragraph 27]{nagell} that we have
$\rho_f(p^\nu) \leq 3\min\{p^{\nu-1},p^{2\nu_p(\D_f)}\}$ for
any prime power, whence
$\rho_f(p^\nu) \leq 3$ if $p\nmid \D_f$. We deduce that
\begin{align*}
  \sum_{m \leq x}  \frac{\tau''(m)\rho_f(m)}{m}
&\leq \prod_{p\leq
  x}\Big( 
1+\frac{\tau''(p)\rho_f(p)}{p}+\sum_{\nu\geq 2} 
\frac{\tau''(p^\nu)\rho_f(p^\nu)}{p^\nu}\Big)\\
&\leq \exp\Big(\rho_f(2)+\sum_{p\leq
  x}\frac{2\rho_f(p)}{p}+\sum_{p\leq 
x}\sum_{\nu\geq 2} 
\frac{(1+\nu)^4\rho_f(p^\nu)}{p^\nu}\Big).
\end{align*}
Now
\begin{align*}
\sum_{p\leq 
x}\sum_{\nu\geq 2} 
\frac{(1+\nu)^4\rho_f(p^\nu)}{p^\nu}
&\leq
\sum_{p^\sigma\|  \Delta_f}
\Big(
\sum_{2\leq \nu\leq 2\sigma} 
\frac{3(1+\nu)^4}{p}
+
\sum_{\nu>2\sigma} 
\frac{3(1+\nu)^4p^{2\sigma}}{p^\nu}
\Big)\\
&\ll \sum_{p^\sigma\|  \Delta_f}\frac{\sigma^5}{p}
\ll \sum_{p^\alpha\|  a}\frac{\alpha^5}{p}.
\end{align*}
Hence  we obtain 
\begin{align*}
 \sum_{m \leq x}  \frac{\tau''(m)\rho_f(m)}{m}
&\ll \phidd(a)^c\exp\Big(\sum_{p\leq
  x}\frac{2\rho_f(p)}{p}\Big),
\end{align*}
for a suitable absolute constant $c>0$. On noting that
$
\rho_f(p)=1+\kappa+(\frac{-a}{p}),
$
for $p>2$, this therefore concludes the 
proof that 
$$
S_{0,\kappa}(x) \ll 
\phidd(a)^c x (\log x)^{1+\kappa},
$$
which is satisfactory for the lemma. 

When $\delta>0$ we invoke partial summation. This yields
$$
S_{\delta,\kappa}(x)= \frac{S_{0,\kappa}(x)}{x^\delta}+\delta
\int_1^x \frac{S_{0,\kappa}(t)}{t^{\delta+1}}\d t
\ll \phidd(a)^c x^{1-\delta} (\log x)^{1+\kappa},
$$
as required. 
\end{proof}

Note that the estimate in Lemma \ref{lem:sky} is also valid for the
divisor function $\tau$, since $\tau(n)\leq \tau'(n)$ for every $n \in
\NN$. 
We will also need a version of Nair's lemma for binary forms
and the generalised divisor function
$\tau_k(n):=\sum_{n=d_1\cdots d_k}1$.

\begin{lemma}\label{lem:nair}
Let $\ve>0$ and $A,B\geq 2$ with 
$\min\{A,B\}\geq \max\{A,B\}^\ve$. Then we have 
$$
\sum_{|a| \leq A, ~|b| \leq B}  \tau_k (|a^4-b^4|) \ll_{k}
AB(\log AB)^{3k-3}.
$$
Furthermore, when $A=B$ and $p> q\geq 0$ such that $p+q=2$, we have 
$$
\sum_{|a|, |b| \leq A} \frac{\tau_k (|a^4-b^4|)}{\max\{|a|,|b|\}^p 
\min\{|a|,|b|\}^q} \ll_{k} (\log A)^{3k-2}.
$$
\end{lemma}

\begin{proof}
For the moment let $F\in\ZZ[x_1,x_2]$ be a binary form 
of degree~$d$  with $\disc (F)\neq 0$ and $F(1,0)F(0,1)\neq 0$.
Let  $\|F\|$ denote the maximum modulus of its
coefficients and put
$$
\rho^*_F(m):=\frac{1}{ \varphi(m)}\#\Big\{ (n_1,n_2)\in (0,m]^2\,:\,
\begin{array}{ll} \gcd(n_1,n_2,m)=1 \\
F(n_1,n_2)\equiv 0
\bmod{m}\end{array}\Big\},
$$
for any  $m\in\NN$.
Then it follows from \cite[Corollary~1]{nair} 
that 
$$
\sum_{|a| \leq A, ~ |b| \leq B}  \tau_k (|F(a,b)|) \ll_{d,k}
\|F\|^\ve\Big(
    ABE +\max\{A,B\}^{1+\ve}\Big),
$$
where
$$
E:=
\prod_{d<p\leq
\min\{A,B\}}\Big(1+\frac{\rho_F^*(p)(k-1)}{p}\Big).
$$
Clearly $\max\{A,B\}^{1+\ve}\ll AB$ under the hypotheses of the
lemma. 

When $F(x_1,x_2)=x_1^4-x_2^4$ one 
sees that 
$$
\rho_F^*(p)=
\#\{ x \bmod{p}:
x^4\equiv 1 \bmod{p}\}=
3+\chi(p),
$$
for $p>2$, where
$\chi$ is the real non-principal character modulo $4$.
It therefore follows from
this result and Merten's theorem that 
\begin{align*}
\sum_{a\leq A,b\leq B} \tau_k(|a^4-b^4|)
&\ll
AB
\prod_{4<p\leq
\min\{A,B\}}\Big(1+\frac{(3+\chi(p))(k-1)}{p}\Big)\\
&\ll AB(\log AB)^{3k-3},
\end{align*}
which establishes the first part of the lemma. 
The second part is an easy consequence of the first part, on 
breaking $\max\{|a|,|b|\}$ and $\min\{|a|,|b|\}$ into dyadic intervals.
\end{proof}

We will also require versions of these results in which the
$\min\{|a|,|b|\}$ appearing in the denominator of the second estimate
in Lemma \ref{lem:nair} is
replaced by $|a-b|$. This is achieved in the following result

\begin{lemma}\label{lem:nair'}
Let $\ve>0$ and $x\geq 2, y> 0$.
Let $p> q\geq 0$ such that $p+q=2$. Then we have the estimates
\begin{align}
\label{eq:anna-swim}
\sum_{0<\max\{a,x/y\}<b\leq x}\frac{\nu_2(a^2-b^2)\tau(|a^4-b^4|)}{b^{p}|a-b|^{q}}
&\ll (\log x)^{3}\log (y+2),\\
\label{eq:anna-swim'}
\sum_{0< y\min\{a,b-a\} <b\leq
  x}\frac{\tau(|a^4-b^4|)}{\max\{a,b\}^{p}|a-b|^{q}} 
&\ll_k \frac{(\log x)^{4}}{y^{1-q}}+1.
\end{align}
\end{lemma}

\begin{proof}
In particular it is clear from the conditions of summation that 
$b>0$ and $|a-b|>0$ for all $a,b$ under consideration in the two sums.
Beginning with the
estimation of the sum in \eqref{eq:anna-swim}, which we
denote by $S(x,y)$, it will be convenient to set
$u=b+a$ and $v=b-a$. In particular  $b^4-a^4=uv(u^2+v^2)/2$. We deduce
that
\begin{equation}\label{eq:arch}
\begin{split}%T%
S(x,y)
&\ll
\sum_{
\substack{
x/y<u\leq 2x\\ 0<v<u}}\frac{\nu_2(uv)\tau(uv(u^2+v^2))}{u^{p}v^{q}}\\
&\ll
\sum_{
\substack{
x/y<u\leq 2x\\ 0<v<u}}\frac{\tau'(uv(u^2+v^2))}{u^{p}v^{q}}\\
&\ll
\sum_{
x/y<u\leq 2x}
\frac{\tau'(u)}{u^{p}}\sum_{0<v< u}
\frac{\tau'(v)\tau'(u^2+v^2)}{v^{q}}.
\end{split}
\end{equation}
with $\tau'$ defined as in the statement of Lemma~\ref{lem:sky}.  %T%
We may assume that $y<2x$, say, since it is only $u>1/2$ that
contribute to the sum. 
Taking $\delta=q$ and $\kappa=1$ in 
Lemma~\ref{lem:sky} we deduce that
\begin{align*}
S(x,y)
&\ll(\log x)^2
\sum_{
x/y<u\leq 2x}
\frac{\tau'(u)}{u}  \phidd(u)^c =(\log x)^2 S_1(x,y), %T%
\end{align*}
say. We claim that 
$S_1(x,y)\ll (\log x)(\log (y+2))$,
which will clearly suffice to complete 
the proof of \eqref{eq:anna-swim}.
To establish the claim we deduce from an application of the
Selberg--Delange method 
(see  \cite[\S II.5]{ten}, for example)
that  there is a constant $c_0$ such that 
$$%T%
  \Sigma(x):=\sum_{u\leq x} \frac{\tau'(u)}{u}  \phidd(u)^c= c_0(\log x)^2+O(\log x).
$$
But then  $S_1(x,y)=\Sigma(2x)-\Sigma(x/y)\ll (\log x)(\log (y+2))$, as required.

Turning to the sum in \eqref{eq:anna-swim'}, which we
denote by $S'(x,y)$,  it is clear that
$$
S'(x,y)\leq
\sum_{0<ya <b\leq x}
\hspace{-0.1cm}
\frac{\tau(|a^4-b^4|)}{b^{p}a^{q}}
+\hspace{-0.1cm}
\sum_{\substack{0< y(b-a) <b\leq x\\
b\leq 2a}}
\frac{\tau(|a^4-b^4|)}{a^{p}(b-a)^{q}}=S_1'(x,y)+S_2'(x,y),
$$
say. 
For the first sum we consider the contribution
from $a,b$ such that $A\leq b<2A$ and $A'\leq a<2A'$ for 
$A\leq x$ and $A'\leq 2A/y$. Such $a,b$ contribute
$
\ll A^{1-p}{A'}^{1-q}(\log x)^3
$
to $S_1'(x,y)$ by Lemma \ref{lem:nair}.
Summing over the dyadic intervals for $A,A'$ we conclude that 
$S_1'(x,y)\ll (\log x)^4/y^{1-q}+1$. 
Finally, to estimate $S_2'(x,y)$
we set $u=b+a$ and $v=b-a$, so that 
$$
S_2'(x,y)\ll \sum_{0< yv <u\leq 2x}
\frac{\tau(uv(u^2+v^2))}{u^{p}v^{q}}
\ll
\sum_{
0<u\leq 2x}
\frac{\tau(u)}{u^{p}}\sum_{0<v< u/y}
\frac{\tau(v)\tau(u^2+v^2)}{v^{q}}.
$$
Tracing through the argument used to handle the virtually 
identical sum arising in our estimate for $S(x,y)$ in \eqref{eq:arch},
we deduce that 
$$
S_2'(x,y)
\ll
\frac{(\log x)^2}{y^{1-q}}
\sum_{0<u\leq 2x}
\frac{\phidd(u)^c\tau(u)}{u}+1\ll 
\frac{(\log x)^4}{y^{1-q}}+1,
$$
as required for  \eqref{eq:anna-swim'}.
\end{proof}

Taking $(p,q)=(2,0)$ in the second part of Lemma \ref{lem:nair}
the sum in question can be interpreted as
$$
\sum_{\substack{
x=[a, b]\in \Pone(\QQ)\\ H_1(x)\leq A}} 
\frac{\tau_k (|a^4-b^4|)}{H_1(x)^2},
$$
where $H_n$ is the exponential height on $\PP_\QQ^n(\QQ)$. In \S \ref{s:2}
sums of the shape 
$$
\sum_{\substack{
x=[a, b,m]\in C(\QQ)\\ H_2(x)\leq A}} 
\frac{\tau_k (|a^4-b^4|)}{H_2(x)}
$$
will arise, for suitable non-singular conics $C\subset \Ptwo$ such that
$C(\QQ) \neq \emptyset$. In fact we will only need to consider the
case in which $k=3$.

\begin{lemma}\label{lem:nair''}
Let $\ve>0$ and $A\geq 2$. Then for any $i,j\in \{0,1\}$
we have 
$$
\sum_{\substack{1\leq a, b,m\leq A\\
a^2+(-1)^i b^2=2^j m^2}}  \frac{\tau_3 (|a^4-b^4|)}{\max\{a,b\}} \ll  (\log
A)^{26}.
$$
\end{lemma}

\begin{proof}
Let $\Sigma_{i,j}(A)$ denote the sum that is to be estimated. We
will achieve the desired bounds by parametrising the solutions to the
underlying equations. 
Suppose first that $j=1$. 
Solutions of $a^2-b^2=2m^2$ are parametrised via
$$%T%
a=u^2+2v^2, \quad b=\pm(u^2-2v^2), \quad
m=2uv,
$$
for coprime $u,v\in \NN$. Likewise, solutions of
$a^2+b^2=2m^2$ are parametrised via
$$%T%
a=u^2+2uv-v^2, \quad b=\pm(-u^2+2uv+v^2), \quad
m=u^2+v^2,
$$
for coprime $u,v\in \NN$. 
It follows that
$$
\Sigma_{i,1} (A)\ll 
\sum_{u,v\leq \sqrt{A}}\frac{\tau_3 (|F_i(u,v)|)}{\max\{u,v\}^2},
$$
where $F_0(u,v)=u^2v^2(u^4+4v^4)$ and 
$F_1(u,v)=uv(u^2-v^2)(u^2+v^2)^2$.
We will show that  
$$
\widetilde{\Sigma}_i (T):=
\sum_{u,v\leq T} \tau_3 (|F_i(u,v)|) \ll
T^2(\log T)^{25},
$$
which once combined with a dyadic summation will suffice to 
establish the upper bound for 
$\Sigma_{i,1} (A)$ in the statement of the lemma. 
To estimate $\widetilde{\Sigma}_i (T)$ we shall
apply 
\cite[Corollary~1]{nair}, exactly 
as in our proof of Lemma~\ref{lem:nair}. To
deal with the fact that our binary forms have repeated factors we
invoke the inequalities
$
\tau_3(\ell m^2)\leq \tau_3(\ell^2 m^2)\leq
\tau_6(\ell m).
$
But then it easily follows that 
$$
\widetilde{\Sigma}_i (T)\ll 
\sum_{u,v\leq T} \tau_6 (|G_i(u,v)|) \ll T^2
(\log T)^{25},
$$
as required, where $G_0(u,v)=uv(u^4+4v^4)$ and 
$G_1(u,v)=uv(u^2-v^2)(u^2+v^2)$.

For the case $j=0$ the argument is exactly the same, but this time
built on the more familiar parameterisation of the Pythagorean
triples. The outcome is the bound for 
$\Sigma_{i,0}(A)$ recorded in the statement of the lemma. 
\end{proof}

\subsection{The divisor problem for binary forms}

Throughout this section we let $i$ denote a generic
element from the set $\{1,2\}$.
Let $L_i,Q \in \ZZ[x_1,x_2]$ be binary forms, with $\deg L_i=1$
and $\deg Q=2$, such that $L_1,L_2$ are non-proportional and~$Q$ is
irreducible over $\QQ$.  Let $\B\subseteq [-1,1]^2$ be a convex region
whose boundary is defined by a piecewise continuously differentiable
function. Assume that $L_i(\x)>0$ and~$Q(\x)>0$ for every $\x \in \B$.
The workhorse in this paper is an asymptotic formula for sums akin to
$$
\sum_{\substack{\x\in \ZZ^2\cap X\B\\\gcd(x_1,x_2)=1}}
\frac{\tau(L_1(\x) L_2(\x) Q(\x))}{\max\{|x_1|,|x_2|\}^2},
$$
where $\tau$ is the divisor function and $X\B:=\{X\x: \x\in\B\}$.

In fact the arguments that appear in our work call for a rather more general
type of sum.  Suppose that $g=h*\tau$ is the Dirichlet convolution of
the divisor function with a multiplicative  arithmetic function
$h$ that satisfies
\begin{equation}
   \label{eq:1.22}
   \sum_{d\in \NN} \frac{|h(d)|}{d^{1/4}}\ll 1.
\end{equation}
Let $V\subseteq [0,1]^4$ be a region cut out by a finite number of
hyperplanes with absolutely bounded coefficients.
For any $Y\geq 2$ we define
$$
   g(L_1(\x),L_2(\x),Q(\x);Y;V)
:=
\sum_{\substack{
d\mid L_1(\x)L_2(\x)Q(\x)\\
d_i=\gcd(d,L_i(\x)), ~
d_3=\gcd(d,Q(\x))\\
(\frac{\log d_1}{\log Y},\frac{\log d_2}{\log Y},\frac{\log
   d_3}{2\log Y} , \frac{\log\max \{|x_1|,|x_2|\}}{\log Y})\in V}}
(1*h)(d).
$$
Then we will encounter sums of the shape
$$
   S_g(X,Y;V) :=\sum_{\substack{\x\in \ZZ^2\cap X\B\\\gcd(x_1,x_2)=1}}
\frac{g(L_1(\x),L_2(\x),Q(\x);Y;V)}{\max\{|x_1|,|x_2|\}^2}.
$$
If one takes $V=[0,1]^4$ and  $Y$ a multiple
of $X$ then $g(L_1,L_2,Q;Y;V)=g(L_1L_2Q)$. Moreover, if one
takes 
$$
h(n)=\begin{cases}
1, & \mbox{if $n=1$,}\\
0, & \mbox{otherwise},
\end{cases}
$$
then one arrives at exactly the sum
involving $\tau$  that was mentioned at the start of this subsection.

For any prime $p$ and $\nu_1,\nu_2,\nu_3\geq 0$, let
\begin{equation}
   \label{eq:rho_dag}
   \rho_p^\dagger({\nu_1},{\nu_2},{\nu_3}):=
\#\left\{\x \in (\ZZ/p^{\nu_1+\nu_2+\nu_3+1}\ZZ)^2:
\begin{array}{l}
p\nmid \x, \\
p^{\nu_i}\| L_i(\x),\\
p^{\nu_3}\| Q(\x)
\end{array}
\right\}
\end{equation}
and
\begin{equation}
   \label{eq:rho_dag-}
\overline{\rho}_p^\dagger(\nu_1,\nu_2,\nu_3)
:=p^{-{2(\nu_1+\nu_2+\nu_3+1)}}\rho_p^\dagger(\nu_1,\nu_2,\nu_3).
\end{equation}
Here, we follow the convention that $p^\nu \| n$ if and only if $\nu_p(n)=\nu$.
The following asymptotic formula is established in our
companion paper \cite[Corollaire~3]{L1L2Q}.

\begin{lemma}\label{lem:uniform}
Let $\ve>0$. Assume that $2\leq X\leq Y\leq X^{1/\ve}$. Then we have
$$
S_g(X,Y;V)=4C^* \vol(\B)\vol(V_0)(\log Y)^4
+O_{L_i,Q}\big((\log X^2/Y)(\log Y)^3+(\log X)^{3+\ve}\big),
$$
where
\begin{equation}\label{eq:C*}
C^*:=\prod_p \Big(1-\frac{1}{p}\Big)^{3}\sum_{\bnu \in \ZZ_{\geq 0}^3}
g(p^{\nu_1+\nu_2+\nu_3})
\overline{\rho}_p^\dagger({\nu_1},{\nu_2},{\nu_3})
\end{equation}
and
\begin{equation}
   \label{eq:V0}
V_0:=
V\cap \Big\{\mathbf{v}\in [0,1]^4: 
\max\{v_1,v_2,v_3\}\leq v_4\leq \frac{1}{2}\Big\}.
\end{equation}
\end{lemma}

We will ultimately apply Lemma \ref{lem:uniform} with $L_1L_2Q(a,b)$
equal to the discriminant $\D(a,b)$ of the conic \eqref{eq:conic}. 
The exponent of $\log Y$ in the lemma reflects the fact that we are
dealing with binary forms with three irreducible factors. The
attentive reader will observe a correlation with our description of the rank of
$\Pic X$ in \S \ref{s:intro}.

We will be interested in applications of Lemma \ref{lem:uniform} when
$g:\NN \rightarrow
\RR$ is the multiplicative arithmetic function defined  via
\begin{equation}
   \label{eq:g}
   g(p^\nu):=
\begin{cases}
\max\{1,\nu-1\}, & \mbox{if $p=2$,}\\
1+\nu(1-1/p)/(1+1/p), & \mbox{if $p>2$.}
\end{cases}
\end{equation}
It easily follows that $g=h*\tau$, with
\begin{equation}\label{eq:h}
\begin{split}
h(p^\nu)=(g*\mu*\mu)(p^\nu)
&=
\begin{cases}
0, & \mbox{if $p>2$ and $\nu\geq 2$,}\\
-2/(p+1), & \mbox{if $p>2$ and $\nu=1$,}\\
0, & \mbox{if $p=2$ and $\nu\geq 4$,}\\
1, & \mbox{if $p=2$ and $\nu=3$,}\\
0, & \mbox{if $p=2$ and $\nu=2$,}\\
-1, & \mbox{if $p=2$ and $\nu=1$.}
\end{cases}
\end{split}
\end{equation}
In particular $|h(n)|\ll n^{-1+\ve}$ for any
$\ve>0$, whence \eqref{eq:1.22} holds.

\section{Preliminary manipulations}\label{s:prelim}

Recall the definition of the quadratic forms
$\Phi_1,\Phi_2$ from \eqref{eq:bundle} and \eqref{eq:phi2}.
We begin this section by relating $N_{U,H}(B)$ to the quantity
\begin{equation}
   \label{eq:def-N1}
N_1(B):=\#\left\{\x\in \NN^5:\begin{array}{l}
\hcf(x_0,\ldots,x_3)=1,~
\max\{x_0,\ldots,x_3\}\leq B,\\
\Phi_1(\x)=\Phi_2(\x)=0,~\{x_0,x_1\}\neq \{x_2,x_3\}
\end{array}
\right\},
\end{equation}
in which the main difference is that the count is restricted
to positive integer solutions.  This is achieved in the following result.

\begin{lemma}\label{lem:red-1}
   We have $N_{U,H}(B) =8N_1(B) +O(B).$
\end{lemma}

\begin{proof}
It follows from the calculation of the lines in \S \ref{s:local} that
the condition~$x\in U$ is equivalent to $\{|x_0|,|x_1|\}\neq \{|x_2|,|x_3|\}$,
for any $x=[\x]\in X$.
We therefore deduce that
$$
N_{U,H}(B)=\frac{1}{2}\#\left\{\x\in \ZZ^5: \begin{array}{l}
\hcf(x_0,\ldots,x_4)=1, ~\|\x\|\leq B,\\
\Phi_1(\x)=\Phi_2(\x)=0, ~\{|x_0|,|x_1|\}\neq \{|x_2|,|x_3|\}
\end{array}
\right\}.
$$
It follows  from the equation $\Phi_2(\x)=0$ that the
condition $\hcf(x_0,\ldots,x_4)=1$ is equivalent to
$\hcf(x_0,x_1,x_2, x_3)=1$ for any vector $\x$ in which we are
interested.
Furthermore, for $\x$ such that $\Phi_1(\x)=\Phi_2(\x)=0$, it is clear
from \eqref{eq:norm} that
\begin{align*}
\|\x\|&=\max\left\{|x_0|,|x_1|,
|x_2|,|x_3|,\frac{1}{\sqrt{3}}\sqrt{x_0^2+x_1^2+x_2^2-x_3^2}\right\}
= \max\{|x_0|,\ldots,|x_3|\}.
\end{align*}

We proceed to consider the contribution from the vectors $\x \in
\ZZ^5$ for which $\|\x\|\leq~B$ and
$$
\Phi_1(\x)=\Phi_2(\x)=0, \quad x_0x_1x_2x_3x_4=0.
$$
Let us begin with the case $x_i=0$,
for $i\in \{0,1,2,3\}$. This hyperplane section leads us to estimate the
number of solutions to an equation of the form $x^2\pm y^2\pm 2z^2=0$
with $\hcf(x,y,z)=1$ and $\max\{|x|,|y|,|z|\}\leq B$. An application
of Lemma \ref{lem:bhb} therefore yields a contribution of $O(B)$
from this case.  Turning to the contribution from the case $x_4=0$ and
$x_0x_1x_2x_3\neq 0$, our task is to estimate the
number of primitive vectors $(x,y,z,t)\in (\ZZnz \cap [-B,B])^4$ such
that $xy=zt$ and $x^2+y^2+z^2=t^2$.
Eliminating $t$ produces an absolutely
irreducible  quartic curve $x^2y^2=z^2(x^2+y^2+z^2)$, which an
application of \cite[Theorem 3]{annal}
reveals to contribute $O(B^{1/2+\ve})$ points overall.
Bringing everything together we have therefore shown that
$$
N_{U,H}(B)=\frac{1}{2}\#\left\{\x\in \ZZnz^5: \begin{array}{l}
\hcf(x_0,\ldots,x_3)=1,\\
\max\{|x_0|,\ldots,|x_3|\}\leq B,\\
\Phi_1(\x)=\Phi_2(\x)=0, \\
\{|x_0|,|x_1|\}\neq \{|x_2|,|x_3|\}
\end{array}
\right\}+O(B).
$$

It remains to show that we can restrict attention to positive
values of $x_0,\ldots,x_4$. But this follows on noting that $\Phi_2$
is invariant under sign changes in the components of $\x$, and
$\Phi_1$ demands that $x_0x_1$ should share the same sign as $x_2x_3$.
Thus the above cardinality is $16N_1(B)$, with $N_1(B)$ given by
\eqref{eq:def-N1}, and the lemma follows.
\end{proof}

The next stage of the argument involves parametrising the solutions to
the equation $\Phi_1(\x)=0$. It is a simple exercise to show that the
set of $x_0,x_1,x_2,x_3\in\NN$ such that $x_0x_1=x_2x_3$ and
$\hcf(x_0,\ldots,x_3)=1$, is in bijection with the set 
of $\y=(y_{02},y_{03},y_{12},y_{13})\in\NN^4$ such that
$$
\hcf(y_{02},y_{13})=\hcf(y_{03},y_{12})=1,
$$
the relation being given by
$$
x_0=y_{02}y_{03}, \quad x_1=y_{12}y_{13}, \quad x_2=y_{02}y_{12},
\quad
x_3=y_{03}y_{13}.
$$
Define
$$
   \Psi(\y):=
\max\{y_{02}y_{03},y_{12}y_{13},y_{02}y_{12},y_{03}y_{13}\}.
$$
On recalling the definition \eqref{eq:def-N1}, it is easy to see that
$$
N_1(B)=\#\left\{(\y,x_4)\in \NN^5:\begin{array}{l}
\hcf(y_{02},y_{13})=\hcf(y_{03},y_{12})=1,\\
\Phi_2(\x)=0, ~\Psi(\y)\leq B,\\
y_{02}\neq y_{13}, ~y_{03}\neq y_{12}.
\end{array}
\right\},
$$
where $\x$ denotes $(y_{02}y_{03}, y_{12}y_{13}, y_{02}y_{12},
y_{03}y_{13}, x_4)$.

Given integers $a,b$, recall the definition \eqref{eq:conic} of the
plane conic $C_{a,b}\subset\Ptwo$.
Writing~$(a,b)=(y_{02},y_{13})$ and $(x,y,z)=(y_{03},y_{12},x_4)$, it
follows that $N_{1}(B)$ is the number of
$(a,b,x,y,z)\in \NN^5$ such that \eqref{eq:conic} holds, with
$$
\hcf(a,b)=\hcf(x,y)=1,
\quad ab, xy\neq 1,\quad
\max\{a,b\}\max\{x,y\}\leq B.
$$
For $a,b\in\NN$ we define the quantities
\begin{align}
   \label{eq:def-M}
\tM_{a,b}(B)
&:=
\#\left\{(x,y,z)\in C_{a,b}\cap \NN^3:
\begin{array}{l}
\hcf(x,y)=1, \\
\max\{a,b\}<\max\{x,y\},\\
\max\{a,b\}\max\{x,y\}\leq B
\end{array}
\right\},\\
   \label{eq:hatM}
\widehat{M}_{a,b}(B)
&:=
\#\left\{(x,y,z)\in C_{a,b}\cap \NN^3:
\begin{array}{l}
\hcf(x,y)=1, xy\neq 1\\
\max\{a,b\}\max\{x,y\}\leq B
\end{array}
\right\},\\
   \label{eq:def-M'}
M_{a,b}(B)&:=
\#\left\{(x,y,z)\in C_{a,b}\cap \NN^3:
\begin{array}{l}
\cp{x}{y},\\
\max\{a,b\}\max\{x,y\}\leq B
\end{array}
\right\},
\end{align}
where $C_{a,b}\subset \Ptwo$ is the conic \eqref{eq:conic}.
There is an abuse of notation here, in that we
are interested in vectors $(x,y,z)\in \NN^3$ that lie in the affine
cone above $C_{a,b}$, rather than points $[x,y,z]\in
C_{a,b}(\QQ)$ in which coordinates are chosen so that $x,y,z>0$.
Note that we must automatically have $xy\neq 1$ in the definition of
$\tM_{a,b}(B)$ since $1\leq \max\{a,b\}<\max\{x,y\}$.
Hence $\tM_{a,b}(B) \leq \widehat{M}_{a,b}(B)\leq M_{a,b}(B).$
We have
$$
N_{1 }(B) =\sum_{\substack{a,b\leq B \\ \gcd(a,b)=1, ab\neq 1}}
\widehat{M}_{a,b}(B)
$$
and we are now ready to establish the following result.

\begin{lemma}\label{lem:red-2}
We have
$$
N_1(B)=2\sum_{\substack{a,b<\sqrt{B}\\\hcf(a,b)=1, ab\neq 1}} \tM_{a,b}(B)
+O\big(B(\log B)^3\big),
$$
where $\tM_{a,b}(B)$ is given by \eqref{eq:def-M}.
\end{lemma}

\begin{proof}
Fundamental to our argument is the observation $(x,y,z)\in C_{a,b}\cap\NN^3$ if
and only if $(b,a,z)\in C_{y,x}\cap \NN^3$. From this it easily
follows that
$$
N_1(B)=2\sum_{\substack{a,b<\sqrt{B}\\\hcf(a,b)=1, ab\neq 1}} \tM_{a,b}(B)
+N_1'(B),
$$
where
$$
N_{1}'(B):=\sum_{\substack{a,b<\sqrt{B}\\\hcf(a,b)=1, ab\neq 1}}
\#\left\{(x,y,z)\in C_{a,b}\cap \NN^3:\begin{array}{l}
\hcf(x,y)=1,\\
\max\{a,b\}=\max\{x,y\}
\end{array}
\right\}.
$$
To establish the lemma it therefore suffices to show that
$
N_{1}'(B)\ll B(\log B)^3.
$
Without loss of generality we may view $N_{1}'(B)$ as being the
overall contribution from vectors such that
$$
\max\{a, b\}=b=y = \max\{x,y\},
$$
the remaining $3$ cases being handled in an identical fashion.
Arguing as above, we find that
$$
N_{1}'(B)
\leq 4
\#\left\{(a,x,y,z)\in \NN^4:
\begin{array}{l}
y^4+x^2a^2+y^2(a^2-x^2)=2z^2,\\
\gcd(y,xa)=1,\\
\max\{x,a\}\leq y\leq \sqrt{B}
\end{array}
\right\}.
$$
The overall contribution to the right hand side from the case $a=y$,
so that $a=y=~1$, is clearly $O(1)$.
We will estimate the remaining contribution by first fixing a choice of $y$ and
$a$, and then summing over all available $x$ and $z$.
First we note that~$y$ is necessarily odd in any $4$-tuple $(a,x,y,z)$.

For fixed $a,y$ the form $x^2(y^2-a^2)+2z^2$ defines a positive definite
binary quadratic form of discriminant $D=-8(y^2-a^2)\neq 0$. In particular
the total number of available $x,z$ is bounded by the  number of
representations $R_D(y^2(y^2+a^2))$ of $y^2(y^2+a^2)$ by a complete
system of inequivalent forms of discriminant $D$.
It easily follows that
$$
N_{1}'(B)
\leq 4
\sum_{\substack{y,a\in \NN\\a<y\leq \sqrt{B}\\2\nmid y, ~\hcf(a,y)=1}}
R_D(y^2(y^2+a^2))+O(1),
$$
with $D=-8(y^2-a^2)$.
By the classical theory of binary quadratic forms (see Iwaniec and
Kowalski \cite[\S 22]{iw-kow}, for example) we have
$$
R_D(y^2(y^2+a^2))=2 \sum_{d\mid y^2(y^2+a^2) } \chi_D(d)\ll
\tau\big(y^2(y^2+a^2)\big) =
\tau(y^2)\tau(y^2+a^2),
$$
where $\chi_D$ is the Kronecker symbol. The last equality follows since
$\cp{a}{y}$.
An application of Lemma \ref{lem:sky} with $\delta=\kappa=0$ 
therefore reveals that there is an
absolute constant $c>0$ such that
\begin{align*}
N_{1}'(B)
\ll \sum_{y\leq \sqrt{B}} \tau(y^2)
\sum_{a\leq \sqrt{B}}\tau(y^2+a^2) +\sqrt{B}
&\ll \sqrt{B}\log B
\sum_{y\leq \sqrt{B}} \phidd(y)^c\tau(y^2),
\end{align*}
where $\phidd$ is given by \eqref{eq:funcpp}.
Finally, on opening up the divisor function, we note that 
\begin{align*}
\sum_{y\leq \sqrt{B}} \phidd(y)^c\tau(y^2)
&\leq 
\sum_{uvz\leq \sqrt{B}} \phidd(uvz)^c,
\end{align*}%T%
since $d\mid y^2$ if and only if $d=uv^2$ with $u$ square-free
 and $uv\mid y$. Noting that $(\sigma+\sigma')^5\leq
 2^5(\sigma^5+\sigma'^5)$ it easily follows from the definition of
 $\phidd$ that $\phidd(mn)\leq \phidd(m)\phidd(n)$ for any $m,n\in
 \NN$. Hence
$$%T%
\sum_{y\leq \sqrt{B}} \phidd(y)^c\tau(y^2)
\leq 
\sum_{uv\leq \sqrt{B}} \phidd(uv)^{2^5c}
\sum_{z\leq \sqrt{B}/(uv)} \phidd(z)^{2^5c}.
$$
Combining \eqref{eq:1} with partial summation 
readily leads to the conclusion that this has order
$\sqrt{B} (\log B)^2$. 
Thus $N_{1}'(B)\ll B(\log B)^3$, as required to complete 
the proof of the lemma.
\end{proof}

\section{Reducing the range of summation}\label{s:ham}

It will greatly facilitate our arguments if we can reduce the
range of summation for $a,b$ in the statement of Lemma \ref{lem:red-2}.
The main aim of this section is to establish the following result.

\begin{lemma}\label{lem:ab-small}
Let $T>0$ and recall the definition \eqref{eq:def-M'} of
$M_{a,b}(B)$. Then we have 
\begin{align*}
\sum_{\substack{\sqrt{B}/T<\max\{a,b\}\leq \sqrt{B}\\
\cp{a}{b}, ab\neq 1}} M_{a,b}(B) &\ll B(\log B)^3\log (T+2),\\
\sum_{\substack{T\min\{a,b,|a-b|\}< \max\{a,b\}\leq \sqrt{B}\\
\cp{a}{b}, ab\neq 1}} M_{a,b}(B) &\ll \frac{B(\log B)^4}{T^{2/3}}+B(\log B)^3.
\end{align*}
\end{lemma}

Applying the second estimate in Lemma \ref{lem:ab-small} with $T=1/2$ and
inserting this into Lemmas \ref{lem:red-1}
and \ref{lem:red-2}, we easily conclude that
$$
N_{U,H}(B)\ll B(\log B)^4.
$$
Achieving a corresponding lower bound for $N_{U,H}(B)$ is 
straightforward. We will have to work much harder to deduce an asymptotic
formula.

\begin{proof}[Proof of Lemma \ref{lem:ab-small}]
For the conic $C_{a,b}$ defined in \eqref{eq:conic} we have
$\Delta=-2(a^4-b^4)$ and $\Delta_0\ll 1$, since $\cp{a}{b}$ by
assumption. Taking 
$B_1=B_2=B/\max\{a,b\}$ and $B_3=\sqrt{2}B$, 
it therefore follows from Lemma \ref{lem:bhb} that 
$$
M_{a,b}(B)
\ll  \tau(|a^4-b^4|)\Big(1+
\frac{B}{\max\{a,b\}^{2/3}|a^4-b^4|^{1/3}}\Big). 
$$
Since $|a^4-b^4|\geq 
\max\{a,b\}^{3}|a-b|$
for positive integers $a,b$, we
see that
\begin{equation}\label{majMab}\begin{split}
M_{a,b}(B)
&\ll  \tau(|a^4-b^4|)\Big(1+
\frac{B}{\max\{a,b\}^{5/3}|a-b|^{1/3}}\Big). 
\end{split}\end{equation}
Thus there are two terms to consider, both of which we must sum over
the relevant values of $a,b.$  
It follows from Lemma \ref{lem:nair} that the first term in 
contributes 
$\ll B(\log B)^3$, which is satisfactory.
Moreover, it follows from Lemma \ref{lem:nair'} with 
$(p,q)=(5/3,1/3)$
that the second term also makes a satisfactory contribution. 
This completes the proof of the lemma.
\end{proof}

Recall the definitions \eqref{eq:zet} of $Z_1$ and $Z_2$.
Then $\log Z_1=(\log B)/(\log\log B)$ and 
$2^{Z_2}=(\log B)^{\log 2}$.
Our next task is twofold. Firstly, we would like to be able to
restrict attention to values of $a,b$ belonging to the set
\begin{equation}
   \label{eq:AA}
   \A_1:=\left\{(a,b)\in \NN^2:
\begin{array}{l}
\max\{a,b\}< \sqrt{B}/Z_1^c,\\
\max\{a,b\}\leq Z_2^2\min\{a,b,|a-b|\},\\
\cp{a}{b}, ab\neq 1,\\
\nu_2(a^2-b^2)\leq Z_2
\end{array}
\right\}.
\end{equation}
Here  $c>0$ is an absolute constant that
when appearing as an exponent of $Z_1$, we will view as being
sufficiently large to ensure that all of our error terms are
satisfactory. Secondly, we wish to show that $\tM_{a,b}(B)$ can be
replaced by $M_{a,b}(B)$ in Lemma~\ref{lem:red-2}, with a negligible error.
Both of these objectives are achieved in the following result.

\begin{lemma}\label{lem:red-3}
We have
$$
N_1(B)=2\sum_{(a,b)\in \A_1} M_{a,b}(B)
+O\Big(\frac{B(\log B)^4}{\log\log B}\Big),
$$
where $M_{a,b}(B)$ is given by \eqref{eq:def-M'} 
and $\A_1$ is given by \eqref{eq:AA}.
\end{lemma}

\begin{proof}
Taken together with Lemma \ref{lem:red-2}, an application of Lemma
\ref{lem:ab-small} easily implies that
$$
N_1(B)=2\sum_{\substack{\max\{a,b\}<\sqrt{B}/Z_1^c\\
\max\{a,b\}\leq Z_2^2\min\{a,b,|a-b|\}\\
\hcf(a,b)=1, ab\neq 1}}
\tM_{a,b}(B) +
O\Big(\frac{B(\log B)^4}{\log \log B}\Big).
$$
Let us proceed by indicating how
to replace $\tM_{a,b}(B)$ by $M_{a,b}(B)$ in the summand.
Using the fact that $\max\{a,b\}< \sqrt{B}/Z_1^c$ it is easy to see that
\begin{align*}
M_{a,b}(B)-\tM_{a,b}(B)
&\leq
\#\left\{(x,y,z)\in C_{a,b}\cap \NN^3:
\begin{array}{l}
\hcf(x,y)=1,\\ 
\max\{x,y\} < \sqrt{B}/Z_1^c
\end{array}
\right\}\\
& \ll  \tau(|a^4-b^4|)\Big(1+
\frac{B^{2/3}}{Z_1^{4c/3}\max\{a,b\}|a-b|^{1/3}}\Big),
\end{align*}
on taking $B_1=B_2=\sqrt{B}/Z_1^c$ and $B_3=\sqrt{2}B/Z_1^{2c}$ in Lemma
\ref{lem:bhb}.
Employing the simple upper bound 
$\tau(|a^4-b^4|)\leq Z_1^{O(1)}$,  we therefore obtain the overall contribution 
$$
\ll \frac{BZ_1^{O(1)}}{Z_1^{2c}}+
\frac{B^{2/3}Z_1^{O(1)}}{Z_1^{4c/3}}\sum_{a,b< \sqrt{B}/Z_1^c} \frac
{1}{\max\{a,b\}|a-b|^{1/3}}
\ll \frac{BZ_1^{O(1)}}{Z_1^{2c}}.
$$
This is satisfactory if $c$ is chosen to be large enough. 

In order to restrict to a summation over $(a,b)\in \A_1$ we must
consider the contribution from $a,b$ such that $\nu_2(a^2-b^2) > Z_2$.
In view of Lemma \ref{lem:nair}, Lemma \ref{lem:nair'} and 
\eqref{majMab}, one sees that the overall 
contribution is bounded by
\begin{align*}
&\ll \sum_{\substack{a,b 
<\sqrt{B}/Z_1^c\\\hcf(a,b)=1, ab\neq 
1\\\nu_2(a^2-b^2) > Z_2}}
   \tau(|a^4-b^4|)\Big(1+ \frac{B}{\max\{a,b\}^{5/3}|a-b|^{1/3}}\Big)\\
&\ll 
 \frac{B(\log B)^3}{Z_1^{2c}}+
\frac{B}{Z_2}
\sum_{\substack{a,b 
<\sqrt{B}/Z_1^c\\ a-b\neq 0}}
   \frac{\nu_2(a^2-b^2)\tau(|a^4-b^4|)}{\max\{a,b\}^{5/3}|a-b|^{1/3}}\\
&\ll \frac{B(\log B)^4}{\log \log B}.
\end{align*}
This is satisfactory.
\end{proof}

The proof of Lemma \ref{lem:red-3} exhibits a feature 
common to much of what follows. 
The notation $Z_1^{O(1)}$ means that we allow an arbitrary,
but absolutely bounded, power of $Z_1$. The term
$Z_1^{c}$ should be thought of as a parameter since we are
free to take $c$ large enough to nullify the effect of the $O(1)$.
Thus one may think of $Z_1^{O(1)}/Z_1^c$ as smaller than any
negative power of $\log B$.

\section{Parametrisation of the conics}\label{sec:par}

In this section we concern ourselves with estimating $M_{a,b}(B)$, as
given by \eqref{eq:def-M'}. Recall the definition \eqref{eq:conic} of
$C_{a,b}$. Fundamental to our approach is the
observation that for each $a,b\in\NN$, the conic $C_{a,b}$
always contains the rational point $\xi$ in \eqref{eq:xi}.
We are therefore in a position to parametrise all of the rational
points on the conic by considering the residual intersection with
$C_{a,b}$ of an arbitrary line through $\xi$.
Define the binary quadratic forms
\begin{equation}\label{eq:Qi}
\begin{split}
Q_1(s,t)&:=2s^2+(a^2-b^2)t^2-4ast, \\
Q_2(s,t)&:=-2s^2+(a^2-b^2)t^2, \\
Q_3(s,t)&:=-2as^2 +2(a^2-b^2) st- a(a^2-b^2)t^2,
\end{split}
\end{equation}
for given $a,b\in \NN$ such that $\cp{a}{b}$ and $ab\neq 1$. In
particular we will always have $a^2\pm b^2\neq 0$.
One easily checks that
\begin{equation}
   \label{eq:velib}
   Q_3(s,t)=t^{-1}\big(sQ_1(s,t)+(s-at)Q_2(s,t)\big).
\end{equation}
The main aim of this section is to establish the following
result.

\begin{lemma}\label{lem:para-1}
We have
$$
M_{a,b}(B)=\#\left\{
(s,t)\in \ZZ^2:
\begin{array}{l}
\cp{s}{t}, ~s(s-at)\neq 0,\\
s/t\neq (a^2 -b^2)/(2a),\\
t>0, ~0<- Q_3(s,t),\\
\mbox{$0<-Q_j(s,t)\leq \frac{\lambda B}{\max\{a,b\}}$ for $j=1,2$}
\end{array}
  \right\}+O(1),
$$
where $\lambda=\gcd(Q_1(s,t),Q_2(s,t))$ and the implied constant is
absolute.
\end{lemma}

\begin{proof}
Recall the definitions \eqref{eq:conic} and \eqref{eq:xi} of $C_{a,b}$
and $\xi$, respectively. Let
$$
L_{\xi}:=\big\{(a^2-b^2)x-(a^2+b^2)y-2az=0\big\}.
$$
This is the tangent line to $C_{a,b}$ along $\xi$.
Let $\mathcal{L}$ denote the set of projective lines in~$\Ptwo$
that pass through $\xi$, and let $\mathcal{L}(\QQ)$ be the corresponding
subset that are defined over~$\QQ$. We will write
$\mathcal{U}=\mathcal{L}\setminus \{L_\xi\}.$
Now let $U\subset C_{a,b}$ denote the open subset formed by deleting
$\xi$ from the conic.
The sets $U(\QQ)$ and $\mathcal{U}(\QQ)$ are in
bijection.

The general element of $\mathcal{L}(\QQ)$ is given by
$$
L_{s,t}:=\big\{sx+(s-at)y-tz=0\big\},
$$
for $s,t \in \ZZ$ such that $\gcd(s,s-at,t)=\gcd(s,t)=1$.
In order to have a point in $\mathcal{U}(\QQ)$ we must insist 
that $(s,t)\neq (a^2-b^2,2a)$. We can replace this by the conditions
$t\neq 0$ and 
$s/t\neq (a^2 -b^2)/(2a)$, with an error $O(1)$. 
Finally we note that there 
is a bijection between the lines for which $t<0$ and
$t>0$. We henceforth fix our attention on the latter.

To make the bijection between $U(\QQ)$ and $\mathcal{U}(\QQ)$
completely explicit a routine calculation reveals that a point
$[x,y,z]$ is in the intersection $L_{s,t}\cap C_{a,b}$
if and only if~$x=-y$, or else $Q_2(s,t)x=Q_1(s,t)y$, in the
notation of \eqref{eq:Qi}. The first alternative leads us to the point
$\xi$, which is to be ignored. The second alternative
implies that
$$
tQ_1(s,t)z=\big(sQ_1(s,t)+(s-at)Q_2(s,t)\big)x=tQ_3(s,t)x,
$$
by \eqref{eq:velib}, in the equation for 
$L_{s,t}$.  Any vector $(x,y,z)\in \ZZ^3$ that
represents a point in $C_{a,b}(\QQ)$ is primitive if and only if $\gcd(x,y)=1$.
Writing $\lambda$ as in the statement of the lemma, we therefore find that
$$
(x,y,z)=\pm (Q_1(s,t)/\lambda, Q_2(s,t)/\lambda, Q_3(s,t)/\lambda).
$$
for any $[x,y,z]\in (U\cap L_{s,t})(\QQ)$.

So far we have recorded an explicit bijection between elements of
$U(\QQ)$ and $\mathcal{U}(\QQ)$. In deriving an expression for
$M_{a,b}(B)$ in terms of this bijection we will need to restrict the
corresponding values of $s,t$ that are to be considered.
Specifically we will only be interested in the values of $(s,t)\in
\ZZ\times \ZZ_{>0}$ for which $\cp{s}{t}$ and the 
corresponding values of $x,y,z$ lie in
the region defined by the inequalities
$$
x,y,z>0, \quad \max\{a,b\}\max\{x,y\}\leq B.
$$
Finally, we will want to exclude the possibility
that $s(s-at)=0$. Since there are only $O(1)$ such
values of $s,t$ to worry about, this therefore
concludes the proof that
$$
M_{a,b}(B)=\sum_{\ve\in\{\pm 1\}}\#\left\{
(s,t)\in \ZZ^2:
\begin{array}{l}
\cp{s}{t}, ~s(s-at)\neq 0,\\
s/t\neq (a^2 -b^2)/(2a),\\
t>0, ~0<\ve Q_3(s,t),\\
\mbox{$0<\ve Q_j(s,t)\leq \frac{\lambda B}{\max\{a,b\}}$ for $j=1,2$}
\end{array}
  \right\}+O(1).
$$

Our final task is to show that $\ve=-1$ in this expression.
Suppose that $\ve=+1$. The conditions imply that
$Q_2(s,t)>0$ and $t>0$.
But then it follows that $\sign(a-b)=1$. 
However  we also have $Q_3(s,t)>0$, whence
$$
-\sqrt{a^2-b^2}(s-t')^2-(2a-\sqrt{a^2-b^2})s^2-(a-\sqrt{a^2-b^2}){t'}^2>0,
$$
where we have written ${t'}=t\sqrt{a^2-b^2} $ for ease of notation.
This contradiction  establishes the lemma.
\end{proof}

Lemma \ref{lem:para-1} allows us to translate the underlying problem
to one that involves counting primitive lattice points in a
complicated region contained in $\RR^2$. Before we proceed to consider
this region in more detail, it will be necessary to gain a better
understanding of the greatest common divisor $\lambda$.

\begin{lemma}\label{lem:la}
Let $a,b$ be coprime positive
integers such that $ab\neq 1$.
Let $s,t$ be coprime integers and let
$\lambda=\gcd(Q_1(s,t),Q_2(s,t))$.  Then
we have $\lambda=2^\nu \lambda_1\lambda_2$, where
\begin{equation}
   \label{eq:nu}
\nu=
\begin{cases}
0, & \mbox{if $2\mid ab$ and $2\nmid t$},\\
1, & \mbox{if $2\mid ab$ and $2\mid t$},\\
1, & \mbox{if $2\nmid ab$ and $2\nmid s$},\\
\min\{2+\nu_2(s), \nu_2(a^2-b^2)\}, & \mbox{if $2\nmid ab$ and $2\mid s$},
\end{cases}
\end{equation}
and
$$
\lambda_1=\odd{s}{a^2-b^2}, \quad
\lambda_2=\odd{s-at}{a^2+b^2},
$$
in the notation of \eqref{eq:odd-notat}.
\end{lemma}

\begin{proof}
Let us write $\lambda=2^\nu \lambda'$ for $\nu\geq 0$ and
$\lambda'\in \NN$ odd.
Observe that
\begin{equation}
   \label{eq:lam-exp}
\begin{split}
\lambda=\hcf(Q_1,Q_2)&=\hcf(Q_1-Q_2, Q_2)=\gcd(4s(s-at), -2s^2+(a^2-b^2)t^2).
\end{split}
\end{equation}
The precise value of $\nu$ will depend intimately on the $2$-adic
valuations of $a,b,s$ and~$t$. Let us write
$$
a=2^\alpha a', \quad b=2^\beta b', \quad s=2^\sigma s', \quad t=2^\tau t',
$$
for $\alpha, \beta, \sigma, \tau\geq 0$ and $a'b's't'$ odd.
It follows from \eqref{eq:lam-exp} that
$$
\nu=\min\{2+\sigma+\nu_2(2^\sigma s'-2^{\alpha+\tau}t'a') ,
\nu_2(-2^{1+2\sigma}{s'}^2+(2^{2\alpha}{a'}^2-2^{2\beta}{b'}^2)2^{2\tau}{t'}^2)
\}.
$$
Suppose first that $\alpha=\beta=0$. Then
$$
\nu=\min\{2+\sigma+\nu_2(2^\sigma s'-2^{\tau}t'a') ,
\nu_2(-2^{1+2\sigma}{s'}^2+({a'}^2-{b'}^2)2^{2\tau}{t'}^2)
\}.
$$
Thus either $\sigma=0$, in which case $\nu=1$, or else $\sigma\geq
1$. In the latter case $\tau=0$ and it follows that
$$
\nu=\min\{2+\sigma ,
\nu_2(-2^{1+2\sigma}{s'}^2+({a'}^2-{b'}^2){t'}^2)\}=\min\{2+\sigma
,\nu_2({a'}^2-{b'}^2)\}.
$$
Differentiating according to whether $\alpha\geq 1$ and $\beta=0$, or
$\alpha=0$ and $\beta\geq 1$,
it is easily checked that
$$
\nu=
\begin{cases}
1, & \mbox{if $\tau \geq 1$,}\\
0, & \mbox{if $\tau =0$.}
\end{cases}
$$

Turning to the odd part $\lambda'$ of $\lambda$, 
we deduce from \eqref{eq:lam-exp}
that
\begin{align*}
\lambda'
&=\odd{s(s-at)}{-2s^2+(a^2-b^2)t^2}\\
&=\odd{s}{-2s^2+(a^2-b^2)t^2}
\odd{s-at}{-2s^2+(a^2-b^2)t^2}\\
&=\odd{s}{a^2-b^2}\odd{s-at}{a^2+b^2}.
\end{align*}
This completes the proof of the lemma.
\end{proof}

Note that the last equality in \eqref{eq:nu} is only possible if
$\nu\geq 2$ since $a^2-b^2$ is divisible by $4$ if $a$ and
$b$ are both odd.
Define the three quadratic forms
\begin{equation}
   \label{eq:def_pqr}
   \begin{split}
p_u(s,t)&:-2s^2- \frac{1-u^2}{|1-u^2|}
t^2+\frac{4}{\sqrt{|1-u^2|}} st,\\
q_u(s,t)&:=2s^2 - \frac{1-u^2}{|1-u^2|}  t^2,\\
r_u(s,t)&:=2s^2-2 \frac{1-u^2}{\sqrt{|1-u^2|}} st+
\frac{1-u^2}{|1-u^2|}t^2,
\end{split}
\end{equation}
for any positive $u\neq 1$.  We may then write
$$
Q_1(s,t)=-p_{b/a}(s,\alpha t),\quad
Q_2(s,t)=-q_{b/a}(s,\alpha t), \quad
Q_3(s,t)=-a r_{b/a}(s,\alpha t),
$$
in \eqref{eq:Qi}, where $\alpha=\sqrt{|a^2-b^2|}$.

For any $X>0$ we define the region
\begin{equation}   \label{eq:RBA}
\begin{split}
\R(X)
&:=
   \left\{(s,t)\in \RR\times \RR_{>0}:
\begin{array}{l}
Q_3(s,t)<0,\\
0>Q_1(s,t), Q_2(s,t)\geq -X
\end{array}
\right\}\\
&=
   \left\{(s,t)\in \RR\times \RR_{>0}:
\begin{array}{l}
0<r_{b/a}(s,\sqrt{|a^2-b^2|}t),\\
0<p_{b/a}(s,\sqrt{|a^2-b^2|}t)\leq X,\\
0<q_{b/a}(s,\sqrt{|a^2-b^2|}t)\leq X
\end{array}
\right\}.
\end{split}
\end{equation}
Furthermore, we set
\begin{equation}
   \label{eq:RBA'}
\R'(X):=
   \left\{(s,t)\in \R(X): s(s-at)\neq 0,~s/t\neq (a^2 -b^2)/(2a)
\right\}
\end{equation}
and 
\begin{equation}\label{eq:R-dag}
\R^\dagger(X):=
 \left\{(s,t)\in \RR^2:
\begin{array}{l}
\mbox{$Q_j(s,t)\neq 0$ for $1\leq j\leq 3$},\\
|Q_1(s,t)|, |Q_2(s,t)|\leq X\\
st(s-at)\neq 0,~s/t\neq (a^2 -b^2)/(2a)
\end{array}
\right\}.
\end{equation}
Bringing together Lemmas \ref{lem:para-1} and \ref{lem:la}, we may
deduce that
\begin{equation}\label{eq:ruff}
M_{a,b}(B)=\sum_{\nu=0}^\infty
\Osum_{\lambda_1\mid a^2-b^2}
\Osum_{\lambda_2\mid a^2+b^2}
L(a,b;B;\nu, \la_1, \la_2) +O(1),
\end{equation}
where
$L(B)=L(a,b;B;\nu, \la_1, \la_2)$ is the number of
$(s,t)\in \ZZ^2$ subject to the following conditions:
\begin{enumerate}
\item\label{i:1} $\cp{s}{t}$;
\item\label{i:2} $(s,t)\in \R(2^\nu \la_1\la_2 B/\max\{a,b\})$;
\item\label{i:-} $s/t\neq (a^2 -b^2)/(2a) $
and $s(s-at)\neq 0$;
\item\label{i:5} $\la_1 \mid s$ and $\la_2 \mid s-at$;
\item\label{i:6} $\odd{s/\la_1}{(a^2-b^2)/\la_1}=1$ and
$\odd{(s-at)/\la_2}{(a^2+b^2)/\la_2}=1$; and
\item\label{i:7} the $2$-adic orders of $s,t$ are 
determined by $\nu$ via \eqref{eq:nu}.
\end{enumerate}
Here we recall the convention that the
symbol $\osum$ implies a restriction to odd parameters in the summation
and  we note that \eqref{i:2} and \eqref{i:-} are together equivalent
to
 $(s,t)\in \R'(2^\nu \la_1\la_2 B/\max\{a,b\})$ in the notation of \eqref{eq:RBA'}.

\section{Removing the coprimality conditions}\label{sec:1}

The way forward should now be clear. For given values of $a,b$, and
appropriate values of $\nu, \la_1$ and $\la_2$, we must attempt to
produce an asymptotic formula for the number of primitive lattice
points in a complicated region in $\RR^2$. We will do so using
exponential sums. The first step, however,  is to remove
the coprimality conditions that go into the definition of $L(B)$.
Let $\k=(k_1,k_2)$ and $\bla=(\la_1,\la_2)$ such that
$\gcd(k_1k_2\la_1\la_2,ab)=1$ and let $\ell\in \NN$. 
Define the set
\begin{equation}
   \label{eq:lattice-1}
   \sfl(\k,\bla,\ell):=\{ (s,t)\in \ZZ^2: \mbox{
$[k_1\la_1,\ell ]\mid s$, $k_2\la_2 \mid s-at$ and $\ell \mid
   t$}\}.
\end{equation}
Then $\sfl(\k,\bla,\ell) \subseteq \ZZ^2$ is a lattice of
rank $2$, with determinant as in the following result.

\begin{lemma}\label{lem:det}
Let $\k, \bla, \ell$ be as above. Then we have
$$
\det 
\sfl(\k,\bla,\ell)=\frac{k_1k_2\la_1\lambda_2\ell^2}{\gcd(k_1
\lambda_1,\ell) \gcd(k_2\lambda_2,\ell)}. 
$$
\end{lemma}

\begin{proof}
Writing $s=\ell s'$
and $t=\ell t'$ it follows that
$\det \sfl=\ell^2 \det \sfl'$,
where $\sfl'$ is the set of $(s',t')\in \ZZ^2$ such that
$\lambda_1'\mid s'$  and $\lambda_2'\mid s'-at'$, with
$\lambda_i'$ equal to $k_i\la_i/\gcd(k_i\la_i, \ell)$. The proof of
the lemma
then follows on noting that $\gcd(\lambda_1'\lambda_2',a)=1$.
\end{proof}

We proceed to remove the coprimality
conditions that appear in \eqref{i:6}.
Thus an application of M\"obius inversion yields
\begin{equation}\label{eq:Lk1k2}
L(B)=\Osum_{k_1\mid (a^2-b^2)/\la_1}
\Osum_{k_2\mid (a^2+b^2)/\la_2}
\mu(k_1)\mu(k_2)
L_{k_1,k_2}(B),
\end{equation}
where now $L_{k_1,k_2}(B)$
is the number of $(s,t)\in \sfl(\k,\bla,1)$ such that
\eqref{i:1}--\eqref{i:-} and \eqref{i:7} hold
in the definition of $L(B)$.

We now  consider the regions \eqref{eq:RBA} and \eqref{eq:RBA'} in
more detail.
Define
\begin{equation}
   \label{eq:SBA}
\mathcal{S}_u:=
   \left\{(s,t)\in \RR\times \RR_{>0}:
0<p_u(s,t),q_u(s,t)\leq 1, ~r_u(s,t)>0
\right\}
\end{equation}
and
\begin{equation}
   \label{eq:SBA'}
\mathcal{S}_u':= \left\{(s,t)\in \mathcal{S}_u:
s(s-t/\sqrt{|1-u^2|})\neq 0, ~
s/t\neq (1-u^2)/(2\sqrt{|1-u^2|})
\right\},
\end{equation}
for any positive $u\neq 1$.
We  observe that 
\begin{equation}\label{eq:train}
(s,t)\in \R'(X) 
\Longleftrightarrow
\Big(\frac{s}{\sqrt{X}}, \frac{\sqrt{|a^2-b^2|}t}{\sqrt{X}}\Big)\in
\mathcal{S}_{b/a}'. 
\end{equation}
We then have the following result.

\begin{lemma}\label{lem:volume}
Let $X>0$ and $u\neq 1$ be positive. Then we have
$$
\vol(\R'(X))=\vol(\R(X))
=\frac{X\vol(\Sab)}{\sqrt{|a^2-b^2|}}.
$$
Furthermore, we have
$$
\R(X)\subseteq [-c\sqrt{X},c\sqrt{X}]\times (0,c\sqrt{X/|a^2-b^2|}],
\quad
\mathcal{S}_u \subseteq [-c,c]\times (0,c],
$$
for an absolute constant $c>0$.
\end{lemma}

\begin{proof}
The first part of the lemma is self-evident, and so it remains to establish the
bounds on $\R(X)$ and $\mathcal{S}_u$ in the second part.
For this it will clearly suffice to show that
$s\ll 1$ and $0< t\ll 1$ for any $(s,t)\in \mathcal{S}_u$.
Suppose first that $u^2>1$. Then it follows from the inequality
$q_u(s,t)\leq 1$ that $2s^2+t^2\leq 1$, whence $s,t\ll 1$ in this case.
If $u^2<1$ then we have $0<2s^2-t^2\leq 1$ and
$0<-2s^2-t^2+4st/\sqrt{1-u^2}\leq~1$.
In particular it follows that $0<t<\sqrt{2}s$.
Let $\eta>0$. If~$t<s\sqrt{2-\eta}$ then
$1\geq~2s^2-t^2>\eta~s^2 $. Thus $s\leq 1/\sqrt{\eta}$ and
$t\leq \sqrt{2/\eta}$ in this case.
Alternatively, if $s\sqrt{2-\eta}\leq t<\sqrt{2}s$,
then we deduce that
\begin{align*}
1\geq -2s^2-t^2+\frac{4st}{\sqrt{1-u^2}}
&=
s^2\Big(-2-\Big(\frac{t}{s}\Big)^2+\frac{4}{\sqrt{1-u^2}}\Big(\frac{t}{s}\Big)\Big)
\\
&>
4s^2(-1+\sqrt{2-\eta}),
\end{align*}
since $1/\sqrt{1-u^2}> 1$.
Taking $\eta=1/2$ it therefore follows that $s,t \ll 1$ in every case,
which completes the proof of the lemma.
\end{proof}

Recall the definition \eqref{eq:AA} of the set $\A_1$. For any $(a,b)\in
\A_1$ it follows from \eqref{eq:nu} that we are only interested in values of
$$
\nu\leq \max\{1,\nu_2(a^2-b^2)\}\leq Z_2 = \log\log B.
$$
Returning to our expression \eqref{eq:Lk1k2} for $L(B)$, we may now
show that there is a negligible
contribution from large values of $k_1,k_2$
and also from points for which the $2$-adic order of $s$ is large.
Let
\begin{equation}
   \label{eq:K}
K:=(\log B)^{100}
\end{equation}
and let
\begin{equation}
   \label{eq:XX}
X=\frac{2^\nu \la_1\la_2 B}{\max\{a,b\}}.
\end{equation}
We may now establish the following result. 

\begin{lemma}\label{lem:bbq}
We have
$$
\sum_{(a,b)\in \A_1}
\sum_{\nu\leq Z_2}
\Osum_{k_1\lambda_1\mid a^2-b^2}
\Osum_{k_2\lambda_2\mid a^2+b^2}
L_{k_1,k_2}(B) \ll B(\log
B)^{4-2\log 2 +\ve} ,
$$
where the summations are subject to $\max\{k_1,k_2\}> K$ or
$\nu_2(s)> 2Z_2$.
\end{lemma}

\begin{proof}
Let $\sigma\geq 0$ and let $L_{k_1,k_2}(B;\sigma)$ denote the contribution to 
$L_{k_1,k_2}(B)$ from $s$ such that 
$\nu_2(s)\geq \sigma$. 
Let $\k'=(2^\sigma k_1,k_2)$.
Lemma \ref{lem:det} reveals  that 
$\sfl(\k',\bla,1)$ is an integer lattice of rank 
$2$
and determinant $\la_1\la_2k_1k_2=2^\sigma \la_1\la_2k_1k_2$.
Hence it follows from Lemmas \ref{lem:lattice} and \ref{lem:volume} that
\begin{equation}\label{eq:rich}
L_{k_1,k_2}(B;\sigma)
\ll 1+ \frac{\vol\R(X)}{2^\sigma \la_1\la_2k_1k_2}
\ll 1+ \frac{2^{\nu-\sigma} B}{\max\{a,b\}^{3/2}|a-b|^{1/2}k_1k_2},
\end{equation}
where $X$ given by \eqref{eq:XX} and we have taken $|a^2-b^2|\geq \max\{a,b\}|a-b|$.
The contribution from the first term is 
\begin{align*}
\sum_{(a,b)\in \A_1}
\sum_{\nu\leq Z_2}
\Osum_{k_1\lambda_1\mid a^2-b^2}
\Osum_{k_2\lambda_2\mid a^2+b^2} 1
&\ll
\log \log B\sum_{a,b}
\tau_3(|a^4-b^4|)
\ll
\frac{B(\log B)^{6+\ve}}{Z_1^{2c}},
\end{align*}
by Lemma \ref{lem:nair}, which is satisfactory. 

Suppose that we are dealing with the overall contribution
from $\max\{k_1,k_2\}> K$.  Note that
$$
\sum_{\nu \leq Z_2}
2^\nu \Osum_{k_1\lambda_1\mid a^2-b^2}
\Osum_{k_2\lambda_2\mid a^2+b^2} 1\ll \log B 
\Osum_{k\la \mid a^4-b^4}1 \leq \tau_3(|a^4-b^4|) \log B.
$$
We may therefore take $\sigma=0$ 
and  $k_1k_2\geq
\max\{k_1,k_2\}> K$ in the second term of \eqref{eq:rich}, 
together with the lower bound $|a-b|\geq Z_2^{-2}\max\{a,b\}$ from
\eqref{eq:AA}, to obtain the overall contribution
$$
\ll
\frac{B Z_2 \log B}{K}
\sum_{(a,b)\in \A_1}
\frac{\tau_3(|a^4-b^4|)}{\max\{a,b\}^{2}}
\ll  B,
$$
by Lemma \ref{lem:nair}.

Next suppose we are dealing with 
the contribution 
from $\nu_2(s)> 2Z_2$. Taking $\sigma=\lfloor 2Z_2\rfloor$ we 
deduce from \eqref{eq:nu} that
$
\nu=
\nu_2(a^2-b^2)$.
Hence the overall contribution from the second term in \eqref{eq:rich}
is 
\begin{align*}
&\ll \frac{B}{2^{\lfloor 2Z_2 \rfloor}}\sum_{(a,b)\in \A_1}
\Osum_{k_1\lambda_1\mid a^2-b^2}
\Osum_{k_2\lambda_2\mid a^2+b^2} 
\frac{2^{\nu_2(a^2-b^2)} }{\max\{a,b\}^{3/2}|a-b|^{1/2}k_1k_2}\\
&\ll \frac{B(\log \log B)^2}{(\log B)^{\log 2}}\sum_{(a,b)\in \A_1}
\Osum_{\lambda_1\mid a^2-b^2}
\Osum_{\lambda_2\mid a^2+b^2} 
\frac{1 }{\max\{a,b\}^{3/2}|a-b|^{1/2}}\\
&\ll \frac{B(\log \log B)^2}{(\log B)^{\log 2}}\sum_{(a,b)\in \A_1}
\frac{\tau(|a^4-b^4|) }{\max\{a,b\}^{3/2}|a-b|^{1/2}},
\end{align*}
since $\sum_{k\mid n}1/k\ll \log\log n$ and $\nu_2(a^2-b^2)\leq Z_2$
in $\A_1$. 
Applying \eqref{eq:anna-swim} we see that this is 
$O(B(\log B)^{4-\log 2+\ve})$, which is acceptable
and so completes the proof of the lemma.
\end{proof}

In view of this result we may now proceed under the assumption that
$k_1,k_2\leq~K$ in \eqref{eq:Lk1k2}, and we may freely replace 
$L_{k_1,k_2}(B)$ by $\widetilde{L}_{k_1,k_2}(B)$, which is 
the same but with the additional restriction that $\nu_2(s)\leq~2Z_2$.
We now apply M\"obius inversion in 
\eqref{eq:Lk1k2}, allowing us to take
\begin{equation}\label{eq:Lk1k2l}
L(B)=\Osum_{\substack{k_1\mid (a^2-b^2)/\la_1\\ k_1\leq K}}
\Osum_{\substack{k_2\mid (a^2+b^2)/\la_2\\ k_2\leq K}}
\Osum_{\ell\in \NN} \mu(k_1)\mu(k_2)\mu(\ell)
\widetilde{L}_{k_1,k_2,\ell}(B),
\end{equation}
with acceptable error. 
Here $K$ is given by \eqref{eq:K} and
$\widetilde{L}_{k_1,k_2,\ell}(B)$
denotes the number of $(s,t)\in~\sfl(\k,\bla,\ell)$ such that
\eqref{i:2}, \eqref{i:-}  and \eqref{i:7} hold
in the definition of $L(B)$, with $2\nmid \gcd(s,t)$ and $\nu_2(s)\leq
2Z_2$.  Note that it will facilitate our ensuing investigation to restrict the
summation to odd values of $\ell$ here.

When it comes to estimating $\widetilde{L}_{k_1,k_2,\ell}(B)$ asymptotically as
$B\rightarrow \infty$ we will encounter problems
when $\ell$ is large.
For given $T>0$ and $Y\geq 1$, we now let $L^\dagger_{k_1,k_2,\ell}(Y,T)$ denote
the number of 
$$
(s,t)\in \sfl(\k,\bla,\ell)\cap
\R^\dagger(2^\nu \la_1\la_2 Y/\max\{a,b\})
$$
for which $\gcd(s,t)>T$, where
$\R^\dagger(X)$ is given by \eqref{eq:R-dag}.
Note that $\widetilde{L}_{k_1,k_2,\ell}(B) \leq 
L^\dagger_{k_1,k_2,\ell}(B,1/2)$ in the above notation. 
We have the following key result.

\begin{lemma}\label{lem:large-gcd}
Suppose $T>0$ and $A,Y\geq 1$, with $A\leq \sqrt{B}/Z_1^c$ and $Y\leq B^5.$ Then we have
$$
\sum_{\substack{
(a,b)\in \NN^2\\
\cp{a}{b}, ab\neq 1\\
\nu_2(a^2-b^2)\leq Z_2\\
\max\{a,b\}\leq A}}
\sum_{\nu\leq Z_2}
\Osum_{\substack{k_1\lambda_1\mid a^2-b^2\\ k_1\leq K}}
\Osum_{\substack{k_2\lambda_2\mid a^2+b^2\\ k_2\leq K}}
\Osum_{\ell \in \NN}
L^\dagger_{k_1,k_2, \ell}(Y;T)
\ll
A^2Z_1^{O(1)}\lfloor T \rfloor+ \frac{YZ_1^{O(1)}}{T}.
$$
\end{lemma}

\begin{proof}
The idea is to reintroduce
a coprimality condition on the $(s,t)$ that are to be counted.
We have 
\begin{align*}
L^\dagger_{k_1,k_2, \ell}(Y;T)
= \sum_{L>T}
\#\left\{
(s,t)\in \sfl(\k,\bla,\ell)\cap\R^\dagger\Big(\frac{2^\nu \la_1\la_2
  Y}{\max\{a,b\}}\Big) :
\gcd(s,t)=L
\right\}.
\end{align*}
Let
\begin{equation}
   \label{eq:ghost'}
\la_1':=k_1\la_1, \quad \la_2':=k_2\la_2
\end{equation}
and
\begin{equation}
   \label{eq:ghost}
\lambda_1'':=\frac{\la_1'}{\gcd(\la_1',L)}, \quad 
\lambda_2'':=\frac{\la_2'}{\gcd(\la_2',L)}.
\end{equation}
Note that the summand is zero unless $\ell \mid L$.
Making the change of variables
$(s,t)=L(s',t')$ with $\cp{s'}{t'}$,  we conclude that the
summand is bounded by the number of coprime vectors $(s',t')\in \ZZ^2$
such that
\begin{equation}
   \label{eq:lai'}
\la_1'' \mid s', \quad
\la_2'' \mid s'-t'a,
\end{equation}
with $(s',t')\in \R^\dagger(2^\nu \la_1\la_2 Y/(L^2\max\{a,b\}))$.

Recall the definition \eqref{eq:Qi} 
of $Q_1$ and $Q_2$ and let 
$\lambda^+:=\gcd(Q_1(s',t'),Q_2(s',t'))$.
Since $\cp{s'}{t'}$ it follows that
\begin{align*}
\lambda^+&=\gcd(4s'(s'-t'a),{t'}^2(a^2-b^2)-2{s'}^2)
\geq \gcd(s',a^2-b^2)\gcd(s'-t'a,a^2+b^2).
\end{align*}
In view of the fact that $\cp{\la_1''}{\la_2''}$ and \eqref{eq:lai'}
holds, we obtain $\lambda^+\geq \la_1\la_2/M$,
with $M:=\gcd(\la_1\la_2,L)$. In particular  $M\mid a^4-b^4$.
It now follows that
there exists an absolute constant $c>0$ such that
$$
0< -\frac{\max\{a,b\}Q_j(s',t')}{\la^+}\leq c
\frac{2^{Z_2 -1} M Y}{L^2}\leq
c \frac{ M Y\log B}{L^2},
$$
for $j=1,2$. Writing $L=ML'$, we deduce that
\begin{equation}
   \label{eq:new-height}
(s',t')\in \R^\dagger\left(\frac{\la^+ Y'}{\max\{a,b\}}\right),
\quad Y':=c \frac{  Y\log B}{M{L'}^2}.
\end{equation}
Since $\la^+\mid Q_j(s',t')$ we note that 
$$
\max\{a,b\}\leq \max\{a,b\} \cdot \frac{|Q_j(s',t')|}{\la^+}
\leq 
c \frac{ M Y\log B}{L^2}= cY'.
$$

Let us write $E(Y)$ for the term that is to be estimated in the statement of
the lemma. We will employ the estimate $\tau(n)\leq
Z_1^{O(1)}$, valid for any non-zero $n\leq B^c$.
Hence in $E(Y)$ there are at most
$\tau_3(|a^2-b^2|)\tau_3(a^2+b^2)\tau(|a^4-b^4|)\leq Z_1^{O(1)}$  possible values of
$k_1,k_2, \la_1, \la_2$ and $M$. Furthermore there are at most
$\tau(L) \leq Z_1^{O(1)}$   values of
$\ell$ dividing $L$.
On noting that the summation over $\nu$ contributes~$O(Z_1)$, we
conclude that 
$$
E(Y)
\ll Z_1^{O(1)}
\sum_{\substack{a,b<A\\
\gcd(a,b)=1\\ ab\neq 1}}
\max_{M\ll A^4}
  \sum_{\substack{L'> T/M}}
\#\left\{
(s',t')\in \ZZ^2: \begin{array}{l}
\cp{s'}{t'},\\
\mbox{\eqref{eq:new-height} holds}
\end{array}
\right\}.
$$
We would now like to compare the inner cardinality with 
$\widehat{M}_{a,b}(Y')$, 
as given by \eqref{eq:hatM}.  But this follows immediately from 
the bijection described in the
proof of Lemma~\ref{lem:para-1} and the observation that 
$s'(s'-t'a)\neq 0$ when 
$(s',t')\in \R^\dagger(\la^+ Y'/\max\{a,b\})$, so that 
$xy\neq 1$ as required for 
 $\widehat{M}_{a,b}(Y')$.
In this way we may approximate the summand by 
$\widehat{M}_{a,b}(Y')+O(1)\ll \widehat{M}_{a,b}(Y')$, since
$\widehat{M}_{a,b}(Y')\geq 1$ for $Y'\gg \max\{a,b\}$,
whence
\begin{equation}
\label{eq:2lines}
E(Y)
\ll Z_1^{O(1)}
\sum_{\substack{a,b<A\\
\gcd(a,b)=1\\ ab\neq 1}}
\max_{M\ll A^4}
  \sum_{\substack{L'> T/M}}
\widehat{M}_{a,b}(Y').
\end{equation}

We write $E_1(Y)$ (resp. $E_2(Y)$)  for the overall
contribution to the right hand side from $a,b,L'$ such that
$\max\{a,b\}\leq \sqrt{Y'}$ (resp. $\max\{a,b\}> \sqrt{Y'}$).
To begin with it follows from Lemma \ref{lem:ab-small} 
that
$$
E_1(Y) \ll Z_1^{O(1)}
\max_{M}\sum_{L'>T/M} Y'(\log B)^4 \ll \frac{YZ_1^{O(1)}}{T}.
$$
It remains to estimate $E_2(Y)$. We return to \eqref{eq:2lines}, now
with $\sqrt{Y'}<\max\{a,b\}<A$ and $ab\neq 1$ in the
summation over $a,b$. 
We will write $E_{2,1}(Y)$ for the contribution to the right hand side
from values of $L'\leq T$, and
$E_{2,2}(Y)$ for the contribution from values of $L'> T$.
Beginning with the former,
we note that $E_{2,1}(Y)=0$ if $T<1$. If on the other hand $T\geq 1$ then 
an application of  \eqref{majMab}
gives
\begin{align*}
E_{2,1}(Y)
&\ll 
Z_1^{O(1)} \max_{M} \sum_{T/M<L'\leq T} \sum_{a,b<A}\left(1+ \frac{Y'}{\max\{a,b\}^{5/3}|a-b|^{1/3}} \right)\\
&\ll
A^2Z_1^{O(1)}T+ \frac{YZ_1^{O(1)}}{T}
\end{align*}
This is satisfactory for the lemma.

Finally we must deal with
$E_{2,2}(Y)$.  For this we will reverse the roles of the variables
$a,b$ and $x,y$ in $\sum_{a,b}\widehat{M}_{a,b}(Y')$.  We have
\begin{align*}
E_{2,2}(Y)&\ll Z_1^{O(1)}
\max_M \sum_{L'>T}\#
\left\{
a,b,x,y,z : \begin{array}{l}
\gcd(a,b)=\gcd(x,y)=1, \\
ab, xy \neq 1,\\
\sqrt{Y'}<\max\{a,b\}<A,\\
(a^2-b^2)x^2+(a^2+b^2)y^2=2z^2,\\
\max\{a,b\}\max\{x,y\}\leq Y'
\end{array}
\right\},
\end{align*}
On writing
$Y'':=cY(\log B)/{L'}^2$, it now follows from Lemma \ref{lem:ab-small}  
that
\begin{align*}
E_{2,2}(Y)\ll Z_1^{O(1)}\max_M
\sum_{L'>T}
\sum_{\substack{x,y\leq \sqrt{Y''}\\\cp{x}{y}\\xy\neq 1}}
M_{y,x}(Y'') 
\ll 
YZ_1^{O(1)} 
\sum_{L'>T}\frac{1}{{L'}^2}
\ll \frac{YZ_1^{O(1)}}{T}.
\end{align*}
This is satisfactory and therefore completes the  proof of the lemma.
\end{proof}

We would now like to eliminate the contribution to
\eqref{eq:Lk1k2l} from large $\ell$.
Taking 
$
Y=B, A=Z_1^{-c}\sqrt{B}$ and $T=Z_1^{c/10}$ 
in Lemma \ref{lem:large-gcd} we deduce that the
overall contribution from $\ell>Z_1^{c/10}$ is
$
\ll B Z_1^{O(1)}/Z_1^{c/10},
$
which is satisfactory when $c$ is taken to be sufficiently large.
Drawing this observation
together with \eqref{eq:ruff}, our
argument so far has established the following  result.

\begin{lemma}\label{lem:small_kl}
We have
\begin{multline*}
N_1(B)=2
\sum_{(a,b)\in \A_1} \sum_{\nu\leq \max\{1,\nu_2(a^2-b^2)\}}
\Osum_{\substack{k_1\lambda_1\mid a^2-b^2\\k_1\leq
   K}}
\Osum_{\substack{k_2\lambda_2\mid a^2+b^2\\ k_2\leq K}}\\
\times \Osum_{\ell \leq  Z_1^{c/10}}
\mu(k_1)\mu(k_2)\mu(\ell)
\widetilde{L}_{k_1,k_2, \ell}(B)
+O\Big(\frac{B(\log B)^4}{\log\log B}\Big),
\end{multline*}
where $\A_1$ is given by \eqref{eq:AA} and $K$ by \eqref{eq:K}.
\end{lemma}

\section{Lattice point counting}\label{s:asy}

Our task in this section is to set the scene for an asymptotic formula for
the quantity $\widetilde{L}_{k_1,k_2, \ell}(B)$.
Let $\sfl=\sfl(\k,\bla,\ell)\subset \ZZ^2$ be the 
lattice defined in \eqref{eq:lattice-1}.
On writing
\begin{equation}
   \label{eq:Rnu}
\R_\nu':=\R'\Big(\frac{2^\nu \la_1\la_2 B}{\max\{a,b\}}\Big),
\end{equation}
where $\R'(X)$ is defined by \eqref{eq:RBA'}, it follows from
the previous section that
\begin{equation}\label{eq:Tuesday}
\widetilde{L}_{k_1,k_2,\ell}(B)=\#\left\{
(s,t)\in \sfl\cap \R_\nu':
\mbox{\eqref{eq:nu} holds},~ 2\nmid \gcd(s,t),~\nu_2(s)\leq 2Z_2
\right\}.
\end{equation}

We now come to a rather delicate feature of the proof.
Ignoring the supplementary $2$-adic conditions in our new expression
for $\widetilde{L}_{k_1,k_2,\ell}(B)$,
the obvious next step would be to try and approximate this cardinality
with something like the volume of the region in question, divided by the
determinant of $\sfl$.
Ignoring also the contribution from the error term in
this approximation, one would then be led to sum this quantity over all of the
remaining parameters.   An absolutely crucial observation here
is the following: for some ranges of the parameters $\la_1,\la_2$
we will have $\widetilde{L}_{k_1,k_2, \ell}(B)=0$ in Lemma
\ref{lem:small_kl}, even if the approximation
$\vol(\R_\nu')/\det \sfl$  is non-zero. Thus a further
reduction on the set of allowable parameters $\la_1,\la_2$
is necessary.

For any $R>0$ and $(a,b)\in \A_1$, define the
region
\begin{equation}
   \label{eq:VV}
   V_{a,b}(R):=\left\{
(t_1,t_2)\in \RR_{\geq 1}^2:
\begin{array}{l}
\max\{a,b\}t_1\leq Rt_2, \\
\max\{a,b\}t_2\leq Rt_1, \\
\max\{a,b\}^3/R\leq t_1t_2,\\
t_1t_2 \leq \max\{a,b\}R
\end{array}
\right\}.
\end{equation}
We will show how the summation over $\la_1,\la_2$ is necessarily
restricted to this set for a suitable choice of $R$.

It follows from the definition \eqref{eq:AA} 
of $\A_1$ that $\min\{a,b,|a-b|\}\geq
\max\{a,b\}/Z_2^2$. Taken together with Lemma \ref{lem:volume},
the fact that we are interested in integers $(s,t)\in \sfl\cap \R_\nu'$
therefore implies that
$$
1\leq |s| \leq c\sqrt{X}, \quad 1\leq t
\leq c\frac{\sqrt{X}Z_2}{\max\{a,b\}}, \quad
1\leq |s-at|\leq c\sqrt{X}Z_2,
$$
with $X$ given by \eqref{eq:XX}. The middle inequality here implies that
$$
\max\{a,b\}^3 \leq c 2^\nu BZ_2^2\la_1\la_2,
$$
in which we recall our convention that $c$ is used to denote a generic
absolute positive constant.
But we also know that $\la_1\mid s$ and $\la_2\mid s-at$, so that we
may take $|s|\geq \la_1$ and $|s-at|\geq \la_2$ in the first and third
inequality, giving
$$
\max\{a,b\}\la_1 \leq c 2^\nu B\la_2, \quad
\max\{a,b\}\la_2 \leq c 2^\nu BZ_2^2\la_1.
$$

The final inequality that we seek to establish
is
\begin{equation}
   \label{eq:germ}
\la_1\la_2 \leq c \max\{a,b\} 2^\nu BZ_2^6.
\end{equation}
This is rather more subtle and arises
from an inherent symmetry between $\la_1,\la_2$ and~$\mu_1,\mu_2$,
where the latter are non-zero integers such that
\begin{equation}
   \label{eq:chick}
\la_1\mu_1=a^2-b^2, \quad \la_2\mu_2=a^2+b^2.
\end{equation}
In particular all of $\la_1,\la_2,\mu_1,\mu_2$ are coprime to $a$ and
$b$.  It follows
from the divisibility information on $s$ and $s-at$
that there exist non-zero integers $\sigma, \tau$ such that
$s=\la_1 \sigma$ and $s-at=\la_2\tau$.
The height conditions for $s, s-at$
imply that
\begin{align*}
1&\leq \sigma^2 \leq c\frac{2^\nu B\la_2}{\max\{a,b\}\la_1}
=c\frac{2^\nu
   B(a^2+b^2)|\mu_1|}{\max\{a,b\}|a^2-b^2|\mu_2}
\leq c\frac{2^\nu
   BZ_2^2|\mu_1|}{\max\{a,b\}\mu_2},\\
1&\leq \tau^2 \leq c\frac{2^\nu BZ_2^2\la_1}{\max\{a,b\}\la_2}=c\frac{2^\nu
   BZ_2^2|a^2-b^2|\mu_2}{\max\{a,b\}(a^2+b^2)|\mu_1|}
\leq c
\frac{2^\nu
   BZ_2^2\mu_2}{\max\{a,b\}|\mu_1|}.
\end{align*}
Furthermore, we have
\begin{equation}
   \label{eq:spain}
\mu_1\mu_2at=\mu_1\mu_2(\la_1\sigma-\la_2\tau)
=\mu_2(a^2-b^2)\sigma-\mu_1(a^2+b^2)\tau.
\end{equation}
In particular 
\begin{equation}
   \label{eq:trans_cong}
   \mu_2\sigma+\mu_1\tau\equiv 0 \pmod{a},
\end{equation}
since $a,b$ are coprime.
Suppose that 
$\mu_2\sigma+\mu_1\tau=0$. Then since $\gcd(\mu_1,\mu_2)=2^i$
for $i\in \{0,1\}$, it follows that 
$\sigma=2^{-i}\mu_1x$ and $\tau=-2^{-i}\mu_2 x$ for a non-zero
integer $x$. Using \eqref{eq:spain} we easily deduce that
$$
\frac{s}{t}=\frac{a^2-b^2}{2a},
$$
which contradicts the hypotheses in the definition of
\eqref{eq:RBA'}.  
Hence $\mu_2\sigma+\mu_1\tau\neq 0$ in~\eqref{eq:trans_cong}, giving the further
inequality
$$
\frac{\max\{a,b\}}{Z_2^2}\leq a\leq |\mu_2\sigma|+|\mu_1\tau|
\leq c \sqrt{\frac{2^\nu
   BZ_2^2|\mu_1|\mu_2}{\max\{a,b\}}},
$$
whence
$$\frac{\max\{a,b\}^3}{Z_2^6}\leq c 2^\nu B |\mu_1|\mu_2=
c 2^\nu B \frac{|a^4-b^4|}{\la_1\la_2}\leq
c 2^\nu B \frac{\max\{a,b\}^4}{\la_1\la_2}.
$$
This therefore implies \eqref{eq:germ}.

Bringing our argument together we have therefore shown that
the summation over $\la_1,\la_2$ in Lemma  \ref{lem:small_kl} is
subject to
$
(\la_1,\la_2)\in V_{a,b}(c 2^\nu BZ_2^6).
$
In fact, in view of the bounds $k_1,k_2\leq K$, it is trivial to
see that
$$
(k_1\la_1,k_2\la_2)\in V_{a,b}(c BK^3).
$$
Here, using the definitions of $K$ and $Z_2$, we have been able
to replace $2^\nu Z_2^6$ by an additional factor of $K$.
The following result refines this  somewhat.

\begin{lemma}\label{lem:final_prelim}
We have
\begin{multline*}
N_1(B)=2
\sum_{(a,b)\in \A_1} \sum_{\nu\leq \max\{1,\nu_2(a^2-b^2)\}}
\Osum_{\substack{k_1\lambda_1\mid a^2-b^2\\k_1\leq
   K}}
\Osum_{\substack{k_2\lambda_2\mid a^2+b^2\\
k_2\leq K}}\chi(k_1\la_1,k_2\la_2; B/K)\\
\times \Osum_{\ell \leq  Z_1^{c/10}}
\mu(k_1)\mu(k_2)\mu(\ell)
\widetilde{L}_{k_1,k_2, \ell}(B)
+O\Big(\frac{B(\log B)^4}{\log\log B}\Big),
\end{multline*}
where $\A_1$ is given by \eqref{eq:AA}, $K$ by
\eqref{eq:K} and if $V_{a,b}(R)$ is given by \eqref{eq:VV} then 
\begin{equation}
  \label{eq:chi}
\chi(t_1,t_2;R)=
\chi_{a,b}(t_1,t_2;R):=
\begin{cases}
1, & \mbox{if $(t_1,t_2)\in V_{a,b}(R)$,}\\
0, & \mbox{otherwise}.
\end{cases}
\end{equation}
\end{lemma}

\begin{proof}
We have already shown that the
lemma is true with $\chi(k_1\la_1,k_2\la_2;cBK^3)$
in place of $\chi(k_1\la_1,k_2\la_2;B/K)$.
To establish the lemma it suffices to estimate the overall
contribution to the main term from values of $k_1,k_2,\la_1,\la_2$ for
which
$$
(k_1\la_1,k_2\la_2)\in
V_{a,b}(cBK^3)\setminus V_{a,b}(B/K),
$$
in the notation of \eqref{eq:VV}.  This forces $\la_1,\la_2$ to satisfy one
of four further inequalities. Let us show how to handle the
contribution corresponding to $\la_1,\la_2$
satisfying
\begin{equation}
   \label{eq:hunt}
\frac{Bk_1\la_1 }{k_2 K\max\{a,b\}}
<\la_2\leq \frac{cBK^3k_1\la_1}{k_2\max\{a,b\}},
\end{equation}
the remaining cases being handled in an identical manner.

Let $L_1(B)$ denote the contribution to
$\widetilde{L}_{k_1,k_2, \ell}(B)$ arising from
$(s,t)\in \ZZ^2$ for which $\gcd(s,t)\leq Z_1^{c/10}$.
Lemma~\ref{lem:large-gcd} implies that
the overall contribution to Lemma \ref{lem:final_prelim} from 
the remaining quantity is $O(B Z_1^{O(1)}/Z_1^{c/10})$, which is
satisfactory for large enough $c$.  This mimics the argument used to 
establish Lemma \ref{lem:small_kl}, where we also needed to reduce the
allowable size of $\gcd(s,t)$.

To estimate $L_1(B)$, we
write $(s,t)=L(s',t')$ with $\gcd(s',t')=1$ and 
$L\leq Z_1^{c/10}$ an odd integer
divisible by $\ell$.
The summation over $\nu$ allows us to assume that 
$2^{\nu}\mid 4s'$.  Let us write
$\la_i', \la_i''$ as in \eqref{eq:ghost'} and \eqref{eq:ghost}.
We conclude from Lemmas~\ref{lem:lattice} and \ref{lem:volume} that
$$
L_1(B)\ll \Osum_{\substack{L\leq Z_1^{c/10}\\ \ell \mid L}}
\Big(1+\frac{2^\nu \la_1\la_2 BZ_2}{L^2(\det \sfl')\max\{a,b\}^2}\Big),
$$
where now $\sfl'\subseteq \ZZ^2$ is the lattice of $(s',t')\in \ZZ^2$
for which
$$
\la_1'' \mid s', \quad
\la_2'' \mid s'-at',\quad 2^\nu\mid 4s'.
$$
Lemma \ref{lem:det} therefore leads us to an overall contribution of
\begin{align*}
&\ll
\sum_{(a,b)\in \A_1} \sum_{\nu\leq Z_2}
\Osum_{k_i,\la_i}
\Osum_{L\leq Z_1^{c/10}}
\sum_{\ell \mid 
L}\Big(1+\frac{\gcd(k_1k_2\la_1\la_2,L) 
BZ_2}{L^2k_1k_2\max\{a,b\}^2}\Big).
\end{align*}
Here the summation over
$k_i,\la_i$ is subject to the restriction that they should be odd,
with $k_i\lambda_i\mid a^2+(-1)^ib^2$, $k_i\leq K$ and
\eqref{eq:hunt} holding. In particular $k_1\la_1$ and
$k_2\la_2$ are coprime, so that 
$\gcd(k_1\lambda_1,\ell) \gcd(k_2\lambda_2,\ell)=
\gcd(k_1k_2\lambda_1\la_2,\ell)$.

For fixed $a,b,L$ there are 
$\leq \tau_3(|a^4-b^4|)\tau(L)Z_2\leq Z_1^{O(1)}$ values of~$\nu,k_i,\la_i,\ell$. Hence
the contribution from the first term in the above summand is
$O(BZ_1^{O(1)}/Z_1^{c})$, which is satisfactory for large $c$.  
The remaining contribution is clearly 
\begin{align*}
&\ll B(\log\log B)^2
\sum_{(a,b)\in \A_1}
\Osum_{k_i,\la_i}
\Osum_{L\leq Z_1^{c/10}}
\frac{\tau(L)\gcd(k_1k_2\la_1\la_2,L) }{L^2k_1k_2\max\{a,b\}^2},
\end{align*}
on carrying out the summation over $\nu$.
But for any $h\in \NN$ we have 
\begin{equation}\label{eq:exp}
\begin{split}
\sum_{n\leq x}\frac{\tau(n)\gcd(h,n) }{n^2}
\leq \sum_{d\mid h}
\sum_{\substack{n\leq x\\ d\mid n}}\frac{\tau(n)d}{n^2}
&\leq \sum_{d\mid h} \frac{\tau(d)}{d}
\sum_{m\leq x/d}\frac{\tau(m)}{m^2} \\
&\ll
\Big(\frac{h}{\varphi(h)}\Big)^2 \\
&\ll (\log\log h)^2.
\end{split}
\end{equation}
Hence we obtain the overall contribution
\begin{equation}\label{eq:la-r}
\ll B(\log\log B)^4
\sum_{(a,b)\in \A_1}\Osum_{k_i,\la_i}
\frac{1}{k_1k_2\max\{a,b\}^2}=E(B),
\end{equation}
say. 
We will find it convenient to proceed under the
additional hypothesis
\begin{equation}
   \label{eq:hunt'}
   \la_1\la_2\leq \sqrt{2}\max\{a,b\}^2.
\end{equation}
Underlying this assumption is
the symmetry that exists between $\la_i$ and $\mu_i$
such that~\eqref{eq:chick} holds.
Thus if \eqref{eq:hunt'} fails then it is immediately clear from the
definition~\eqref{eq:AA} of $\A_1$ that
$$
|\mu_1|\mu_2\leq \sqrt{2}\max\{a,b\}^2.
$$
Furthermore, \eqref{eq:hunt} translates into
$$
\frac{\max\{a,b\}(a^2+b^2)k_2|\mu_1|}{cBk_1K^3|a^2-b^2|}
\leq \mu_2<
\frac{\max\{a,b\}(a^2+b^2)k_2 K |\mu_1|}{Bk_1|a^2-b^2|}.
$$
This shares the same basic structure as \eqref{eq:hunt}, giving upper
and lower bounds for $\mu_2$ with ratio of order $K^4$. 
This discussion shows that we may proceed under the assumption that
\eqref{eq:hunt'} holds in our estimation of $E(B).$

We now break the summation over $a,b$ into dyadic intervals, in order
to fix attention on the range $A\leq
\max\{a,b\}< 2A$ for
$
A\leq \sqrt{B}/Z_1^c.
$
Interchanging the order of summation \eqref{eq:la-r} becomes
\begin{align*}
E(B)&\ll B(\log\log B)^4
\sum_A \frac{1}{A^2}  \Osum_{k_i,\la_i}
\frac{1}{k_1k_2}N(A),
\end{align*}
where
$$
N(A):=
\#\left\{
(a,b)\in \A_1:
A\leq \max\{a,b\}< 2A, ~k_i\la_i\mid a^2+(-1)^ib^2
\right\}.
$$
Furthermore, the summation over $k_i,\la_i\in \NN$ is over odd
integers subject to $k_i\leq K$ and
\begin{equation}
   \label{eq:ak}
\frac{B\la_1}{2 AK^2}
\leq \frac{Bk_1\la_1}{2k_2 AK}
<\la_2\leq \frac{cBK^3k_1\la_1}{k_2 A}
\leq \frac{cBK^4\la_1}{ A}.
\end{equation}

Let $\rho_i(q)$ be the number of square roots of $(-1)^{i+1}$ modulo $q$.
For given $q_1,q_2$ the conditions $q_1\mid 
a^2-b^2$ and $q_2 \mid a^2+b^2$ force the
vector $(a,b)$ to lie on one of at most~$\rho_1(q_1)\rho_2(q_2)$ integer
sublattices of $\ZZ^2$, each of determinant $q_1q_2$.
Now it is easy to see that
$\rho_1(q_1)\leq 2^{\omega(q_1)}$ and
$$
\rho_2(q_2)\leq
\sum_{e\mid q_2} |\mu(e)|\chi(e)\leq
r(q_2),
$$
where $\chi$ is the real non-principal character modulo $4$ and $r$
denotes the sums of two squares function. It follows
that $\rho_2(rq)\leq 2^{\omega(r)}r(q)$ for
any $r,q\in \NN$.
Since $a,b$ lie in an ellipse in $\RR^2$ with area $O(A^2)$, an %T%
application 
of Lemma \ref{lem:lattice}  furnishes the estimate
\begin{align*}
E(B) &\ll B(\log\log B)^4
\sum_A \frac{1}{A^2} \Osum_{k_i,\la_i}
\frac{2^{\omega(k_1k_2)}}{k_1k_2}
2^{\omega(\la_1)}r(\la_2)\Big(\frac{A^2}{k_1k_2\la_1\la_2}
+1\Big)\\
&\ll B(\log\log B)^4(\log K)^{4}
\sum_A \frac{1}{A^2} \Osum_{\la_1,\la_2}
2^{\omega(\la_1)}r(\la_2)\Big(\frac{A^2}{\la_1\la_2} +1\Big),
\end{align*}
on summing over $k_1,k_2\leq K$.
Here the inner summation is over
$\la_1,\la_2$ such that
$$
\frac{B\la_1}{2AK^2}
<\la_2\leq \frac{cBK^4\la_1}{A}, \quad \la_1\la_2\ll A^2,
$$
as follows from  \eqref{eq:hunt'} and \eqref{eq:ak}.
Thus we conclude that 
\begin{align*}
E(B) &\ll B(\log\log B)^{8}
\sum_A
\Osum_{\la_1,\la_2}
\frac{2^{\omega(\la_1)}r(\la_2)}{\la_1\la_2}\\
&\ll B(\log\log B)^{8}
\sum_A
\Osum_{\la_1}
\frac{2^{\omega(\la_1)}}{\la_1} \log (2cK^6)\\
&\ll B(\log B)^3(\log\log B)^{9},
\end{align*}
on summing first over $\la_2$, then over $\la_1$ and finally over the
$O(\log B)$ choices for $A$. This completes the proof of the lemma.
\end{proof}

Let $a,b,\nu,k_i,\la_i, \ell$ be an arbitrary choice of parameters
that appear in the main term in Lemma 
\ref{lem:final_prelim}'s estimate for $N_1(B)$.
We now come to our estimation of \eqref{eq:Tuesday}, where
$\sfl=\sfl(\k,\bla,\ell)\subset \ZZ^2$ is the lattice defined in
\eqref{eq:lattice-1} and
$\R_\nu'\subset \RR^2$ is given by~\eqref{eq:Rnu}.
Given that $\R_\nu'$ is clearly defined with piecewise continuous boundary,
we would like to apply the well-known formula \eqref{eq:classic}.
However, there are two complications that prevent a routine application
of this estimate. Firstly, we will need to take care of the 
$2$-adic conditions
apparent in~\eqref{eq:Tuesday}.  Secondly we need to deal with the fact that once summed over the
remaining parameters, the error term in
the above asymptotic formula for $\#(\sfl\cap \R_\nu')$ does not make a
satisfactory overall contribution from the point of view of the main
theorem.  

In spite of these objections we will dedicate the remainder of this
section to interpreting the main term in \eqref{eq:classic} in the
present context, delaying our discussion of the error term until the
subsequent section.
Let us consider the first issue mentioned above, namely the $2$-adic
conditions  on $s,t$. We retain the shorthand 
notation $\sfl=\sfl(\k,\bla,\ell)$
and make the observation that
$\vol(\R_\nu')=2^\nu\vol(\R_0')$, for any $\nu\geq 0.$
Furthermore, we will set
$$
\sfl_{i,j}:=\{(s,t)\in\sfl: 2^i\mid s, ~2^j\mid t\},
$$
for $i,j\geq 0$.
Since the parameters in $\sfl$ are all odd it easily follows
from Lemma~\ref{lem:det}
that $\det \sfl_{i,j}=2^{i+j}\det \sfl$.
It  is here that our earlier restriction to odd values of $\ell$ pays
dividends. Finally, it will be convenient to define
\begin{equation}
   \label{eq:pi}
   \pi_k:=
\begin{cases}
1, &\mbox{if $2\mid k$,}\\
0, &\mbox{if $2\nmid k$},
\end{cases}
\end{equation}
for any $k\in \NN$.

We will separate
our investigation according to the value of $\nu$.
We suppress the dependence on $k_1,k_2,\ell$ in
the expression for $\widetilde{L}_{k_1,k_2,\ell}(B)$ in
\eqref{eq:Tuesday}, replacing it by~$L^\nu(B)$ in order to
underline the dependence on $\nu$.
We have $3$ basic possibilities to consider: either  $\nu=0$ or $\nu=1$ or
$\nu \geq 3$. Note that \eqref{eq:nu}
ensures that the possibility~$\nu=2$ does not arise.

Let us begin by supposing that $\nu=0$, which according
to \eqref{eq:nu} is only possible when $2\mid ab$ and $2\nmid t$, so
that $\pi_{ab}=1$.
Recall the notation
introduced above for $\sfl_{i,j}$. It follows that
\begin{align*}
L^0(B)
&=\sum_{0\leq \sigma\leq 2Z_2}
\#\left\{
(s,t)\in \sfl\cap \R_0':
2^\sigma\| s, 2\nmid t
\right\}\\
&=\sum_{0\leq \sigma\leq 2Z_2}\sum_{i\geq 
0}\sum_{j\geq 
0}\mu(2^i)\mu(2^j)\#(\sfl_{\sigma+i,j}\cap \R_0').
\end{align*}
In line with \eqref{eq:classic} we expect the cardinality in the
summand to satisfy an asymptotic formula with main term
\begin{align*}
\frac{\vol(\R_0')}{\det \sfl_{\sigma+i,j}}=
\frac{\vol(\R_0')}{2^{\sigma+i+j} \det \sfl}.
\end{align*}
We may now conclude as follows.

\begin{lemma}\label{lem:L0}
We have
$$
L^0(B)=\pi_{ab}\Big(\frac{1}{2}-\frac{1}{2^{\lfloor 2Z_2
    \rfloor}}\Big)\frac{\vol(\R_0')}{  
\det \sfl} + E^0(B),
$$
where
$$
E^0(B):=\pi_{ab}
\sum_{0\leq \sigma\leq 2Z_2}\sum_{i,j\geq 0}\mu(2^i)\mu(2^j)
\Big(\#(\sfl_{\sigma+i,j}\cap
   \R_0')-
\frac{\vol(\R_0')}{\det \sfl_{\sigma+i,j}}
\Big).
$$
\end{lemma}

We have found it useful to include $\pi_{ab}$ in the main
term for this estimate, to help keep track of the fact that
we are only interested in the value of $L^0(B)$ when $2\mid ab$.
When $\nu=1$ it follows from
\eqref{eq:nu} that
$2\mid t$ if $2\mid ab$ and
$2\nmid s$ if $2\nmid ab$.
Hence \eqref{eq:Tuesday} yields
$$
L^1(B)=
\pi_{ab}\#\left\{
(s,t)\in \sfl\cap \R_1':
2\nmid s, ~2\mid t\right\}
+\pi_{ab+1}\#\left\{
(s,t)\in \sfl\cap \R_1':
2\nmid s\right\}.
$$
Arguing as above, we now have
\begin{align*}
L^1(B)=
\pi_{ab}\sum_{i \geq 0}\mu(2^i)
\#(\sfl_{i,1} \cap \R_1')
+
\pi_{ab+1}
\sum_{i\geq 0}\mu(2^i)\#(\sfl_{i,0} \cap \R_1').
\end{align*}
Drawing out the obvious main term, as previously,
we therefore conclude the proof of the following result.

\begin{lemma}\label{lem:L1}
We have
$$
L^1(B)=\left(1-\frac{\pi_{ab}}{2}\right)\frac{\vol(\R_0')}{ \det \sfl}
  + E_\al^1(B)+E_\be^1(B),
$$
where
\begin{align*}
E_\al^1(B)&:=
\pi_{ab}\sum_{i \geq 0}\mu(2^i)
\Big(\#(\sfl_{i,1}\cap \R_1') -
\frac{\vol(\R_1')}{\det \sfl_{i,1}}\Big), \\
E_\be^1(B)&:=\pi_{ab+1} \sum_{i\geq 0}\mu(2^i)\Big(\#(\sfl_{i,0}\cap
   \R_1')-
\frac{\vol(\R_1')}{\det \sfl_{i,0} }
\Big).
\end{align*}
\end{lemma}

Assume now that $\nu\geq 3$, under which assumption
\eqref{eq:nu} implies that $2\nmid ab$ and~$2\mid s$,
with $
\nu=\min\{2+\nu_2(s), \nu_2(a^2-b^2)\}.
$
In particular, if $\nu_2(s)\leq \nu_2(a^2-b^2)-2$ then it follows that
$\nu_2(s)=\nu-2$. This is then automatically bounded above by~$2Z_2$.
Otherwise we must have
$\nu_2(s)\geq  \nu_2(a^2-b^2)-1$  and $\nu_2(a^2-b^2)=\nu$.
Moreover, we must force $2\nmid t$ in our considerations.
We deduce from \eqref{eq:Tuesday} that $L^\nu(B)$ is the sum of
$$
T_1:=\pi_{ab+1}\sum_{i,j\geq 0}\mu(2^i)\mu(2^j)
\#(\sfl_{\nu-2+i,j}\cap \R_\nu')
$$
and
$$
T_2:=\pi_{ab+1}\sum_{\nu-1\leq \sigma\leq 2Z_2}\sum_{i,j\geq 0}\mu(2^i)\mu(2^j)
\#(\sfl_{\sigma+i,j}\cap \R_\nu').
$$
Note the first term is always present, since $\nu\leq
\max\{1,\nu_2(a^2-b^2)\}$,
but the latter term only appears if $\nu=\nu_2(a^2-b^2)$.
The following result is now available.

\begin{lemma}\label{lem:L2}
Let $\nu \geq 3$. Then we have
$$
L^{\nu}(B)=
\pi_{ab+1}\frac{\vol(\R_0')}{\det \sfl}
  + E_\al^{\nu}(B),
$$
if $\nu<\nu_2(a^2-b^2)$ and
$$
L^{\nu}(B)=
\pi_{ab+1}\Big(2-\frac{2^{\nu}}{2^{\lfloor 2Z_2 \rfloor+2}}\Big)\frac{\vol(\R_0')}{\det \sfl}
  + E_\al^{\nu}(B)+E_\be^{\nu}(B),
$$
if $\nu=\nu_2(a^2-b^2)$, where
\begin{align*}
E_\al^\nu(B)&:= \pi_{ab+1}
\sum_{i,j\geq 0}\mu(2^i)\mu(2^j)
\Big(
\#(\sfl_{\nu-2+i,j}\cap \R_\nu')
-\frac{\vol(\R_\nu')}{\det
   \sfl_{\nu-2+i,j}}\Big),\\
E_\be^\nu(B)&:= \pi_{ab+1}
\sum_{\nu-1\leq \sigma\leq 2Z_2}\sum_{i,j\geq 0}\mu(2^i)\mu(2^j)
\Big(\#(\sfl_{\sigma+i,j}\cap
\R_\nu')
-\frac{\vol(\R_\nu')}{\det \sfl_{\sigma+i,j}}\Big).
\end{align*}
\end{lemma}

\section{The error terms}\label{s:2}

It is now time to establish that the error terms in Lemmas
\ref{lem:L0}, \ref{lem:L1} and \ref{lem:L2} make a satisfactory
overall contribution once summed up over the relevant
parameters appearing in Lemma \ref{lem:final_prelim}.
In order not to be encumbered with superfluous 
notation let us fix our attention on the error 
term
$$
   E(B):=\#(\sfl\cap \R'(X))-
\frac{\vol(\R'(X))}{\det \sfl},
$$
where $X$ is given by \eqref{eq:XX}. 
The outcome of our investigation will be a proof of the
following result.

\begin{lemma}\label{lem:exp}
Let
\begin{align*}
\EE(B):=~&
\sum_{(a,b)\in \A_1} \sum_{\nu}
\Osum_{\substack{k_1\lambda_1\mid a^2-b^2\\k_1\leq K}}
\Osum_{\substack{k_2\lambda_2\mid a^2+b^2\\k_2\leq K}}\\
&\times \Osum_{\ell \leq  Z_1^{c/10}}
|\mu(k_1)\mu(k_2)\mu(\ell)|
\chi(k_1\la_1,k_2\la_2;B/K)
|E(B)|,
\end{align*}
where $K$ is given by \eqref{eq:K} and the $\nu$ summation is over
$\nu \leq \max\{1,\nu_2(a^2-b^2)\}$. Then we have
$$
\EE(B)\ll B\frac{Z_1^{O(1)}}{Z_1^{c}}+B(\log B)^3(\log\log B)^2.
$$
\end{lemma}

There is a degree of over simplification here. Indeed,
for $\nu \geq 0$  we are actually
interested in controlling the overall contribution from the error
terms~$
E^0(B), E_\al^\nu(B)$ and $E_\be^\nu(B)$
that appear in Lemmas \ref{lem:L0}--\ref{lem:L2}.
However 
the argument goes through for each of these, since 
the summation over $i,j,\sigma$ in the true error terms
is over $i,j\leq 1$ and
$\sigma\leq 2Z_2$. Hence at the expense of a harmless extra factor of
$\log\log B$ in Lemma~\ref{lem:exp} the overall
contribution is seen to be satisfactory.

The proof of Lemma \ref{lem:exp} is long and technical. 
We begin with some general facts about approximating the
characteristic function of suitable
bounded regions $\mathcal{S}\subset \RR^2$ by smooth Gaussian
weights. For any $\x\in \RR^2$ let
$$
B(\x,r):=\{\y\in \RR^2: \|\y-\x\|\leq r\}
$$ 
denote the ball centered on
$\x$ with radius $r$, where
$\|\mathbf{z}\|=\sqrt{z_1^2+z_2^2}$ denotes the Euclidean norm of a
vector $\mathbf{z}\in \RR^2$. 
Let $H\geq 1$ be a parameter at our disposal.
Define
$$
S_-:=\RR^2\setminus \bigcup_{\x\not \in \mcal{S}} B\Big(\x,\frac{1}{\sqrt{H}}\Big),
\quad
S_+:=\bigcup_{\x\in \mcal{S}} B\Big(\x,\frac{1}{\sqrt{H}}\Big),
$$
and notice that 
$
S_-\subseteq \mathcal{S}\subseteq S_+.
$
We introduce the infinitely differentiable weight functions
$$
w_{\pm}(\x):=\frac{H^2}{\pi} \int_{S_{\pm}} \exp(-\|\y-\x\|^2H^2) \d\y
$$
and set
\begin{equation}
  \label{eq:h-inn}
W_{\pm}(x,y):=\int_{-\infty}^\infty\int_{-\infty}^\infty
w_\pm(u,v)e(-ux-vy)\d u \d v,
\end{equation}
for the corresponding Fourier transform. We collect together some basic properties of these functions in
the following result.

\begin{lemma}\label{lem:HB}
Let $N\geq 0$ be an arbitrary integer and let $H\geq 1$. 
Let $\mathcal{S}\subset \RR^2$ be a region enclosed by a piecewise
differentiable boundary, which is contained in $[-r,r]^2$ for some
$r>0.$  Then the following hold:
\begin{enumerate}
\item
For any $\x\in \RR^2$ we have
$0\leq w_\pm(\x)\leq 1$.
\item
There exists a function $\widetilde{w}_-:
\RR^2\rightarrow \RR$ such that for any $\x\in \RR^2$ we have
$$
w_-(\x)+\widetilde{w}_-(\x)\leq 1_{\mathcal{S}}(\x) \leq w_+(\x)\big(1+O_N(H^{-N})\big),
$$
where $1_{\mathcal{S}}$ is the characteristic function of the set
$\mcal{S}$ and 
$$
\widetilde{w}_-(\x)\ll \frac{e^{-H}}{1+|\x|^2}.
$$
\item
We have 
\begin{align*}
\int_{\RR^2} w_\pm(\x)\d \x =
\vol(S_{\pm})=
\vol(\mathcal{S})+O\Big(\frac{1}{\sqrt{H}}\Big).
\end{align*}
\item We have
$
W_\pm(x,y)\ll_N H^{2N}\max\{|x|,|y|\}^{-N}.
$
\end{enumerate}
The implied constants in these estimates depends
implicitly on $r$.
\end{lemma}

\begin{proof}
To begin with we record the trivial inequalities
$$
0\leq w_{\pm}(\x)\leq 
\frac{H^2}{\pi} \int_{\RR^2} \exp(-\|\y-\x\|^2H^2) \d\y=1,
$$
for any $\x\in\RR^2$, which establishes part (1). 

Turning to part (2) we observe that for 
$\x\in \mcal{S}$ we have
\begin{align*}
w_+(\x)
&\geq 
\frac{H^2}{\pi} \int_{\y\in B(\x,\frac{1}{\sqrt{H}})}
\exp(-\|\y-\x\|^2H^2) \d\y\\
&=
\frac{1}{\pi} \int_{\y\in [-\sqrt{H},\sqrt{H}]^2}
\exp(-\|\y\|^2) \d\y\\
&\geq 1+O\big(\exp(-H)\big).
\end{align*}
Since $w_+(\x)\geq 0$ when $\x\not\in \mathcal{S}$ it easily follows
that
$$
1_{\mathcal{S}}(\x) \leq w_+(\x)\big(1+O_N(H^{-N})\big),
$$
for any $N\geq 0$. Next, if $\x\in \mcal{S}$ then $w_-(\x)\leq
1$. Suppose now that $\x\not\in \mcal{S}$. Then we have 
\begin{align*}
w_-(\x)
&\leq 
\frac{H^2}{\pi} \int_{\y\in \RR^2\setminus B(\x,\frac{1}{\sqrt{H}})}
\exp(-\|\y-\x\|^2H^2) \d\y
= O\big(\exp(-H)\big).
\end{align*}
This estimate suffices for part (2) when  
$\|\x\|\leq 4r$. 
Since $\mcal{S}\subseteq [-r,r]^2$ we see that
$\|\y\|\leq 2r$ if  $\y\in S_{-}$.
If $\|\x\|\geq 4r$ then for 
all $\y\in S_{-}$ we have 
$\|\y-\x\|^2\geq \frac{1}{4}\|\x\|^2$ since
$\|\y\|\leq 2r\leq 
\frac{1}{2}\|\x\|$. It follows that 
$$
w_-(\x)
\leq 
\frac{H^2}{\pi} \exp\Big( -\frac{\|\x\|^2H^2}{4}\Big)\int_{\y\in
  S_+}\hspace{-0.1cm}\d\y \ll \exp\Big( -\frac{\|\x\|^2H^2}{5}\Big)\ll
\frac{\exp(-H)}{|\x|^2},
$$
if $\|\x\|\geq 4r$, which therefore suffices to complete the proof of
part (2).

Turning to part (3), we note that the
estimate for $\vol(S_{\pm})$ is 
a consequence of the piecewise differentiability of the boundary of $\mcal{S}$.
Finally part  (4) follows from repeated integration by parts.
\end{proof}

It follows from 
\eqref{eq:train} that
$$
\#(\sfl\cap \R'(X))=\#\{(s,t)\in \sfl: (s/\sqrt{X}, \sqrt{|a^2-b^2|}t/\sqrt{X})\in \Sab'\}.
$$
We will use the Poisson summation formula to examine this quantity. 
First we deduce from Lemma \ref{lem:volume} that $\mathcal{S}_{u}'$
in \eqref{eq:SBA'}
satisfies the hypotheses in Lemma \ref{lem:HB}. We deduce that
$$
\Sigma_- +\widetilde{\Sigma}_- \leq \#(\sfl\cap \R'(X))\leq 
\Sigma_+\big(1+O_N(H^{-N})\big),
$$
where
\begin{equation}
  \label{eq:bourgain}
  \Sigma_{\pm}=\sum_{(s,t)\in
  \sfl}w_{\pm}\Big(\frac{s}{\sqrt{X}},\frac{\sqrt{|a^2-b^2|}t}{\sqrt{X}}\Big), \quad
\widetilde{\Sigma}_{-}=\sum_{(s,t)\in
  \sfl}\widetilde{w}_{-}\Big(\frac{s}{\sqrt{X}},\frac{\sqrt{|a^2-b^2|}t}{\sqrt{X}}\Big).
\end{equation}
We will begin by examining the contribution from the sums $\Sigma_\pm$. 

Let us recall from 
\eqref{eq:lattice-1}  
that
$$
\sfl=\{ (s,t)\in \ZZ^2: \mbox{
$[\la_1',\ell ]\mid s$, $\la_2' \mid s-at$ and $\ell \mid
   t$}\},
$$
for odd  positive integers $\la_i', \ell$
such that $\lambda_i'\mid a^2+(-1)^ib^2$ and $\la_i'=k_i\la_i$.
In particular $\gcd(\la_1',\la_2')=\gcd(\la_i',ab)=1$. 
We will find it convenient to set
$
d_i:=\hcf(\ell,\la_i')$ and $\la_i'':=\la_i'/d_i$,
for $i=1,2$. For any $(s,t) \in \sfl$ we may therefore make the change of
variables
$$
s= \ell \la_1'' \sigma, \quad
s-at=\ell \la_2'' \tau.
$$
Recalling that $\ell\mid t$, we see that under this change of variables we have
\begin{equation}
  \label{eq:ias}
\la_1'' \sigma\equiv \la_2''\tau
\bmod{a},
\end{equation}
and there  is clearly a bijection between elements of $\sfl$ and
solutions to this congruence.  We therefore have
\begin{align*}
\Sigma_\pm
&=
\sum_{\substack{(\sigma,\tau)\in \ZZ^2\\
\mbox{\scriptsize\eqref{eq:ias} holds}}}
w_{\pm}\Big(\frac{
\ell \la_1'' \sigma}{\sqrt{X}},\frac{\sqrt{|a^2-b^2|}(
 \ell \la_1'' \sigma-\ell
\la_2'' \tau)}{a\sqrt{X}}\Big).
\end{align*}
say. Breaking the sum into residue classes modulo $a$, an
application of the Poisson summation formula yields
\begin{align*}
\Sigma_\pm
&=
\sum_{\substack{(\alpha,\beta)\bmod{a}\\
\la_1'' \alpha\equiv \la_2''\be
\bmod{a}}}
\sum_{\substack{(\sigma,\tau)\in \ZZ^2\\
\sigma\equiv \alpha \bmod{a}\\
\tau\equiv \be \bmod{a}}}
w_{\pm}\Big(\frac{\ell 
\la_1'' \sigma}{\sqrt{X}},\frac{\sqrt{|a^2-b^2|}(
 \ell \la_1'' \sigma-\ell\la_2'' \tau)}{a\sqrt{X}}\Big)\\
&=\frac{X}{a^2\ell^2\la_1''\la_2''}
\sum_{\substack{(\alpha,\beta)\bmod{a}\\
\la_1'' \alpha\equiv \la_2''\be
\bmod{a}}}
\sum_{(m,n)\in \ZZ^2}
e\Big(\frac{\al m+\be n}{a}\Big)
W_{\pm,b/a}\Big(\frac{m}{\la_1''T}, \frac{n}{\la_2''T}\Big),
\end{align*}
where
\begin{equation}
  \label{eq:TTT}
T:=\frac{a\ell}{\sqrt{X}},
\end{equation}
and for $\delta>0$ we have temporarily set
$$
W_{\pm,\delta}(x,y):=\int_{-\infty}^\infty\int_{-\infty}^\infty
w_\pm(u,\sqrt{|1-\delta^2|}(u-v))e(-ux-vy)\d u \d v.
$$
It is easy to see that
$W_{\pm,\delta}(x,y)=W_{\pm}(x+y,y/\sqrt{|1-\delta^2|})/\sqrt{|1-\delta^2|}$
in the notation of \eqref{eq:h-inn}.
We may therefore write 
\begin{align*} 
\Sigma_\pm &=\frac{Z}{a}
\sum_{\substack{(\alpha,\beta)\bmod{a}\\
\la_1'' \alpha\equiv \la_2''\be
\bmod{a}}}\hspace{-0.3cm}
\sum_{(m,n)\in \ZZ^2}
e\Big(\frac{\al m+\be n}{a}\Big)
W_{\pm}\Big(\frac{m/\la_1''+n/\la_2''}{T}, \frac{an}{\sqrt{|a^2-b^2|}\la_2''T}\Big),
\end{align*}
where 
\begin{equation}\label{eq:ZZ}
Z:=\frac{X}{\ell^2\la_1''\la_2''\sqrt{|a^2-b^2|}}
=\frac{X}{\det\sfl \sqrt{|a^2-b^2|}},
\end{equation}
by Lemma \ref{lem:det}.

It follows that 
$
\Sigma_\pm = \mcal{M}_\pm+\mcal{E}_\pm,
$
where $\mcal{M}_\pm=Z W_\pm(0,0)$ and 
$\mcal{E}_\pm$ is the overall contribution from non-zero vectors in
the summation over $(m,n)$.   Now it is clear from 
Lemmas \ref{lem:volume} and \ref{lem:HB}
that 
$$
\mcal{M}_\pm=Z\int_{\RR^2} w_\pm(\x)\d \x =
Z\vol(\mathcal{S}_{b/a}')+O\Big(\frac{Z}{\sqrt{H}}\Big)
=
\frac{\vol(\R'(X))}{\det \sfl}
+O\Big(\frac{Z}{\sqrt{H}}\Big).
$$
Recalling the definition \eqref{eq:XX} of $X$ and the expression for
$\det \sfl$ in Lemma \ref{lem:det}, we see that the error
term here contributes 
\begin{align*}
&\ll\frac{B}{\sqrt{H}}
\sum_{(a,b)\in \A_1} 
\Osum_{\substack{k_1\lambda_1\mid a^2-b^2\\k_1\leq K}}
\Osum_{\substack{k_2\lambda_2\mid a^2+b^2\\k_2\leq K}}
\sum_{\nu}
\Osum_{\ell \leq  Z_1^{c/10}}
\frac{\gcd(\ell,k_1k_2\la_1\la_2)2^\nu}{\ell^2k_1k_2\max\{a,b\}\sqrt{|a^2-b^2|}}.
\end{align*}
Carrying out the inner 
sums over $\ell$ and $\nu\leq Z_2$, we deduce from 
\eqref{eq:exp} that these contribute 
$$
\ll \sum_{\nu}\frac{\tau(k_1k_2\la_1\la_2)2^\nu}{
k_1k_2\max\{a,b\}\sqrt{|a^2-b^2|}} \ll 
\frac{\tau(k_1k_2\la_1\la_2)\log B}{
k_1k_2\max\{a,b\}^{3/2}|a-b|^{1/2}},
$$
since $|a^2-b^2|\geq \max\{a,b\}|a-b|$.
Hence the error term contributes
\begin{align*}
&\ll\frac{B\log B}{\sqrt{H}}
\sum_{(a,b)\in \A_1} 
\frac{\tau(|a^4-b^4|)}{\max\{a,b\}^{3/2}|a-b|^{1/2}}
\ll\frac{B(\log B)^5}{\sqrt{H}},
\end{align*}
by \eqref{eq:anna-swim'}, 
once summed over all of the remaining parameters.
This is satisfactory for Lemma \ref{lem:exp} if $H\geq (\log B)^4$,
which we now assume.

Let us turn to the estimation of $\mcal{E}_\pm$, observing that
$$
\sum_{\substack{(\alpha,\beta)\bmod{a}\\
\la_1'' \alpha\equiv \la_2''\be
\bmod{a}}}
e\Big(\frac{\al m+\be n}{a}\Big)=
\begin{cases}
a, & \mbox{if $n\la_1''+m\la_2''\equiv 0 \bmod{a}$,}\\
0, & \mbox{otherwise}.
\end{cases}
$$
We may therefore write
\begin{align*}
\mcal{E}_\pm
&=Z
\sum_{\substack{(m,n)\in \ZZ^2\setminus\{\mathbf{0}\}\\
n\la_1''+m\la_2''\equiv 0 \bmod{a}}}
W_{\pm}\Big(\frac{m/\la_1''+n/\la_2''}{T}, \frac{an}{\sqrt{|a^2-b^2|}\la_2''T}\Big),
\end{align*}
where $T$ is given by \eqref{eq:TTT} and $Z$ by \eqref{eq:ZZ}.
We may assume that $m,n$ are both $O(B^c)$ for a suitable 
absolute constant $c>0$ by the rapid decay properties enjoyed by the
function $W_{\pm}$, as described in 
Lemma \ref{lem:HB}.

We begin by considering the overall contribution from terms with $m=0$
in this summation, which forces $n$ to be divisible by  $a$.
We deduce from taking $N=1$ in 
part (4) of Lemma~\ref{lem:HB} 
that $W_\pm(x,y)\ll H^2/|y|$, whence this part of the sum is
\begin{align*}
Z
\sum_{\substack{n'\in \ZZ\setminus\{0\}\\ n'\leq B^c}}
W_{\pm}\Big(\frac{an'}{\la_2''T},
\frac{a^2n'}{\sqrt{|a^2-b^2|}\la_2''T}\Big) 
&\ll
\frac{H^2\la_2''ZT\log B}{a}.
\end{align*}
Recalling the definitions of $T,X,Z$ from \eqref{eq:TTT},
\eqref{eq:XX} and \eqref{eq:ZZ} and noting 
from \eqref{eq:AA} that $|a-b|\geq Z_2^{-2}\max\{a,b\}$, we see that this is 
\begin{align*}
&\ll
\frac{\gcd(\ell,k_1\la_1)}{\ell} \cdot \sqrt{\frac{\la_2}{\la_1}}
\cdot \frac{H^2 2^{\nu/2} Z_2\sqrt{B}\log B}{k_1\max\{a,b\}^{3/2}}. 
\end{align*}
It follows from the statement of Lemma \ref{lem:final_prelim} 
and the definition of \eqref{eq:VV} that
$$
\frac{k_2\la_2}{k_1\la_1}\leq \frac{B}{\max\{a,b\}K}.
$$
Hence the above sum is
$$
\ll
\frac{\gcd(\ell,k_1\la_1)}{\ell \sqrt{k_1k_2}} \cdot \frac{H^2 2^{\nu/2} Z_2
B\log B}{\max\{a,b\}^2\sqrt{K}}.
$$
Note that 
$\sum_{\ell\leq L} |\mu(\ell)|\gcd(\ell,n)\ell^{-1}\ll 2^{\omega(n)}\log L$ for any
non-zero integer $n$ and $2^{Z_2/2}Z_2\ll \log B$. 
Moreover by checking the inequality at prime powers we see that
$$
\sum_{k\la \mid n}\frac{|\mu(k)|2^{\omega(k\la)}}{\sqrt{k}}\leq
\tau_5(n).
$$
Once inserted into Lemma \ref{lem:final_prelim},  we therefore obtain
the overall contribution
\begin{equation}\label{eq:lemon}
\begin{split}%T%
&\ll \frac{H^2 B(\log B)^3}{\sqrt{K}}
\sum_{(a,b)\in \A_1}
\sum_{k_1\lambda_1\mid a^2-b^2}
\sum_{k_2\lambda_2\mid a^2+b^2}
\frac{|\mu(k_1)\mu(k_2)|2^{\omega(k_1k_2\la_1\la_2)}}{\sqrt{k_1k_2}\max\{a,b\}^2}\\
&\ll \frac{H^2 B(\log B)^3}{\sqrt{K}}
\sum_{(a,b)\in \A_1} 
\frac{\tau_5(|a^4-b^4|)}{\max\{a,b\}^2}\\
&\ll \frac{H^2B(\log B)^{16}}{\sqrt{K}},
\end{split}
\end{equation}
by Lemma~\ref{lem:nair}.
In view of \eqref{eq:K} this is satisfactory for Lemma \ref{lem:exp}
if $H\leq (\log B)^{6}$. 
Similarly, the overall contribution from 
terms with $n=0$ is seen to be satisfactory. 
In view of our earlier constraints on $H$ we are led to take the value 
$$
H=(\log B)^6
$$
in our construction of the weight functions $w_{\pm}$.

Next we consider the contribution from terms with
$m\la_2''+n\la_1''=0$.  The general solution of this equation is
$(m,n)=k(-\la_1'',\la_2'')$ for non-zero integer $k\leq B^c$. 
Hence part (4) of Lemma \ref{lem:HB} gives the contribution  
\begin{align*}
Z
\sum_{\substack{k\in \ZZ\setminus\{0\}\\ k\leq B^c}}
W_{\pm}\Big(0, \frac{ak}{\sqrt{|a^2-b^2|}T}\Big)
&\ll
\frac{H^2\sqrt{|a^2-b^2|}ZT\log B}{a}\\
&\ll
\frac{\gcd(\ell,k_1\la_1k_2\la_2)}{\ell} \cdot
\frac{1}{\sqrt{\la_1\la_2}} \cdot \frac{H^2 2^{\nu/2}Z_2 \sqrt{B}\log
  B}{k_1k_2\max\{a,b\}^{1/2}} 
\end{align*}
from this part of the sum. It follows from \eqref{eq:VV} that
$$
\frac{1}{k_1\la_1k_2\la_2}\leq \frac{B}{K\max\{a,b\}^3}.
$$
Hence the contribution is
$$
\ll
\frac{\gcd(\ell,k_1\la_1k_2\la_2)}{\ell \sqrt{k_1k_2}} \cdot \frac{H^2 2^{\nu/2}Z_2 B\log B}{\max\{a,b\}^{2}\sqrt{K}},
$$
which as previously once 
we insert it into Lemma \ref{lem:final_prelim} gives the 
satisfactory overall contribution \eqref{eq:lemon}.
In the same fashion one deduces that there is a satisfactory overall
contribution to the sum from vectors $(m,n)$ for which  
$(a^2-b^2)m/\la_1''=(a^2+b^2)n/\la_2''$.

In what follows we may therefore approximate 
$\mcal{E}_\pm$ by the corresponding sum $\mcal{E}_\pm'$, say, 
in which 
$$
mn\neq 0, \quad
m\la_2''+n\la_1''\neq 0, \quad
\frac{(a^2-b^2)m}{\la_1''}\neq \frac{(a^2+b^2)n}{\la_2''},
$$
with a satisfactory error. 
Now it is clear that 
$n\la_1''+m\la_2''\equiv 0 \bmod{a}$
if and only if
$n\ell\la_1''+m\ell\la_2''=a\ell t'$ for some integer $t'$.  
Let us write $s=\ell\la_1''n$ and $t=\ell t'$. Then we have 
\begin{align*}
\mcal{E}_\pm'
&=Z
\sum_{\substack{(s,t)\in \sfl\\
st(s-at)\neq 0\\
s/t\neq (a^2-b^2)/(2a)}}
W_{\pm}\Big(\frac{at}{\ell \la_1''\la_2''T}, \frac{as}{\sqrt{|a^2-b^2|}\ell \la_1''\la_2''T}\Big).
\end{align*}
We may freely assume that $|s|,|t|\leq B^{3/2}$ in this summation by
the fast decay 
properties of $W_\pm$ recorded in Lemma \ref{lem:HB}.

Recall the definition of the quadratic forms $Q_j(s,t)$ from \eqref{eq:Qi} for
$1\leq j\leq 3.$ We would now like to show
that there is a negligible contribution from $(s,t)$ for which one of
these forms vanishes. 
If $Q_1(s,t)=0$ (resp. $Q_2(s,t)=0$) then 
there exists an integer $m$ for which $a^2+b^2=2m^2$ 
(resp. $a^2-b^2=2m^2$) and $s/t=a\pm
m$ (resp. $s/t=\pm m$). 
Thus we have $(s,t)=k(a\pm m,1)$ (resp. 
$(s,t)=k(\pm m,1)$) 
for non-zero integer $k$
satisfying $|k|\leq B^{3/2}\max\{a,b\}\leq B^{2}.$
Furthermore the conditions of summation imply that $\ell\mid k$.
Hence the contribution 
to $\mcal{E}_\pm'$ from $Q_1(s,t)Q_2(s,t)=0$ is 
\begin{equation}\label{eq:sloe}
\ll \frac{H^2\ell \la_1''\la_2''ZT}{a}
\sum_{\substack{0<|k|\leq B^2\\
\ell\mid k}}
\frac{1}{k}.
\end{equation}
The contribution from $(s,t)$ such that $Q_3(s,t)=0$ is more subtle
but ultimately  contributes the same sort of
quantity. Indeed if  $Q_3(s,t)=0$ then 
there exists coprime $m_1,m_2\in \NN$ such that 
$$
b^2-a^2=2^\kappa m_1^2, \quad
b^2+a^2=2^\kappa m_2^2, 
$$
where $\kappa=1$ if $2\nmid ab$ and $\kappa=0$ otherwise.
In particular we have $2a^2=2^{\kappa}(m_2^2-m_1^2)$.
Note that if  $\kappa=0$ then  $2\mid a$ and $2\nmid m_1m_2$.
Hence there exists coprime $a_1,a_2\in \NN$ such that 
$$
m_2-m_1=2^{2(1-\kappa)}a_1^2, \quad 
m_2+m_1=2^{1-\kappa}a_2^2,
$$
with $2^{1-\kappa}a_1a_2=a$. The equation $Q_3(s,t)=0$ yields
$t(a^2-b^2)(2s-at)=2as^2$. Since $a$ is coprime to $a^2-b^2$ it
follows that $2^\kappa m_1^2\mid 2s^2$, whence $m_1 \mid s$.
Writing $s=m_1s'$ we deduce that 
$$
s'/t=-2^{\kappa-1}(m_1\pm m_2)/a.
$$
Note that 
$-2^{\kappa-1}(m_1 - m_2)/a= a_1/a_2$ 
and 
$-2^{\kappa-1}(m_1 + m_2)/a= 2^{\kappa-1}a_2/a_1$.
In the former case we deduce that $(s,t)=k(m_1a_1,a_2)$ and in the
latter case $(s,t)=k(2^{\kappa-1}m_1a_2,a_1)$ for a non-zero integer
$k$. But then the conditions of summation imply that $\ell\mid k$ and
so we end up with the same sum \eqref{eq:sloe} for the contribution 
to $\mcal{E}_\pm'$ from $Q_3(s,t)=0$.

Recall the definitions 
of $T,X,Z$ from \eqref{eq:TTT},
\eqref{eq:XX} and \eqref{eq:ZZ}.
Furthermore, 
recall from \eqref{eq:VV}  that $\la_1\la_2\leq \max\{a,b\}B/K$. 
It therefore follows that \eqref{eq:sloe} has size
\begin{align*}
\ll%T%
\frac{H^2 \la_1''\la_2''ZT\log B}{a}
&\ll \frac{H^2\sqrt{X}\log B}{\ell \sqrt{|a^2-b^2|}}\\
&\ll \frac{2^{\nu/2}H^2\sqrt{B}(\log B)\sqrt{\la_1\la_2}}{\ell
  \max\{a,b\}^{1/2}\sqrt{|a^2-b^2|}}\\
&\ll \frac{H^2B(\log B)^{2}}{\sqrt{K}\ell \max\{a,b\}},
\end{align*}
since %T%
$$%T%
2^{\nu/2}\sqrt{\frac{\max\{a,b\}}{|a-b|}}\ll \log B.
$$
We now need to sum this over the remaining parameters
$a,b,k_i,\la_i,\ell$.  Summing first over $\ell$ we obtain the overall
contribution
\begin{align*}
&\ll \frac{H^2B(\log B)^{3}}{\sqrt{K}}
\sum_{i,j\in \{0,1\}}
\sum_{\substack{0<a,b,m\leq \sqrt{B}\\
a^2+(-1)^i b^2=2^j m^2}}
\frac{\tau_3(|a^4-b^4|)}{ \max\{a,b\}}\ll
\frac{H^2B(\log B)^{29}}{\sqrt{K}},
\end{align*}
by Lemma \ref{lem:nair''}. Our choice \eqref{eq:K} of $K$ ensures that
this is satisfactory.

We may therefore approximate  $\mcal{E}_\pm'$ by
\begin{align*}
Z
\sum_{\substack{(s,t)\in \sfl\\
st(s-at)\neq 0, ~Q_j(s,t)\neq 0\\
s/t\neq (a^2-b^2)/(2a)}}
W_{\pm}\Big(\frac{at}{\ell \la_1''\la_2''T}, \frac{as}{\sqrt{|a^2-b^2|}\ell \la_1''\la_2''T}\Big)
\end{align*}
with satisfactory error.
Taking $N=2$ in 
part (4) of Lemma \ref{lem:HB} we observe that the above sum is 
$O(\mcal{E}_\pm'')$, with 
\begin{equation}\label{eq:green}
\mcal{E}_\pm'':=
\frac{H^4 Z(\ell \la_1''\la_2''T)^2}{
a^2} 
\sum_{\substack{(s,t)\in \sfl\\
st(s-at)\neq 0, ~Q_j(s,t)\neq 0\\
s/t\neq (a^2-b^2)/(2a)}} \min\Big\{\frac{\max\{a,b\}^2}{s^2},\frac{1}{t^2}\Big\}.
\end{equation}
Let $R, A,L_1,L_2>0$. We will estimate the overall contribution from 
$\mcal{E}_{\pm}''$ once inserted into Lemma \ref{lem:final_prelim}, with 
\begin{equation}\label{eq:shang}
A\leq \max\{a,b\}  <2A, \quad L_i\leq \la_i< 2L_i, \quad 
R<\max\{|s|/A,|t|\}\leq 2R.
\end{equation}
We call this contribution
$E=E_\pm(R,A,L_1,L_2)$.  
Recall that $|s|,|t|\leq B^{3/2}$ in the sum by the fast decay
properties of the weights $W_\pm$. 
In particular we may 
clearly assume that
$$
R, A,L_1,L_2\gg 1, \quad R\leq B^{3/2}, \quad
A\leq \sqrt{B}/Z_1^c, \quad L_i\leq 8A^2.
$$
We will show that 
\begin{equation}
   \label{eq:nuke}
   E\ll  \frac{BZ_1^{O(1)}}{Z_1^c}.
\end{equation}
Once summed over the $O((\log B)^4)$ dyadic ranges for $R,A,L_1,
L_2$, this will clearly suffice to complete our treatment of the error
term in Lemma \ref{lem:exp}.

Recall from \eqref{eq:AA} 
that $\min \{a,b,|a-b|\}\geq \max\{a,b\}/Z_2^2$, so that 
$$
\frac{A}{Z_2^2}\leq\min \{a,b,|a-b|\}\leq 2A.
$$
Furthermore we have $k_1,k_2\leq K$. 
Thus we see that the contribution to 
our estimate for $\mcal{E}_{\pm}''$ from $s, t, \la_i$ and $a$
restricted to lie in the above dyadic intervals is 
\begin{align*}
&\ll \frac{H^4(\log B)^{\ve}
Z(\ell \la_1''\la_2''T)^2}{
A^2R^2} \#\left\{
(s,t)\in \sfl:
\begin{array}{l}
\mbox{$Q_j(s,t)\neq 0$ for $1\leq j\leq 3$},\\
st(s-at)\neq 0,\\
s/t\neq (a^2-b^2)/(2a),\\
|s|\leq 2AR, ~|t|\leq 2R
\end{array}
\right\}.
\end{align*}
Recalling the definitions of $Z$ and $T$,  
the factor outside the cardinality is seen to be
\begin{align*}
\frac{H^4(\log B)^{\ve}
(\ell \la_1''\la_2'')^2}{
A^2R^2} \cdot ZT^2 
&\ll
\frac{H^4(\log B)^{\ve}
(\ell \la_1''\la_2'')^2}{
A^2R^2} \cdot 
\frac{X}{\ell^2\la_1''\la_2''\sqrt{|a^2-b^2|}}
\cdot
\frac{a^2\ell^2}{X}\\
&\ll
\frac{H^4 (\log B)^{\ve}
\ell^2\la_1''\la_2''}{A R^2}.
\end{align*}
Note that $\la_i''\leq k_i\la_i\ll KL_i$
and $|Q_i(s,t)|\ll A^2R^2$ if 
$|s|\leq 2AR$ and $|t|\leq 2R$.
Hence the contribution under consideration is 
\begin{align*}
&
\ll \frac{H^4\ell^2K^2L_1L_2(\log B)^{\ve}}{AR^2}
 \#\big(\sfl\cap \R^\dagger(cA^2R^2)\big),
\end{align*}
for some absolute constant $c>0$, in the notation of \eqref{eq:R-dag}.
Since
$\ell\leq Z_1^{c/10}$ and $H=(\log B)^6\leq Z_1^{O(1)}$ 
it now follows that
\begin{align*}
 E&\ll 
\frac{L_1L_2Z_1^{c/5+O(1)}}{AR^2}
\sum_{\substack{
(a,b)\in \NN^2\\
\cp{a}{b}, ab\neq 1\\
\max\{a,b\}\leq 2A}}
\Osum_{\substack{k_i\lambda_i\mid a^2+(-1)^ib^2\\ k_i\leq K}}
\Osum_{\ell \in \NN}
 \#\big(\sfl\cap \R^\dagger(cA^2R^2)\big),
\end{align*}
To analyse this sum we observe that 
$
\#\big(\sfl\cap \R^\dagger(cA^2R^2)\big)
\leq L^\dagger_{k_1,k_2,\ell}(Y,1/2),
$
with 
$$
\frac{A^3R^2}{L_1L_2}\ll Y\ll \frac{A^3R^2}{L_1L_2}.
$$
Here $Y\gg 1 $  since
$\la_1\leq |s|\ll AR$ and $\la_2\leq |s-at|\ll AR$, whence 
$
L_1L_2\ll A^2R^2
$
in order for the summand not to vanish. It is trivial to see that
$Y\ll B^5.$ It therefore follows from taking $T=1/2$ in 
Lemma \ref{lem:large-gcd} that 
\begin{align*}
 E \ll 
\frac{L_1L_2Z_1^{c/5+O(1)}}{AR^2}
\cdot Y \ll A^2Z_1^{c/5+O(1)}\ll \frac{BZ_1^{O(1)}}{Z_1^{9c/5}},
\end{align*}
which thereby establishes \eqref{eq:nuke}.

Our final task is to deal with the contribution from the sum
$\widetilde{\Sigma}_-$ in \eqref{eq:bourgain}. 
Applying the estimate for $\widetilde{w}_-$ in Lemma \ref{lem:HB} we deduce
that
$$
\widetilde{\Sigma}_-
\ll \frac{e^{-H}X}{|a^2-b^2|}  
\sum_{\substack{(s,t)\in \sfl}}
\min\Big\{\frac{\max\{a,b\}^2}{s^2},\frac{1}{t^2}\Big\}. 
$$ 
The same argument allows us to restrict attention to $s,t$ for which 
$st(s-at)\neq 0$, $Q_j(s,t)\neq 0$ and 
$s/t\neq (a^2-b^2)/(2a)$. But then we are led to the upper bound 
$$
\widetilde{\Sigma}_-
\ll \frac{e^{-H}X a^2}{|a^2-b^2|H^4Z (\ell \la_1'' \la_2''T)^2} \cdot
\mathcal{E}_{\pm}'' +B(\log B)^3(\log\log B)^2,   
$$
in the notation of \eqref{eq:green}.
As there we 
will estimate the overall contribution 
$\widetilde{E}=\widetilde{E}(R,A,L_1,L_2)$, say, 
from the first term for the dyadic intervals \eqref{eq:shang}. 
Running through the treatment of $E$ we see that the first term
is
\begin{align*}
&\ll \frac{e^{-H}X a^2}{|a^2-b^2|H^4Z (\ell \la_1'' \la_2''T)^2} \cdot
\frac{H^4(\log B)^{\ve}
Z(\ell \la_1''\la_2''T)^2}{
A^2R^2} \cdot 
\#\big(\sfl\cap \R^\dagger(cA^2R^2)\big)\\
&\ll \frac{e^{-H}(\log B)^\ve X }{A^2R^2} \cdot  \#\big(\sfl\cap
\R^\dagger(cA^2R^2)\big)
\\
&\ll \frac{e^{-H}B(\log B)^{1+\ve}L_1L_2  }{A^3R^2} \cdot  \#\big(\sfl\cap
\R^\dagger(cA^2R^2)\big)
\\
\end{align*}
for some absolute constant $c>0$, on substituting the definition
\eqref{eq:XX} of $X$. 
Applying  Lemma \ref{lem:large-gcd} as previously we therefore deduce that 
\begin{align*}
\widetilde{E}
&\ll \frac{e^{-H}B(\log B)^{1+\ve}L_1L_2  }{A^3R^2}
\cdot 
\frac{A^3R^2}{L_1L_2}
\ll e^{-H}B(\log B)^{1+\ve},
\end{align*}
which therefore makes a satisfactory contribution for 
$H=(\log B)^6$. This completes the proof of Lemma \ref{lem:exp}.

\section{The main term}\label{s:3}

We now draw together the various main terms that appear
in Lemmas \ref{lem:L0}, \ref{lem:L1} and~\ref{lem:L2}, and insert them
into Lemma \ref{lem:final_prelim}'s estimate for $N_1(B)$.

By abuse of notation
we will merely equate $L^\nu(B)=\widetilde{L}_{k_1,k_2,\ell}(B)$ with
the main term in the various
estimates from \S \ref{s:asy}.
Recall the definition~\eqref{eq:pi} of $\pi_{ab}$.
We may clearly bring the summation over $\nu$ to 
the innermost sum, finding that
\begin{align*}
\sum_{\nu \leq \max\{1,\nu_2(a^2-b^2)\}}
L^\nu(B)&=
\Big(
\pi_{ab}\Big(\frac{1}{2}-\frac{1}{2^{\lfloor 2Z_2 \rfloor}}\Big)
+
1-\frac{\pi_{ab}}{2}
+\pi_{ab+1}(\nu_2(a^2-b^2)-3)\\
&\qquad\qquad\qquad\qquad+
\pi_{ab+1}\Big(2-\frac{2^{\nu_2(a^2-b^2)}}{2^{\lfloor 2Z_2 \rfloor+2}}\Big)
\Big)\frac{\vol(\R_0')}{\det \sfl}
\\
&=
\big(\delta_{a,b}+O(2^{-Z_2})\big)
\frac{\vol(\R_0')}{\det \sfl},
\end{align*}
where if
$g=h*\tau=h*1*1$ is given by \eqref{eq:g} and \eqref{eq:h}, then
\begin{align*}
   \delta_{a,b}
:=\max\{1, \nu_2(a^2-b^2)\}=g(2^{\nu_2(a^4-b^4)}).
\end{align*}
Note here that $\nu_2(a^4-b^4)=\nu_2(a^2-b^2)+1$ if
$\nu_2(a^2-b^2)\neq 0$ and %T%
$\nu_2(a^4-b^4)=0$ otherwise. %T%
Moreover, we have used the fact that $\nu_2(a^2-b^2)\leq Z_2$ in
controlling the error term. We may therefore deduce from
Lemma~\ref{lem:volume} that
$$
\sum_{\nu \leq \max\{1,\nu_2(a^2-b^2)\}}L^\nu(B)=
\Big(\delta_{a,b}+O\big(\frac{1}{\log B}\big)\Big)
\frac{\la_1\la_2B
   \vol(\Sab)}{(\det \sfl)\max\{a,b\}\sqrt{|a^2-b^2|}},
$$
where $\mathcal{S}_u$ is given by \eqref{eq:SBA} for
any positive $u\neq 1$.

Define
\begin{equation}
   \label{eq:func}
f(u):=\frac{\vol(\mathcal{S}_u)+\vol( \mathcal{S}_{1/u})}{\sqrt{1-u^2}},
\end{equation}
for any $u\in (0,1)$.
Let $\chi(t_1,t_2;R)$ be given by \eqref{eq:chi} and 
for $Y\geq 1$ let
\begin{equation}
   \label{eq:hi}
h(a,b;Y):=
\sum_{\substack{n\mid a^4-b^4\\
d_i=\gcd(n,a+(-1)^ib)\\
     ~d_3=\gcd(n,a^2+b^2)}}
(1*h)(n)\chi(d_1d_2,d_3;Y).
\end{equation}
We deduce from \eqref{eq:h} that 
$(1*h)(p^\nu)\leq 1$, so that 
$h(a,b;Y)\leq \tau(|a^4-b^4|)$.
We are now ready to establish the following result.

\begin{lemma}\label{lem:farine}
We have
\begin{align*}
N_1(B)&\leq \frac{8B}{3\zeta(2)}
\sum_{(a,b)\in \A_2}\frac{f(b/a)h(a,b;B2^{Z_2+1}/K)}{a^2}
  +O\Big(\frac{B(\log
   B)^4}{\log\log B}\Big)\\
N_1(B)&\geq \frac{8B}{3\zeta(2)} %T%
\sum_{(a,b)\in \A_2}\frac{f(b/a)h(a,b;B2^{-(Z_2+1)}/K)}{a^2}
  +O\Big(\frac{B(\log
   B)^4}{\log\log B}\Big),
\end{align*}
where $h(a,b;Y)$ is given by \eqref{eq:hi} and
$$
\A_2:=
\left\{(a,b)\in \NN^2:
a/Z_2^2\leq b \leq a(1-1/Z_2^2), ~a < \sqrt{B}, ~\cp{a}{b}
\right\}.
$$
\end{lemma}

\begin{proof}
Let us write 
$$
\widetilde{\delta_{a,b}}=\delta_{a,b}+O\Big(\frac{1}{\log B}\Big)=
\delta_{a,b}\Big( 1+O\Big(\frac{1}{\log B}\Big)\Big).
$$
Bringing together our expression for $\sum_\nu L^\nu(B)$ with Lemmas
\ref{lem:det}, \ref{lem:final_prelim} and \ref{lem:exp}, we deduce
that $N_1(B)$ can be replaced by
\begin{align*}
2B
\sum_{(a,b)\in \A_1}
\widetilde{\delta_{a,b}}
F(a,b)
\Osum_{\substack{k_1\lambda_1\mid a^2-b^2\\ k_1\leq K}}
\Osum_{\substack{k_2\lambda_2\mid a^2+b^2\\ k_2\leq K}}
&\chi\Big(k_1\la_1,k_2\la_2; 
\frac{B}{K}\Big)
\frac{\mu(k_1)\mu(k_2)}{k_1k_2}\sigma(Z_1^{c/10}),
\end{align*}
with an acceptable error, where $K$ is given by \eqref{eq:K},
$\chi(t_1,t_2;R)$ is given by \eqref{eq:chi},
$F(a,b):=\vol(\Sab)/(\max\{a,b\}\sqrt{|a^2-b^2|})$
  and
\begin{align*}
\sigma(T)&:= \Osum_{\ell \leq  T}
  \mu(\ell)\frac{\gcd(k_1k_2\la_1\la_2 ,\ell) 
}{\ell^2},
\end{align*}
for any $T\geq 1$. Extending the summation to infinity  we easily deduce that 
\begin{align*}
\sigma(T)
-\frac{4}{3\zeta(2)\phid(k_1k_2\la_1\la_2 
)}\ll 
\sum_{\ell> T}
  \frac{\gcd(k_1k_2\la_1\la_2 ,\ell) 
}{\ell^2}
&\ll
\frac{\tau(k_1k_2\la_1\la_2)}{T},
\end{align*}
where $\phid$ is given by \eqref{eq:funcp}.
  Lemma~\ref{lem:volume} implies that $\vol(\Sab)\ll 1$
and the definition \eqref{eq:AA} of $\A_1$ yields $
\sqrt{|a^2-b^2|}\geq \max\{a,b\}/Z_2$.
Hence the overall contribution from the above error term is 
$BZ_1^{O(1)}/Z_1^{c/10}$,
which is satisfactory.
Hence $N_1(B)$ can be replaced by
\begin{align*}
\frac{8B}{3\zeta(2)}
\sum_{(a,b)\in \A_1}
\widetilde{\delta_{a,b}}
F(a,b)
\Osum_{\substack{k_1\lambda_1\mid a^2-b^2\\ k_1\leq K}}
\Osum_{\substack{k_2\lambda_2\mid a^2+b^2\\ k_2\leq K}}
&\chi\Big(k_1\la_1,k_2\la_2; 
\frac{B}{K}\Big)
\frac{\mu(k_1)\mu(k_2)}{k_1k_2 \phid(k_1k_2\la_1\la_2 
)}, 
\end{align*}
with an acceptable error.
The overall contribution to from values of $k_1,k_2$ such
that $\max\{k_1,k_2\}>K$ is 
\begin{align*}
\ll Z_2^2B
\sum_{(a,b)\in \A_1}
\frac{1}{\max\{a,b\}^2}
\sum_{\substack{k_i\lambda_i\mid a^2+(-1)^ib^2\\ \max\{k_1,k_2\}> K}}
\frac{1}{k_1k_2} 
&\ll \frac{Z_2^2B}{K}
\sum_{(a,b)\in \A_1}
\frac{\tau_3(a^4-b^4)}{\max\{a,b\}^2}.
\end{align*}
An application of Lemma \ref{lem:nair} confirms that this is 
also satisfactory.

We may now conclude that
\begin{align*}
N_1(B)=\frac{8B(1+O(1/\log B))}{3\zeta(2)}
\sum_{(a,b)\in \A_1} &F(a,b)
h_0(a,b;B)
+O\Big(\frac{B(\log
   B)^4}{\log\log B}\Big),
\end{align*}
where
$$
h_0(a,b;B):=\delta_{a,b}
\Osum_{k_1\lambda_1\mid a^2-b^2} \frac{\mu(k_1) }{k_1 \phid(k_1\lambda_1) }
  \Osum_{k_2\lambda_2\mid a^2+b^2}
   \frac{\mu(k_2) }{k_2 \phid(k_2\lambda_2) }
\chi\Big(k_1\la_1,k_2\la_2; \frac{B}{K}\Big).
$$
Recall the definition of $\phis$ from \eqref{eq:funcp}.
For any arithmetic function $f$, we have
\begin{align*}
\Osum_{\substack{k_i\lambda_i\mid N}}
\frac{\mu(k_i) }{k_i }f(k_i\la_i)
=\Osum_{\substack{n\mid N}}f(n)
\sum_{k_i\mid n}\frac{\mu(k_i) }{k_i   }
&=\Osum_{\substack{n\mid N}}\phis(n)f(n).
\end{align*}
It therefore follows that
\begin{align*}
h_0(a,b;B)
&=
\delta_{a,b}
\Osum_{m\mid a^2-b^2} \frac{\phis(m)}{\phid(m) }
\Osum_{m_3\mid a^2+b^2} \frac{\phis(m_3)}{\phid(m_3) }
\chi\Big(m,m_3;\frac{B}{K}\Big)\\
&=
\delta_{a,b}
\sum_{\substack{n\mid a^4-b^4\\ 
m=\odd{n}{a^2-b^2}\\ m_3=\odd{n}{a^2+b^2}}} 
(1*h)(mm_3)
\chi\Big(m,m_3;\frac{B}{K}\Big)\\
&=
\sum_{\substack{n\mid a^4-b^4\\
m_i=\odd{n}{a+(-1)^ib}\\
     ~m_3=\odd{n}{a^2+b^2}}}
(1*h)(n)\chi\Big(m_1m_2,m_3;\frac{B}{K}\Big),
\end{align*}
where $h$ is given by \eqref{eq:h}.  Recall
the inequality $\nu_2(a^2-b^2)\leq Z_2$ satisfied by any~$(a,b)\in
\A_1$.  A little thought  reveals that
$$%T%
h(a,b;B2^{-(Z_2+1)}/K)\leq
h_0(a,b;B)\leq h(a,b;B2^{Z_2+1}/K),
$$
in the notation of \eqref{eq:hi}.

Note that 
$$
\sum_{(a,b)\in \A_1}F(a,b)h(a,b;B2^{Z_2+1}/K)
\ll
\sum_{(a,b)\in \A_1}
\frac{\tau(|a^4-b^4|)}{\max\{a,b\}^{3/2} |a-b|^{1/2}}\ll 
(\log B)^4,
$$
by  Lemma \ref{lem:nair'}.
Bringing everything together we have so far established the upper and
lower bounds
\begin{align*}
N_1(B)&\leq \frac{8B}{3\zeta(2)}
\sum_{(a,b)\in \A_1}F(a,b)h(a,b;B2^{Z_2+1}/K)
  +O\Big(\frac{B(\log
   B)^4}{\log\log B}\Big)\\
N_1(B)&\geq \frac{8B}{3\zeta(2)} %T%
\sum_{(a,b)\in \A_1}F(a,b)h(a,b;B2^{-(Z_2+1)}/K)
  +O\Big(\frac{B(\log
   B)^4}{\log\log B}\Big),
\end{align*}
We proceed to enlarge the set of allowable $a,b$ slightly, by
handling separately the contribution from $a,b$ such that
$\nu_2(a^2-b^2)>Z_2$. 
The third estimate in Lemma \ref{lem:nair'} ensures that 
this contributes $O(B(\log B)^4/(\log \log B))$, which is
satisfactory. Further applications of 
Lemma \ref{lem:nair'} show similarly 
that it is possible to 
enlarge the set of allowable $a,b$ to include the ranges
$\sqrt{B}/Z_1^c\leq \max\{a,b\}< \sqrt{B}$ and 
$\max\{a,b\}>Z_2^2|a-b|$, and to restrict attention
to $a,b$ for which   
$\min\{a,b\}\leq \max\{a,b\}(1-1/Z_2^2)$.

Finally, we break the
summation over $a,b$ into those for which $a>b$ and those for which
$a<b$, observing that
$$
   F(a,b)+F(b,a)=
\frac{\vol(\mathcal{S}_{b/a})+\vol(\mathcal{S}_{a/b})}{\max\{a,b\}\sqrt{|a^2-b^2|}}.
$$
This therefore allows us to restrict to a 
summation over the set $\A_2$, which thereby
completes the proof of the lemma.
\end{proof}

We now have everything in place to complete the proof of the theorem.
In what follows let us write $Y$ for either of the quantities $B2^{Z_2+1}/K$ or
$B2^{-(Z_2+1)}/K$. Define %T%
$$
\Sigma_{\theta_1,\theta_2}(Y):=
\sum_{\substack{\theta_1 a\leq b\leq (1-\theta_2)a
\\ a< \sqrt{B}, ~\cp{a}{b}}}
\frac{h(a,b;Y)}{a^{2}},
$$
for any $\theta_1, \theta_2>0$ such that $\theta_1+\theta_2<1$.
Observing that
$f(u)= f(0)+\int_0^{u}f'(t)\d t,$
it now follows that
\begin{equation}\label{eq:champ}
\sum_{(a,b)\in \A_2}\frac{f(b/a)h(a,b;Y)}{a^2}
=
f(0)\Sigma_{Z_2^{-2}, Z_2^{-2}}(Y)
+\int_0^{1-Z_2^{-2}} 
f'(t)
\Sigma_{\max\{t,Z_2^{-2}\}, Z_2^{-2}}(Y) \d t.
\end{equation}
Hence Lemma \ref{lem:farine} renders it   sufficient to estimate
$\Sigma_{\theta_1,\theta_2}(Y)$ asymptotically. 
This is achieved in the following result.

\begin{lemma}\label{lem:sig}%T%
Let
\begin{equation}\label{eq:W0}
W_0:=
\left\{
\mathbf{w}\in \RR_{\geq 0}^4:
\begin{array}{l}
w_1+w_2+w_4\leq 1+2w_3,\\
2w_3+w_4\leq 1+w_1+w_2,\\
3w_4\leq 1+w_1+w_2+2w_3,\\
w_1+w_2+2w_3\leq 1+w_4,\\
\max\{w_1,w_2,w_3\}\leq w_4,\\
w_4\leq 1-w_4
\end{array}
\right\}.
\end{equation}
Then we have
$$
\Sigma_{\theta_1,\theta_2}(Y)=
2\vol(W_0)C^*(1-\theta_1-\theta_2)
(\log B)^4+O\big((\log B)^{3+\varepsilon}\big),
$$
where $C^*$ is given by \eqref{eq:C*}, with
$(L_1,L_2,Q)=(x_1-x_2,x_1+x_2,x_1^2+x_2^2)$ and $g$ given by \eqref{eq:g}.
\end{lemma}

\begin{proof}
Recall the definition \eqref{eq:hi} of $h(a,b;Y)$ and define the
region
$$
W:=
\left\{
\mathbf{w}\in \RR_{\geq 0}^4:
\begin{array}{l}
w_1+w_2+w_4\leq 1+2w_3,\\
2w_3+w_4\leq 1+w_1+w_2,\\
3w_4\leq 1+w_1+w_2+2w_3,\\
w_1+w_2+2w_3\leq 1+w_4
\end{array}
\right\}.
$$
This is contained in $[0,1]^4$.
On recalling the definition \eqref{eq:VV} of
$V_{a,b}(R)$ for $(a,b)\in \A_2$, one easily checks that
\begin{align*}
h(a,b;Y)=
\sum_{\substack{n\mid a^4-b^4\\
d_i=\gcd(n,a+(-1)^ib)\\
     ~d_3=\gcd(n,a^2+b^2)\\
\left(
\frac{\log d_1}{\log Y},
\frac{\log d_2}{\log Y},
\frac{\log d_3}{2\log Y},
\frac{\log a}{\log Y}
\right)\in W
}}
(1*h)(n).
\end{align*}
Recalling \eqref{eq:V0}, it now follows from
Lemma \ref{lem:uniform} that
$$
\Sigma_{\theta_1, \theta_2}(Y)=
2\vol(W_0)C^*(1-\theta_1-\theta_2)
(\log B)^4+O\big((\log B)^{3+\varepsilon}\big),
$$
where $W_0$ is given by \eqref{eq:W0}. %T%
Here we have applied the lemma with the region
$$
\B=\{\x\in (0,1)^2: \theta_1 x_1\leq x_2\leq (1-\theta_2)x_1\},
$$
which has volume $(1-\theta_1-\theta_2)/2$.
%T%
%We can calculate the volume of $W_0$ using the software package
%\texttt{polymake} \cite{polymake}, the
%outcome being that $\vol(W_0)=1/72$.
%The statement of the lemma is now obvious.
\end{proof}

Recalling the definition of the function \eqref{eq:func}, it
is a routine calculation to verify that
$$
f(u)\ll 
1+ \log(1-u),
\quad
f'(u)\ll (1-u)^{-1}                             
$$
for $u \in (0,1)$, using Lemma \ref{lem:volume} and the
definition \eqref{eq:SBA} of $\mathcal{S}_u$.
In particular it follows that 
$\int_0^{1-1/Z_2^2}|f'(t)|\d t\ll \log Z_2\ll \log\log\log B$.  
Combining this  with Lemma 
\ref{lem:sig} and~\eqref{eq:champ} in 
Lemma~\ref{lem:farine},
and inserting it into Lemma \ref{lem:red-1}, we arrive at
the final bound
\begin{align*}
N_{U,H}(B)=&C B(\log B)^4 +
  O\Big(\frac{B(\log B)^4
}{\log\log B}\Big),
\end{align*}
where
$$
C=\frac{128 \vol(W_0)  C^*}{3\zeta(2)}\Big(f(0)+\int_0^1
(1-u)f'(u)\d u\Big)=
\frac{128\vol(W_0)C^*}{3\zeta(2)}\int_0^1 f(u)\d u.
$$
Finally, we note that
\begin{align*}
\int_0^1 f(u)\d u
&=
\int_0^1\int_{\mathcal{S}_{u}}
\frac{1}{\sqrt{1-u^2}}\d s\d t\d u+
\int_0^1\int_{\mathcal{S}_{1/u}} \frac{1}{\sqrt{1-u^2}}\d s\d t\d u\\
&=
\int_0^1\int_{\mathcal{S}_{u}}
\frac{1}{\sqrt{1-u^2}}\d s\d t\d u+
\int_1^\infty\int_{\mathcal{S}_{u}} \frac{1}{u\sqrt{u^2-1}}\d s\d t\d u\\
&=\sigma_\infty,
\end{align*}
where
\begin{equation}
   \label{eq:siginf}
\sigma_\infty
:=
\int\int\int_{\Big\{\substack{0<\max\{1,u\}
p_{u}(s,t)\leq 1,\\
0<\max\{1,u\}q_{u}(s,t)\leq 1,\\
t,u>0, ~r_{u}(s,t)>0}\Big\}} \frac{1}{\sqrt{|1-u^2|}}\d s\d t\d u.
\end{equation}
This therefore completes the proof of the
asymptotic formula
\begin{equation}
   \label{eq:goal}
  N_{U,H}(B)=
\frac{128\vol(W_0)C^*\sigma_\infty}{3\zeta(2)}
B(\log B)^4+O\Big( \frac{B(\log B)^4
}{\log\log B} \Big),  
\end{equation}
where $C^*$ is given by \eqref{eq:C*}.

\section{Peyre's constant}\label{s:local} 

The purpose of this section is to affirm that the constant in our
asymptotic formula~\eqref{eq:goal} for $N_{U,H}(B)$ agrees with the
Peyre's prediction \cite{p}, as required to complete the proof of the theorem.
In general terms the constant $c_{X,H}$ should be a product of three constants
$\alpha(X),\beta(X),\tau_H(X)$.

%T% [this bit is new]
The constant $\alpha(X)$ is a rational number defined in terms of the cone
of effective divisors of $X$. Following a suggestion of the referee we
will show that 
\begin{equation}
   \label{eq:alpha}
   \alpha(X)=2\vol(W_0),
\end{equation}
where $W_0$ is given by \eqref{eq:W0}. 
Although we will not make any use of it, one can show by direct 
calculation that 
$\alpha(X)=1/36.$
We begin by recalling some basic facts about the geometry of del Pezzo
surfaces $V\subset \Pfour$ of degree $4$, as described in  
the book by Manin \cite{manin-book}. 
Such $V$ are isomorphic to the blow-up of $\Ptwo$ along a union of
points $p_1,\ldots,
p_5$, no three of which are collinear. 
Let $E_i$ denote the exceptional divisor above $p_i$, for $1\leq i\leq
5$, let $L_{i,j}$ denote the strict transform of the line going
through $p_i$ and $p_j$, for $1\leq i<j\leq 5$, and let $Q$
denote the strict transform of the unique conic going through all $5$
points. These $16$ divisors constitute the famous $16$ lines 
on $V$.
If $\Lambda$ is the strict transform of a line in $\Ptwo$ not
passing through any of the $5$ distinguished  points, then 
a free basis for the geometric Picard group $\Pic_\cQ (V)$ is given by 
$
\{\Lambda, E_1,\ldots,E_5\},
$
on identifying divisors with their
classes in $\Pic_\cQ (V)$.
In fact $E_1,\ldots,E_5$ can be taken to be the classes of a set of
any $5$ lines in $V$ that are mutually skew.
In terms of this basis the anticanonical
divisor class can be written
$
-K_V=3\Lambda -\sum_{i=1}^5 E_i,
$
and furthermore, we have 
$
L_{i,j}=\Lambda-E_i-E_j$ and $Q=2\Lambda-\sum_{i=1}^5 E_i$.
There is an intersection
form $(\cdot,\cdot)$ on $\Pic_\cQ (V)$, which is non-degenerate,
symmetric  and bilinear.  In terms of this intersection form the
divisors $\Lambda, E_1,\ldots,E_5$ 
satisfy the relations
$$
(\Lambda,\Lambda)=1, \quad (\Lambda,E_i)=0, \quad 
(E_i, E_j)=
\begin{cases}
-1, &\mbox{if $i=j$,}\\
0, &\mbox{if $i\neq j$.}
\end{cases}
$$

Returning to $X$, we 
consider the affine subscheme $T\subset \AA_\ZZ^5=\Spec
\ZZ[x,y,z,u,v]$ 
defined by the equation
$$
(u^2-v^2)x^2+(u^2+v^2)y^2=2z^2,
$$
and the conditions $(x,y,z)\neq \mathbf{0}$ and $(u,v)\neq \mathbf{0}$.
The torus $\mathbb{G}_m^2$ acts on $T$ via the morphism of tori
$
(\lambda, \mu)\mapsto (\la,\la,\la\mu,\mu,\mu)
$
from $\mathbb{G}_m^2$ to $\mathbb{G}_m^5$ and the natural action of 
$\mathbb{G}_m^5$ to $\AA_\ZZ^5$. There is an obvious morphism $\pi$
from $T$ to $X$, given by 
$$
(x,y,z,u,v)\mapsto (ux,vy,uy,vx,z),
$$
which makes of $T$ a $\mathbb{G}_m^2$-torsor over
$X$.  In particular $X$ may be seen in two different
ways as a conic bundle surface, a fact that we have already made
critical use of in \S \ref{s:prelim}.
Let 
$\ve_1,\ve_2\in\{-1,+1\}$ and let $i=\sqrt{-1}$. Then a
straightforward calculation reveals that the $16$ lines on $X$ are
given by
\[
\arraycolsep 0pt
\begin{array}{rclrcl}
M_1(\ve_1,\ve_2)&:\quad
&\begin{cases}
u=\ve_1 v,\\
z=\ve_2 uy,
\end{cases}&
M_2(\ve_1,\ve_2)&:\quad
&\begin{cases}
u=\ve_1 i v,\\
z=\ve_2 ux,
\end{cases}\\
\noalign{\vskip0.75ex\penalty 1000}
M_3(\ve_1,\ve_2)&:\quad
&\begin{cases}
x=\ve_1 y,\\
z=\ve_2 ux,
\end{cases} &
M_4(\ve_1,\ve_2)&:\quad
&\begin{cases}
x=\ve_1 i y,\\
z=\ve_2 vy.
\end{cases}
\end{array}
\]
In particular all of the lines split over $\QQ(i)$
and $X$ contains precisely $8$ lines that are defined over $\QQ$.
Let us write $\mathcal{G}=\Gal(\QQ(i)/\QQ)\cong \ZZ/2\ZZ$
for the Galois group of the splitting field.
It is easily checked that 
$$
\begin{array}{lll}
E_1=M_1(1,1), & E_2=M_1(-1,1), & E_3=M_2(1,1),\\
E_4=M_2(-1,1), & E_5=M_3(1,-1),
\end{array}
$$
are  mutually skew. Since $E_1,E_2,E_5$ 
are defined over $\QQ$  and $E_3,E_4$ are
defined over $\QQ(i)$, but are conjugate under the action of
$\mathcal{G}$, so it follows that
$$
\Pic X=(\Pic_\cQ X)^{\mathcal{G}}=\{\Lambda, E_1,E_2,E_3+E_4,E_5\}.
$$
This retrieves the information $\Pic X\cong \ZZ^5$ that was recorded
in the introduction.

Using the intersection form it is now  routine
to show that we can choose the projection of $X$ to
$\Ptwo$ in such a way that we have the equalities 
$$
\begin{array}{lll}
L_{1,2}=M_3(-1,1), & L_{1,3}=M_4(1,1), &
L_{1,4}=M_4(-1,1),\\ 
L_{1,5}=M_1(1,-1), & L_{2,3}=M_4(1,-1), & L_{2,4}=M_4(-1,-1),\\ 
L_{2,5}=M_1(-1,-1), & L_{3,4}=M_3(-1,-1),&
L_{3,5}=M_2(1,-1),\\ L_{4,5}=M_2(-1,-1), & Q=M_3(1,1),&
\end{array}
$$
in the description of the lines. 
Let us identify 
$\Pic X \otimes_\ZZ \RR$ with its dual using the intersection
form.  Let $\Lambda_{\mathrm{eff}}(X)$ denote the convex cone 
in $\Pic X \otimes_\ZZ \RR$ that is generated by the classes of
effective divisors, and let $\Lambda_{\mathrm{eff}}^{\vee}(X)$
denote its dual cone, with respect to the intersection form.
As is well-known,
$\Lambda_{\mathrm{eff}}(X)$ is generated
by the classes $[O]=\sum_{L \in O} L$ associated to each orbit $O$ of the action of 
$\mathcal{G}$ on the $16$ lines. In our setting we have $12$ orbits overall,
given by 
$$
\begin{array}{ll}
O_1(\ve_1,\ve_2)=\{M_1(\ve_1,\ve_2)\}, &
O_2(\ve_1,\ve_2)=\{M_2(\ve_1,\ve_2), M_2(-\ve_1,-\ve_2)\}, \\
O_3(\ve_1,\ve_2)=\{M_3(\ve_1,\ve_2)\}, &
O_4(\ve_1,\ve_2)=\{M_4(\ve_1,\ve_2), M_4(-\ve_1,-\ve_2)\},
\end{array}
$$
for $\ve_1, \ve_2 \in \{-1,+1\}$.
%We may therefore conclude that $\Lambda_{\mathrm{eff}}(X)$
%is generated by the classes
%$$
%  \begin{array}{l}
%E_1, E_2, E_3, E_4+E_5,\\
%\mbox{$\Lambda- E_i-E_j$ for $(i,j)\in \{(1,2),(1,3),(2,3),(4,5)\}$},\\ 
%2\Lambda-E_1-E_2-E_3-E_4-E_5,\\
%\mbox{$2\Lambda- 2E_i-E_4-E_5$ for $i\in \{1,2,3\}$},
%\end{array}
%$$
%in $\Pic X$. 
Choosing coordinates $\alpha,\beta, \gamma, \delta, \phi$ for
$E_1,E_2,E_3+E_4,E_5$ and $\Lambda$, it therefore follows from 
\cite{p} that 
\begin{align*}
\alpha(X)
&=\vol
\left\{
(\alpha,\beta,\gamma,\delta,\phi)\in \RR^5:
\begin{array}{l}
\alpha,\beta,\gamma,\delta \geq 0,\\
\phi-\alpha-\delta, 
\phi-\beta-\delta\geq 0,\\
\phi-\alpha-\beta, \phi-\gamma \geq 0,\\ 
2\phi-\alpha-\beta-\gamma-\delta\geq 0, \\
2\phi-2\alpha-\gamma, 
2\phi-2\beta-\gamma, 
2\phi-2\delta-\gamma, \\
3\phi-\alpha-\beta-\gamma-\delta=1
\end{array}
\right\}.
\end{align*}
By our work in \S \ref{s:prelim} there is a symmetry between the two
conic fibrations which allowed us to reduce the analysis to points
with $\max\{a,b\}<\max\{x,y\}$. 
Now $\max\{a,b\}$ corresponds to a fibre of the first
fibration and is thus related to 
$M_1(1,1)+M_1(1,-1)$, whereas $\max\{x,y\}$ 
corresponds to a fibre of the second fibration and is related to 
$M_3(1,1)+M_3(1,-1)$. This symmetry extends to the polyhedron
arising in $\alpha(X)$ and therefore permits us to assume that 
$$
\phi-\delta\leq 2\phi-\alpha-\beta-\gamma,
$$
on multiplying the volume through by $2$. With this additional
constraint one easily checks that the inequalities
$2\phi-2\alpha-\gamma\geq 0$ and $2\phi-2\beta-\gamma\geq 0$ are rendered
superfluous.

There is a direct correspondence between the
parameters $w_1,w_2,w_3,w_4$ 
introduced in \S \ref{s:3} and the exceptional curves on $X$. 
Thus $w_1$, corresponding to $\gcd(n,a-b)$, is related to
$M_1(1,1)$. Likewise $w_2$ corresponds to $\gcd(n,a+b)$ and is related to
$M_1(-1,1)$, and $2w_3$ corresponds to $\gcd(n,a^2+b^2)$ and is 
related to $M_2(1,1)+M_2(-1,1)$.
Finally $w_4$ comes from $a$ and is related to 
$M_1(1,1)+M_1(1,-1)$.  
Thus we may write 
$w_1=\alpha$, $w_2=\beta$, $2w_3=\gamma$ and
$w_4=\phi-\delta$. Eliminating $\phi$ using the equation
$
2\phi=1+w_1+w_2+2w_3-w_4,
$
we are therefore led to the expression \eqref{eq:alpha}.

The constant $\beta(X)$ is equal to the cardinality of
$H^1(\QQ, \Pic_{\cQ} X)$, where $\Pic_{\cQ} X$ is the
geometric Picard group. In the present setting we have
\begin{equation}
   \label{eq:beta}
   \beta(X)=1,
\end{equation}
since $X$ is $\QQ$-birationally trivial.

We now turn to the value of the constant $\tau_H(X)$.
For any place $v$ of $\QQ$ let $\omega_{H,v}$ be 
the usual $v$-adic density of points on the
locally compact space $X(\QQ_v)$.  The Tamagawa number associated to
$X$ and $H$ is then given by
$$
\tau_H(X)
= \lim_{s\rightarrow 1}\big((s-1)^5
L(s,\Pic_{\cQ} X)\big) 
\omega_{H,\infty}\prod_{p}\frac{\omega_{H,p}}{L_p(1,\Pic_{\cQ} 
X)}.
$$
In the present setting it follows from our calculation of $\Pic_{\cQ}
X$ that
$$
L(s,\Pic_{\cQ} X)=\zeta_\QQ(s)^4\zeta_{\QQ(i)}(s)=\zeta_\QQ(s)^5 L(s,\chi),
$$
where $\chi$ is the real non-principal character modulo $4$, whence
\begin{align*}
\tau_H(X)
&= \frac{\pi}{4} 
\omega_{H,\infty}\prod_{p}\Big(1-\frac{1}{p}\Big)^{5}\Big(1-\frac{\chi(p)}{p}\Big)\omega_{H,p}
= \omega_{H,\infty}\prod_{p}\Big(1-\frac{1}{p}\Big)^{5}\omega_{H,p}.
\end{align*}
Recall the definitions \eqref{eq:g}, \eqref{eq:rho_dag-} of $g$ and
$\overline{\rho}_p^\dagger(\nu_1,\nu_2,\nu_3)$, respectively. We will
show that
\begin{equation}
   \label{eq:padic}
   \omega_{H,p}=
\kappa_p\Big(1-\frac{1}{p}\Big)^{-1}\Big(1+\frac{1}{p}\Big)\sum_{\bnu 
\in \ZZ_{\geq 0}^3}
g(p^{\nu_1+\nu_2+\nu_3})\overline{\rho}_p^\dagger(\nu_1,\nu_2,\nu_3),
\end{equation}
with
\begin{equation}
   \label{eq:kappa}
\kappa_p:=\begin{cases}
1, &\mbox{if $p>2$, }\\
4/3, &\mbox{if $p=2$. }
\end{cases}
\end{equation}
Furthermore, we will demonstrate the equality
\begin{equation}
   \label{eq:inf}
\omega_{H,\infty}=16 \sigma_\infty,
\end{equation}
where $\sigma_\infty$ is given by \eqref{eq:siginf}.

Subject to the proofs of \eqref{eq:padic} and
\eqref{eq:inf}, let us now verify that the 
constant in \eqref{eq:goal} coincides with that
predicted by Peyre. Bringing together our expression for~$\tau_H(X)$
with \eqref{eq:alpha} and \eqref{eq:beta},  we find that 
\begin{align*}%T%
c_{X,H}
=\alpha(X)\beta(X)\tau_H(X)=\frac{128\vol(W_0)C^* \sigma_\infty }{3\zeta(2)},
\end{align*}
in the notation of \eqref{eq:C*}.  This therefore confirms the
constant in \eqref{eq:goal} is the same one that is predicted by Peyre.

\subsection{Calculation of $\omega_{H,\infty}$}

In order to prove \eqref{eq:inf}
we will adhere to the method outlined by Peyre \cite{p}.
Taking into account the fact that $\x$ and $-\x$ represent the
same point in $\Pfour$, the archimedean density of points on $X$
is equal to
\begin{align*}
\omega_{H,\infty}
&= \frac{1}{2}
\int_{ \{\x\in\RR^5: \,\Phi_1(\x)=\Phi_2(\x)=0,\, 
\|\x\|\leq 1 \}}\omega_L( \x)\\
&= 8
\int_{ \{\x\in\RR_{>0}^5: \,\Phi_1(\x)=\Phi_2(\x)=0,\,
   \max\{x_0,x_1,x_2,x_3\}\leq 1\}}\omega_L( \x),
\end{align*}
where $\omega_L( \x)$ is the Leray form, and we have used the same sort
of symmetry arguments apparent in Lemma \ref{lem:red-1}.
It will be convenient to parametrise
the points via the choice of variables $x_0,x_1,x_2$.
Observe that
$$
\det \left(
\begin{matrix}
\frac{\partial \Phi_1}{ \partial x_3}&\frac{\partial \Phi_2}{
\partial x_3} \cr\cr  \frac{\partial \Phi_1}{ \partial x_4}
&\frac{\partial \Phi_2}{ \partial x_4}
\end{matrix}
\right)=4x_2x_4.
$$
We will take the Leray form $\omega_L(\x)=(4x_2x_4)^{-1}\d x_0\d x_1\d
x_2$. Recall the definition~\eqref{eq:def_pqr} of $p_u,q_u,r_u$. We
will pass from $x_0,x_1,x_2$ to $(s,t,u)$, with 
$t,u>0$,  via the change of variables
$$
(x_0,x_1,x_2)=(p_{u}(s,\sqrt{|1-u^2|}t), uq_{u}(s,\sqrt{|1-u^2|}t),
q_{u}(s,\sqrt{|1-u^2|}t)).
$$
The Jacobian  of this transformation is
%$$
%\det \left(
%\begin{matrix}
%-4s+4t &
%-2 (1-u^2)t+4s & 2ut^2 \cr\cr
%4us & -2u(1-u^2)t & 2s^2+(3u^2-1)t^2 \cr\cr
%4s & -2(1-u^2) t & 2ut^2
%\end{matrix}
%\right),
%$$
calculated to be
$$
8q_u(s,\sqrt{|1-u^2|}t)r_u(s,\sqrt{|1-u^2|}t)=8x_2x_4,
$$
on noting that
$$
(1-u^2)p_u(s,\sqrt{|1-u^2|}t)^2+(1+u^2)q_u(\sqrt{|1-u^2|}t)^2=2r_u(s,\sqrt{|1-u^2|}t)^2.
$$
A modest pause for thought now reveals that \eqref{eq:inf} holds, as claimed.

\subsection{Calculation of $\omega_{H,p}$}

It will ease notation if we write  $\ZZ/p^k$
for $\ZZ/p^k\ZZ$ throughout this section, for any prime power
$p^k$.  It remains to calculate the value of
$\omega_{H,p}=\lim_{n\to \infty}p^{-3n}N(p^n)$, for any prime $p$,
where $N(p^n)$ denotes the number of $\x\in 
(\ZZ/p^n)^5$ for which $\Phi_1(\x)\equiv 
\Phi_2(\x)\equiv
0 \bmod{p^n}$. The usual way of proceeding at this point would be to try and
interpret $p^{-3n}N(p^n)$ as an explicit function of
$p^{-1}$. This would then entail a parallel calculation of $C^*$ in
\eqref{eq:C*}, in order to check that the values match.
Instead we will try and calculate $N(p^n)$ by mimicking the
steps taken in the proof of the theorem.
Recall the definitions \eqref{eq:bundle}, \eqref{eq:phi2} of $\Phi_1$
and~$\Phi_2$.

It is easy to check that
$\omega_{H,p}=(1-1/p)^{-1}\omega_{H,p}^*$, where 
$\omega_{H,p}^*=\lim_{n\to 
\infty}p^{-3n}N^*(p^n)$, for any prime $p$,
with
$$
N^*(p^n):=\#\{\x\in (\ZZ/p^n)^5 : p\nmid \x, ~\Phi_1(\x)\equiv \Phi_2(\x)\equiv
0 \pmod{p^n}\}.
$$
We will show that
\begin{equation}
   \label{eq:padic*}
   \omega_{H,p}^*=
\kappa_p\Big(1+\frac{1}{p}\Big)\sum_{\bnu \in \ZZ_{\geq 0}^3}
g(p^{\nu_1+\nu_2+\nu_3})\overline{\rho}_p^\dagger(\nu_1,\nu_2,\nu_3),
\end{equation}
where $\kappa_p$ is given by \eqref{eq:kappa}.
This will clearly be enough to establish \eqref{eq:padic}.

Let $\x$ be any vector counted by $N^*(p^n)$.  We claim that there are
there are precisely~$\phi(p^n)$ choices of
$(a,b,x,y,z)\in(\ZZ/p^n)^5$ such that
$p\nmid(a,b)$, $p\nmid
(x,y)$ and
$$
\x\equiv (ax,by,ay,bx,z) \pmod{p^n}.
$$
To see this we may suppose without loss of generality that $p\nmid
x_0$. For any of $\phi(p^n)$ choices of $a\in(\ZZ/p^n)^*$,
the value of $x$ is determined uniquely modulo $p^n$ via the congruence
$x_0\equiv ax \bmod{p^n}$.  But then values of $b$ and $y$ are also
determined exactly modulo $p^n$, which
therefore establishes the claim.

Recall the definition
of $C_{a,b}$ from \eqref{eq:conic}. By a 
convenient abuse of notation we will write
$(x,y,z)\in C_{a,b}(\ZZ/p^n)$ to denote that the underlying
quadratic polynomial is congruent to zero modulo $p^n$.
It now follows that
\begin{align*}
N^*(p^n)&=\frac{1}{\phi(p^n)}
\sum_{\substack{a,b\bmod{p^n}\\
     p\nmid(a,b)}}\#\big\{
(x,y,z)\in C_{a,b}(\ZZ/p^n): p\nmid (x,y)\big\}\\
&=\frac{1}{\phi(p^n)}
\sum_{\nu_1,\nu_2,\nu_3\geq 0}
\sum_{\substack{a,b\bmod{p^n}\\
     p\nmid(a,b)\\ p^{\nu_i}\| a+(-1)^i b\\ p^{\nu_3}\| a^2+b^2}}
\#\big\{(x,y,z)\in C_{a,b}(\ZZ/p^n) : p\nmid (x,y)
\big\},
\end{align*}
where the index $i$ belongs to $\{1,2\}$.
Moreover, in view of the fact that \eqref{eq:xi} is always a solution
of the relevant congruence modulo $p^n$ we may assume in all of our
calculations that the summand is always positive.
Here we recall the convention that $p^\nu\| n$ if and only if
$\nu_p(n)=\nu$.

For given $c,d\in\ZZ$, with $p\nmid cd$, and given $\mu,\nu\geq 0$,
let us define
\begin{equation}
   \label{eq:C^*}
\begin{split}
D_{\mu,\nu}^*(p^n)
&=D_{\mu,\nu}^*(p^n;c,d)\\
&:=p^{-2n}\#
\Big\{(x,y,z)\in (\ZZ/p^n)^3:
\begin{array}{l}
p\nmid (x,y),\\
cp^\mu x^2+dp^\nu y^2\equiv
2z^2 \bmod{p^n}
\end{array}
\Big\}.
\end{split}
\end{equation}
It will be convenient to define
$D_{\mu,\nu}(p^n)$ as for $D_{\mu,\nu}^*(p^n)$, but without the
condition that $p\nmid(x,y)$.
A little thought reveals that we have
\begin{equation}
   \label{eq:little}
D_{\mu,\nu}^*(p^n)
=D_{\mu,\nu}(p^n)-p^{-1}D_{\mu,\nu}(p^{n-2})
\end{equation}
for each $n\geq 2$ and $\mu,\nu\geq 0$.
Our calculation of $D_{\mu,\nu}^*(p^n)$ will 
depend intimately on the values of
$\mu,\nu,c,d$ and whether or not $p$ is odd.
We have collected together the necessary information in the following result.

\begin{lemma}\label{lem:final}
Let $p>2$. Then  we have
\begin{equation}
   \label{eq:00>2}
   D_{0,0}^*(p^n)=1-\frac{1}{p^2},
\end{equation}
and furthermore,
\begin{equation}
   \label{eq:mu0>2}
D_{\mu,0}^*(p^n)=\Big(1-\frac{1}{p}\Big)\Big(\mu\Big(1-\frac{1}{p}\Big)+1+\frac{1}{p}\Big)+O(p^{-n/2+\mu/4+1/2}).
\end{equation}
if $\mu\geq 1$ and $(\frac{2d}{p})=1.$
Let $p=2$. Then we have
\begin{equation}
   \label{eq:2}
   D_{\mu,\nu}^*(2^n)=\begin{cases}
1, & \mbox{if $\mu=\nu=0$ and $c+d\equiv 0,2 \bmod{8}$},\\
\mu, & \mbox{if $\mu\geq 3$, $\nu=1$ and $2^{\mu-1}c+d\equiv 1 \bmod{8}$.}
\end{cases}
\end{equation}
\end{lemma}

We postpone the proof of this result until later.
For the moment we return to the expression for $N^*(p^n)$ that we are
trying to evaluate.
For any prime $p$ and $\nu_1,\nu_2,\nu_3\geq 0$ recall the definitions
\eqref{eq:rho_dag} and \eqref{eq:rho_dag-} of
$\rho_p^\dagger(\nu_1,\nu_2,\nu_3)$ and
$\overline{\rho}_p^\dagger(\nu_1,\nu_2,\nu_3)$, respectively.
Using \cite[Lemme~2]{L1L2Q} one confirms that 
$
\rho_p^\dagger(\nu_1,\nu_2,\nu_3)=O(p^{\nu_1+\nu_2+\nu_3+2}).
$
We have
\begin{align*}
N^*(p^n)
&=\frac{p^{4n}}{\phi(p^n)}
\sum_{\nu_1,\nu_2,\nu_3\geq 0}\frac{1}{p^{2\nu}}
\sum_{\substack{a,b\bmod{p^{\nu}}\\
     p\nmid(a,b)\\ p^{\nu_i}\| a+(-1)^i b\\ p^{\nu_3}\| a^2+b^2}}
\frac{\#\big\{(x,y,z)\in C_{a,b}(\ZZ/p^n) : p\nmid (x,y)
\big\}}{p^{2n}},
\end{align*}
where we have written $\nu=\nu_1+\nu_2+\nu_3+1$ for convenience
and the inner cardinality  only depends on congruences modulo $p^\nu$.

Beginning with the case $p>2$ let us write
$$
a^2-b^2=p^{\nu_1+\nu_2}c, \quad
a^2+b^2=p^{\nu_3}d,
$$
for $c,d$ coprime to $p$. Since $p\nmid
(a,b)$ it clearly follows that at most one of $\nu_1,\nu_2,\nu_3$ can
be non-zero. When they are all zero \eqref{eq:00>2} reveals
that the summand here is~$(1-1/p^2)$. When $\nu_1\geq 1$ and $\nu_2=\nu_3=0$ we
know that $2d$ is a quadratic residue modulo $p$ since $C_{a,b}(\ZZ/p^n)$ is
always non-empty.  But then  we may apply \eqref{eq:mu0>2} to estimate the
summand. This same estimate also covers the case in which $\nu_2$ or
$\nu_3$ is positive.
When $p>2$ we therefore deduce that
$$
\lim_{n\rightarrow \infty}\frac{N^*(p^n)}{p^{3n}}
=
\sum_{\nu_1,\nu_2,\nu_3\geq 0}
\Big((\nu-1)\Big(1-\frac{1}{p}\Big)+1+\frac{1}{p}\Big)
\overline{\rho}_p^\dagger(\nu_1,\nu_2,\nu_3)
$$
with $\nu=\nu_1+\nu_2+\nu_3+1$. Recall the definition \eqref{eq:g} of
$g$ for odd prime powers. It is not hard to see that
$$
\Big(1+\frac{1}{p}\Big)g(p^{\nu-1})=
(\nu-1)\Big(1-\frac{1}{p}\Big)+1+\frac{1}{p}.
$$
It therefore follows that \eqref{eq:padic*} holds 
when $p>2$, with $\kappa_p=1$.

It remains to deal with the case $p=2$.  Arguing as for $p>2$ we write
$$
a^2-b^2=2^{\nu_1+\nu_2}c, \quad
a^2+b^2=2^{\nu_3}d,
$$
for odd $c,d$.   Since $2\nmid
(a,b)$ it now follows that either $\nu_1=\nu_2=\nu_3=0$, or
 else~$\nu_1+\nu_2\geq 3$ and $\nu_3=1$.  In the latter case we have
$\nu_1\geq 2$ and $\nu_2=1$, or
$\nu_2\geq 2$ and $\nu_1=1$.
In every case we know that
$C_{a,b}(\ZZ/2^n)$ is non-empty, and in fact has a solution
$(x,y,z)\equiv (1,1,a)\bmod{2^n}$.
Suppose first that $\nu_1=\nu_2=\nu_3=0$. Then
precisely one of $a,b$ must be even, and we automatically have
$c+d\equiv 0,2 \bmod{8}$.  But then \eqref{eq:2} implies that
the innermost summand in our expression for $N^*(2^n)$ is
$1$.  When $\nu_1\geq 2$ and $\nu_2=\nu_3=1$ we may apply \eqref{eq:2}
with $\mu=\nu_1+1$ to deduce that the summand is
$\nu_1+1=\nu_1+\nu_2+\nu_3-1.$  The same
applies  when $\nu_2\geq 2$ and~$\nu_1=\nu_3=1$.
Putting this all together we deduce that
\begin{align*}
\lim_{n\rightarrow \infty} \frac{N^*(2^n)}{2^{3n}}
&=
2\Big(\overline{\rho}_2^\dagger(0,0,0)+
\sum_{\substack{\nu_1,\nu_2,\nu_3\geq 0\\ \nu_1+\nu_2+\nu_3\geq 1}}
(\nu_1+\nu_2+\nu_3-1)\overline{\rho}_2^\dagger(\nu_1,\nu_2,\nu_3)
\Big)\\
&=2 \sum_{\nu_1,\nu_2,\nu_3\geq 0}
\max\{1,\nu_1+\nu_2+\nu_3-1\}\overline{\rho}_2^\dagger(\nu_1,\nu_2,\nu_3).
\end{align*}
Recall the definition \eqref{eq:g} of
$g$ for even prime powers. It is now clear that \eqref{eq:padic*} holds
when $p=2$, with $\kappa_2=4/3$.

\begin{proof}[Proof of Lemma \ref{lem:final}]
Throughout our proof of the lemma we can and will assume that $n$ is large
compared with $\mu$ and $\nu$.

Assume that $p>2$ and let $\mu=\nu=0$. It follows from
Hensel's lemma that
$$
D_{0,0}^*(p^n)=D_{0,0}^*(p)=p^{-2}M_{0,0}^*(p),
$$
where $M_{0,0}^*(p)$
is the number of
$(x,y,z)\in (\ZZ/p)^3$ such that $p\nmid(x,y)$ and
$c x^2+d y^2\equiv 2z^2 \bmod{p}$. But this is just the number of 
$\FF_p$-points on $\PP_{\FF_p}^1$, since a conic with a
$\FF_p$-point is isomorphic to $\PP_{\FF_p}^1$.
Hence $M_{0,0}^*(p)=p^2-1$ and \eqref{eq:00>2} is rendered 
obvious.

We now turn to the calculation of $D_{\mu,0}^*(p^n)$, for $\mu\geq 1$,
for which we will assume that
$
   (\frac{2d}{p})=1,
$
so that $2d$ can be lifted to a quadratic residue modulo $p^n$.
Hence
$$
D_{\mu,0}^*(p^n)=p^{-2n}\#
\big\{(x,y,z)\in (\ZZ/p^n)^3:
p\nmid (x,y),~ 2cp^\mu x^2+ y^2\equiv
z^2 \bmod{p^n}
\big\}.
$$
Similarly for  $D_{\mu,0}(p^n)$. The contribution 
to the right hand side from those
$x,y,z$ for which $p\nmid y$ is
$
2(1-1/p).
$
This follows from an argument based on Hensel's lemma and an easy
analysis of the corresponding congruence modulo $p$.
The contribution from those $x,y,z$ for which
$p\mid y$ is $0$ if $\mu=1$ and
\begin{align*}
&p^{-2n}
\#
\big\{(x,y,z)\in (\ZZ/p^n)^*\times (\ZZ/p^{n-1})^2:
2cp^\mu x^2+ p^2y^2\equiv
p^2z^2 \bmod{p^n}
\big\},
\end{align*}
if $\mu\geq 2$. Assuming $\mu\geq 2$ we therefore obtain the
contribution
\begin{align*}
p^{4-2n}
&\#
\big\{(x,y,z)\in (\ZZ/p^{n-2})^3: 2cp^{\mu-2} x^2+ y^2\equiv
z^2 \bmod{p^{n-2}}
\big\}\\
&-
p^{3-2n}
\#
\big\{(x,y,z)\in (\ZZ/p^{n-2})^3: 2cp^{\mu} x^2+ y^2\equiv
z^2 \bmod{p^{n-2}}
\big\}.
\end{align*}
But this is just $D_{\mu-2,0}(p^{n-2})-p^{-1}D_{\mu,0}(p^{n-2})$.
It therefore follows that
$$
D_{\mu,0}^*(p^n)=
\begin{cases}
2(1-1/p), & \mbox{if $\mu=1$,}\\
2(1-1/p)+D_{\mu-2,0}(p^{n-2})-p^{-1}D_{\mu,0}(p^{n-2}), & \mbox{if
   $\mu\geq 2$.}
\end{cases}
$$
Recalling \eqref{eq:little}, we have
$$
D_{1,0}(p^n)=D_{1,0}^*(p^n)+p^{-1}D_{1,0}(p^{n-2})=
2(1-1/p)+p^{-1}D_{1,0}(p^{n-2}),
$$
and
$$
D_{\mu,0}(p^n)=2(1-1/p)+ D_{\mu-2,0}(p^{n-2}),
$$
if $\mu\geq 2$.
In view of the fact that $D_{1,0}(p)=2$ and $D_{1,0}(1)=1$, we 
see that 
$D_{1,0}(p^{2n+1})=2$ and 
$D_{1,0}(p^{2n})=2-p^{-n}$. Hence 
$$
D_{1,0}(p^n)=2+O(p^{-n/2}).
$$
Moreover, on noting that
$D_{0,0}(p)=2-1/p^2$ and $D_{0,0}(1)=1$,
\eqref{eq:little} and \eqref{eq:00>2} give
\begin{align*}
D_{0,0}(p^{2n+1})
&=1+1/p+\frac{(2-2/p-1/p^2)}{p^n} = 1+1/p+O(p^{-(2n+1)/2+1/2})\\
D_{0,0}(p^{2n})
&=1+1/p-1/p^{n+1} = 1+1/p+O(p^{-(2n)/2-1}).
\end{align*}
Armed with these expression we can conclude that
$$
D_{\mu,0}(p^{n})
=\mu(1-1/p)+ 1+1/p+O(p^{-n/2+\mu/4+1/2}),
$$
from which \eqref{eq:mu0>2} is an easy consequence.

We now examine the case $p=2$.
Recall the definition \eqref{eq:C^*} of
$D_{\mu,\nu}^*(2^n)$ and the corresponding definition of
$D_{\mu,\nu}(2^n)$.  Let $\mu=\nu=0$, under which hypothesis we will
assume that
\begin{equation}
   \label{eq:assume;1}
c+d\equiv 0,2 \pmod{8}.
\end{equation}
Our argument
will differ according to which of $0$ or $2$ it is that $c+d$ is
congruent to modulo $8$. Beginning with the case 
$c+d\equiv 2 \bmod{8}$, it is not hard to see that
$x,y,z$ are all necessarily
odd in the definition of $D_{0,0}^*(2^n)$. Writing $x=1+2x', y=1+2y'$ and
$z=1+2z'$ we deduce that
$$
D_{0,0}^*(2^n)=2^{-2n}\#
\big\{(x',y',z')\in (\ZZ/2^{n-1})^3:  f(x',y')
\equiv {z'}+{z'}^2 \bmod{2^{n-2}}
\big\},
$$
where we have set $f(x',y')=2e +cx'(1+x')+dy'(1+y')$ and
$2e=(c+d-2)/4$.  For any $a\in \ZZ$ and any 
$m\geq 1$, let us denote by $S_a(2^m)$
the number of $x\in \ZZ/2^m$  for which $x(x+1)\equiv a
\bmod{2^m}$.  An easy lifting argument reveals that
$$
S_a(2^m)=S_a(2)=
\begin{cases}
0, &\mbox{if $2\nmid a$,}\\
2, &\mbox{if $2\mid a$}.
\end{cases}
$$
Varying $x',y'$ and applying the latter equality to estimate the
number of relevant $z'$, we may now conclude that
$D_{0,0}^*(2^n)
=2^{-2n}\cdot 2^{n-1}\cdot 2^{n-1}\cdot 4= 1$ when $c+d\equiv
2\bmod{8}$. Turning to the case
$c+d\equiv 0\bmod{8}$, we conclude that
$x,y$ are both necessarily
odd in the definition of $D_{0,0}^*(2^n)$ and $z$ 
is even. Writing $x=1+2x', y=1+2y'$ and~$z=2z'$ we deduce that
$$
D_{0,0}^*(2^n)=2^{-2n}\#
\big\{(x',y',z')\in (\ZZ/2^{n-1})^3:
f(x',y')\equiv 2{z'}^2 \bmod{2^{n-2}}
\big\},
$$
where $f(x',y')$ is as above but this time with $2e=(c+d)/4$.
We now conclude as previously. Altogether we have shown that
\eqref{eq:2} holds with $\mu=\nu=0$, under the 
assumption that \eqref{eq:assume;1} holds.

We proceed to consider $D_{\mu,1}^*(2^n)$, for $\mu\geq 3$, observing that
\begin{align*}
D_{\mu,1}^*(2^n)
&=2^{-2n}\#
\big\{(x,y,z)\in (\ZZ/2^n)^3:
2\nmid (x,y), ~2^\mu cx^2+ 2dy^2\equiv
2z^2 \bmod{2^n}
\big\}\\
&=2^{3-2n}\#
\Big\{(x,y,z)\in (\ZZ/2^{n-1})^3:
\begin{array}{l}
2\nmid (x,y),\\
2^{\mu-1} cx^2+ dy^2\equiv
z^2 \bmod{2^{n-1}}
\end{array}
\Big\}.
\end{align*}
We will make the assumption that
\begin{equation}
   \label{eq:assume;2}
2^{\mu-1}c+d\equiv 1 \pmod{8}.
\end{equation}
For any odd $a\in \ZZ$ and any $m\geq 3$, let $T_a(2^m)$
denote the number of $x\in \ZZ/2^m$  for which $x^2 \equiv a
\bmod{2^m}$.  Using the structure of the group $(\ZZ/2^m)^*$ 
it follows that 
$$
T_{a}(2^m)=
\begin{cases}
0, &\mbox{if $2\nmid a$ and $a\not \equiv 1 \bmod{8}$,}\\
4, &\mbox{if $a \equiv 1 \bmod{8}$},
\end{cases}
$$
for $m\geq 3$. 

Let us consider the contribution from those
$x,y,z$ for which $2\nmid y$, appealing to \eqref{eq:assume;2} where
necessary.  For each
$x,y\in \ZZ/2^{n-1}$ such that $2\nmid y$ and
$2^{\mu-1} cx^2+ dy^2\equiv 1 \bmod{8}$, our calculation of
$T_a(2^{n-1})$ shows that there are exactly $4$ choices
for $z$.
Altogether, we therefore obtain the contribution
$2^{3-2n}\cdot 2^{n-2}\cdot 2^{n-2}\cdot 4=2$
 when $\mu=3$, and
$2^{3-2n}\cdot 2^{n-1}\cdot 2^{n-2}\cdot 4=4$ when $\mu\geq 4$.
The contribution from those $x,y,z$ for which
$2\mid y$ is clearly
\begin{align*}
2^{7-2n}\#
\big\{(x,y,z)\in (\ZZ/2^{n-3})^3: 
2\nmid x,~
2^{\mu-3} cx^2+ dy^2\equiv
z^2 \bmod{2^{n-3}}
\big\}=N_{\mu},
\end{align*}
say. We will show that
\begin{equation}
   \label{eq:duck}
   N_{\mu}=\begin{cases}
1, &\mbox{if $\mu=3$,}\\
\mu-4, &\mbox{if $\mu\geq 4$.}
\end{cases}
\end{equation}
Together with the contribution from the previous treatment, this will
suffice to show that \eqref{eq:2} holds when $\mu\geq 3$ and $\nu=1$,
under the hypothesis  \eqref{eq:assume;2}.

Turning to the proof of \eqref{eq:duck}, suppose first that
$\mu=3$. Then $d\equiv 5\bmod{8}$, by~\eqref{eq:assume;2}, and
$$
N_3=
2^{7-2n}\#
\big\{(x,y,z)\in (\ZZ/2^{n-3})^3:
2\nmid x,~
cx^2+ dy^2\equiv
z^2 \bmod{2^{n-3}}
\big\}.
$$
When $2\mid y$ we note that for any
$x,y\in \ZZ/2^{n-3}$ such that $2\nmid x$ and $cx^2+ dy^2\equiv 1
\bmod{8}$, there are $4$ choices for $z$.
The condition modulo $8$ is clearly equivalent to~$c+ 5y^2\equiv 1 \bmod{8}$. In
particular, assuming that $c\equiv 1\bmod{4}$, we get an overall contribution of
$2^{7-2n}\cdot 2^{n-4}
\cdot 2^{n-5}\cdot 4=1$. Alternatively, the case $2\nmid y$
contributes the same amount via a similar argument, but this time
under the hypothesis $c\equiv 3 \bmod{4}$. 
Together this confirms that $N_3$ is indeed $1$.

The case $\mu=4$ is easily seen to be impossible, 
whence $N_4=0$. Turning to the case
$\mu=5$, for which \eqref{eq:assume;2} implies that
$d\equiv 1 \bmod{8}$, we deduce that both $y$ and~$z$ must be even. Hence it follows
that
$$
N_5=
2^{11-2n}\#
\big\{(x,y,z)\in (\ZZ/2^{n-5})^3:
2\nmid x, ~cx^2+ y^2\equiv z^2 \bmod{2^{n-5}}
\big\}.
$$
The argument now goes as for the case $\mu=3$, with a separate
analysis of the cases~$2\mid y$ and $2\nmid y$. Each contributes $1$
to  $N_5$, but only one such case arises according to
whether $c$ is congruent to $1$ or $3$ modulo $4$.
Hence $N_5=1$.

Finally let us turn to the case $\mu\geq 6$, for which we still
have $d\equiv 1 \bmod{8}$ by~\eqref{eq:assume;2}.  It will now be convenient
to write $N_\mu=N_\mu(n)$ to indicate the dependence on $n$ in the
definition.  Now either $y,z$ are both odd or they are both
even. Since $2^{\mu-3}cx^2+dy^2\equiv 1 \bmod{8}$ whenever $\mu\geq 6$
and $y$ is odd, so it follows that the contribution to $N_{\mu}(n)$
from odd $y$ is $2^{7-2n}\cdot 2^{2(n-4)}\cdot 4=2$, by our expression
for~$T_a(2^m)$. The contribution
from even $y$ is easily seen to be $N_{\mu-2}(n-2)$.
But then induction on $\mu$ reveals that the contribution from even
$y$ is $\mu-6$. Together the two contributions suffice to ensure that
\eqref{eq:duck} holds when $\mu\geq 6$.
\end{proof}

\end{document}